\documentclass[12pt,reqno]{amsart}
\usepackage{amsmath,amsfonts,amsthm,bbm,amssymb,color}
\usepackage[usenames,dvipsnames]{xcolor}
\usepackage[T1]{fontenc}
\usepackage{enumerate}
\usepackage{hyperref}
\hypersetup{
     colorlinks   = true,
     citecolor    = blue,
     linkcolor    = blue
}

\usepackage{pdfsync}{\tiny }

\usepackage[left=1in, right=1in, top=1.1in,bottom=1.1in]{geometry}
\numberwithin{equation}{section}
\newcommand{\Fc}{\mathcal{F}}

\newcommand{\Gc}{\mathcal{G}}
\newcommand{\Sc}{\mathcal{S}}
\newcommand{\Rc}{\mathcal{R}}

\newcommand{\Hc}{\mathcal{H}}

\newcommand{\N}{\mathbb{N}}

\newcommand{\E}{\mathbb{E}}
\newcommand{\R}{\mathbb{R}}
\newcommand{\D}{\mathbb{D}}
\newcommand{\Pb}{\mathbb{P}}
\newcommand{\Hg}{\mathfrak{H}}

\newcommand{\Norm}[1]{\left\lVert#1\right\rVert}
\newcommand{\Abs}[1]{\left|#1\right|}
\newcommand{\Ip}[1]{\left\langle #1 \right\rangle}
\newcommand{\Indi}[1]{\mathbbm{1}_{#1}}
\newcommand{\Multicase}[4]{\left\{\begin{array}{ll}#1 & \mbox{if } #2 \\#3 & \mbox{if } #4\end{array}\right.}
\newcommand{\Sqmatrix}[4]{\left(\begin{array}{cc}#1&#2\\#3&#4\end{array}\right)}
\newcommand{\Comb}[1]{\left(\begin{array}{c}#1\end{array}\right)}

\newcounter{dummy} \numberwithin{dummy}{section}

\newtheorem{Teorema}[dummy]{Theorem}

\newtheorem{Lema}[dummy]{Lemma}

\def\1{{\rm l}\hskip -0.21truecm 1}

\begin{document}
\title[Functional CLT for the self-intersection local time of the fBm]{Functional  limit theorem for the self-intersection local time of the fractional Brownian motion}
 \date{\today}

\author{Arturo Jaramillo \and David Nualart}

\address{David Nualart and Arturo Jaramillo: Department of Mathematics, University of Kansas,   Lawrence, KS 66045, USA.}
\email{nualart@ku.edu}
\thanks{D. Nualart was supported by the NSF grant  DMS1512891.}


\date{\today}


\begin{abstract}
Let $\{B_{t}\}_{t\geq0}$ be a $d$-dimensional fractional Brownian motion with Hurst parameter $0<H<1$, where $d\geq2$. Consider the approximation of the self-intersection local time of $B$, defined as 
\begin{align*}
I_{T}^{\varepsilon}
  &=\int_{0}^{T}\int_{0}^{t}p_{\varepsilon}(B_{t}-B_{s})dsdt,
\end{align*}  
where $p_\varepsilon(x)$ is the heat kernel. We prove that the process $\{I_{T}^{\varepsilon}-\E\left[I_{T}^{\varepsilon}\right]\}_{T\geq0}$, rescaled by a suitable normalization,  converges in law to a constant multiple of a standard Brownian motion for $\frac{3}{2d}<H\leq\frac{3}{4}$ and to a multiple of a sum of independent Hermite processes for $\frac{3}{4}<H<1$, in the space $C[0,\infty)$, endowed with the topology of uniform convergence on compacts.
\end{abstract}

\maketitle

\section{Introduction}
Let $B=\{B_{t}\}_{t\geq0}$ be a $d$-dimensional fractional Brownian motion of Hurst parameter $H\in(0,1)$. Fix $T>0$. The self-intersection local time of $B$ in the interval $[0,T]$ is formally defined by
\[
I:=\int_0^T \int_0^t \delta(B_t-B_s) d s dt,
\]
where $\delta$ denotes the Dirac delta function. A rigorous definition of this random variable may be obtained by approximating the delta function by the heat kernel
\begin{align*}
p_{\varepsilon}(x)
  &:=(2\pi\varepsilon)^{-\frac{d}{2}}\exp\left\{-\frac{1}{2\varepsilon}\Norm{x}^{2}\right\},~~~~x\in\R^{d}.
\end{align*}
In the case $H=\frac{1}{2}$, $B$ is a classical Brownian motion, and its self-intersection local time has been studied by many authors (see Albeverio (1995), Hu (1996), Imkeller, P\'erez-Abreu and Vives (1995), Varadhan (1969), Yor (1985) and the references therein). In the case $H\neq\frac{1}{2}$, the self-intersection local time for $B$ was first studied by Rosen in \cite{Ro} in the planar case and it was further investigated using techniques from Malliavin calculus by Hu and Nualart in \cite{HuNu}. In particular, it was proved that the approximation of the self-intersection local time of $B$ in $[0,T]$, defined by  
\begin{align}\label{eq:I}
I_{T}^{\varepsilon}
  &:=\int_{0}^{T}\int_{0}^{t}p_{\varepsilon}(B_{t}-B_{s}) d s dt,
\end{align}
converges in $L^{2}(\Omega)$ when $H<\frac{1}{d}$. Furthermore, it was shown that when $\frac{1}{d}\leq H<\frac{3}{2d}$, $I_{T}^{\varepsilon}-\E\left[I_{T}^{\varepsilon}\right]$ to converges in $L^{2}(\Omega)$, and for the case $\frac{3}{2d}< H<\frac{3}{4}$, the following limit theorem holds (see \cite[Theorem~2]{HuNu}).
\begin{Teorema}\label{eq:CLTHuNu}
If $\frac{3}{2d}<H<\frac{3}{4}$, then $\varepsilon^{\frac{d}{2}-\frac{3}{4H}}(I_{T}^{\varepsilon}-\E\left[I_{T}^{\varepsilon}\right])$  converges in law to a centered Gaussian distribution with variance $\sigma^{2}T$, as $\varepsilon\rightarrow0$, where the constant $\sigma^{2}$ is given by \eqref{eq:sigmadef}.
\end{Teorema}
The case $H=\frac{3}{2d}$ was addressed as well in \cite{{HuNu}}, where it was shown that the sequence $(\log(1/\varepsilon))^{-\frac{1}{2}}(I_{T}^{\varepsilon}-\E\left[I_{T}^{\varepsilon}\right])$  converges in law to a centered Gaussian distribution with variance $\sigma_{log}^2$, as $\varepsilon\rightarrow0$, where $\sigma_{log}^2$ is the constant given by \cite[Equation~(42)]{HuNu}.

The aim of this paper is to prove a functional version of Theorem \ref{eq:CLTHuNu}, and extend it to the case $\frac{3}{4}\leq H<1$. Our main results are Theorems \ref{Teo:convergence},  \ref{Teo:convergenceHermite}  and \ref{Teo:convergencelog}.
\begin{Teorema}\label{Teo:convergence}
Let $\frac{3}{2d}<H<\frac{3}{4}$, $d\geq2$ be fixed. Then, 
\begin{align}\label{conv:I1}
\{\varepsilon^{\frac{d}{2}-\frac{3}{4H}}(I_{T}^{\varepsilon}-\E\left[I_{T}^{\varepsilon}\right])\}_{T\geq0}
  &\stackrel{Law}{\rightarrow}\{\sigma W_{T}\}_{T\geq0},
\end{align}
in the space $C[0,\infty)$, endowed with the topology of uniform convergence on compact sets, where $W$ is a standard Brownian motion, and the constant $\sigma^{2}$ is given by \eqref{eq:sigmadef}. 
\end{Teorema}
We briefly outline the proof of \eqref{conv:I1}. The proof of the convergence of the finite-dimensional distributions, is based on the application of  a multivariate central limit theorem established by Peccati and Tudor in \cite{PeTu} (see Section \ref{section:4mthm}), and follows ideas similar to those presented in \cite{HuNu}.  On the other hand, proving  the tightness property for the process 
$$\widetilde{I}^{\varepsilon}_{T}:=\varepsilon^{\frac{d}{2}-\frac{3}{4H}}(I_{T}^{\varepsilon}-\E\left[I_{T}^{\varepsilon}\right]),$$ 
presents a great technical difficulty. In fact,  by the Billingsley criterion (see \cite[Theorem~12.3]{Billin}), the tightness property can be obtained by showing that there exists  $p>2$, such that for every $0\le T_{1}\le T_{2}$,
\begin{align}\label{eq:Billinglsleyintro}
\E\left[\Abs{ \widetilde{I}_{T_{2}}^{\varepsilon}-\widetilde{I}_{T_{1}}^{\varepsilon}}^{p}\right]\leq C\Abs{T_{2}-T_{1}}^{\frac p2},
\end{align} 
for some constant $C>0$ independent of $T_{1},T_{2}$ and $\varepsilon$. The problem of finding a bound like \eqref{eq:Billinglsleyintro} comes from the fact that the smallest even integer such that $p>2$ is $p=4$, and a direct  computation of the moment of order four   $  \E\left[ |  \widetilde{I} ^{\varepsilon}_{T_2} -\widetilde{I}_{T_{1}}^{\varepsilon}| ^{4}\right]$  is too complicated to be handled.
To overcome this difficulty,  in this paper we introduce a new approach to prove tightness based on the techniques of  Malliavin calculus.
Let us describe the main ingredients of this approach. 

 First, we write the centered random variable
$Z:=\widetilde{I}_{T_{2}}^{\varepsilon}-\widetilde{I}_{T_{1}}^{\varepsilon}$ as
\[
Z=- \delta D L^{-1} Z,
\]
where $\delta$, $D$ and $L$ are the basic operators in Malliavin calculus.   Then, taking into consideration that $\E\left[DL^{-1}Z\right]=0$ we apply 
  Meyer's inequalities to obtain a bound of the type
  \begin{align}\label{eq:ZpboundMeyer}
  \|Z\|_{L^{p}(\Omega)} \le  c_p\| D^2 L^{-1} Z\|_{L^{p}(\Omega;(\Hg^{d})^{\otimes 2})},
  \end{align}
  for any $p>1$, where the Hilbert space $\Hg$ is defined in Section \ref{subsec:chaos}. Notice that
\[
Z= \varepsilon^{\frac{d}{2}-\frac{3}{4H}}  \int_{0\le s \le t ,  T_1\le t \le  T_2} 
\left(p_\varepsilon (B_t-B_s) -\E\left[  p_\varepsilon  (B_t-B_s)\right]\right) dsdt.
\]
Applying Minkowski's inequality and \eqref{eq:ZpboundMeyer}, we obtain
\[
  \|Z\|_{L^p(\Omega)} \le  c_p \varepsilon^{\frac{d}{2}-\frac{3}{4H}}    \int_{0\le s \le t ,  T_1\le t \le  T_2}  \| D^2 L^{-1} p_\varepsilon (B_t-B_s)\|_p dsdt.
  \]
Then, we get the desired estimate by choosing $p>2$ close to 2,  using the self-similarity of the fractional Brownian motion, the expression of the operator $L^{-1}$ in terms of the Ornstein-Uhlenbeck semigroup, Mehler's formula and Gaussian computations. In this way, we reduce the problem to showing the finiteness of an integral (see Lemma \ref{lem:inegraltightness}),  similar to the integral appearing in the proof  of  the convergence of the variances. It is worth mentioning that this approach for proving tightness has not  been used before, and has its own interest. 

In the case $H>\frac{3}{4}$, the process $\varepsilon^{\frac{d}{2}-\frac{3}{2H}+1}(I_{T}^{\varepsilon}-\E\left[I_{T}^{\varepsilon}\right])$ also converges in law, in the topology of $C[0,\infty)$, but the limit is no longer a multiple of a Brownian motion, but a multiple of a sum of independent Hermite processes of order two. More precisely, if $\{X_{T}^{j}\}_{T\geq0}$ denotes the second order Hermite process, with respect to $\{B_{t}^{(j)}\}_{t\geq0}$, defined in Section \ref{subsec:chaos}, then $\{\widetilde{I}^{\varepsilon}\}_{\varepsilon\in(0,1)}$ satisfies the following limit theorem
\begin{Teorema}\label{Teo:convergenceHermite}
Let $H>\frac{3}{4}$, and $d\geq2$ be fixed. Then, for every $T>0$, 
\begin{align}\label{conv:I2L2}
\varepsilon^{\frac{d}{2}-\frac{3}{2H}+1}(I_{T}^{\varepsilon}-\E\left[I_{T}^{\varepsilon}\right])
  &\stackrel{L^{2}(\Omega)}{\rightarrow}-\Lambda \sum_{j=1}^{d}X_{T}^{j},
\end{align}
where the constant $\Lambda$ is defined by 
\begin{align}\label{eq:Lambdadef}
\Lambda
  :=\frac{(2\pi)^{-\frac{d}{2}}}{2}\int_{0}^{\infty}(1+u^{2H})^{-\frac{d}{2}-1}u^{2}du.
\end{align}
In addition,
\begin{align}\label{conv:I2}
\{\varepsilon^{\frac{d}{2}-\frac{3}{2H}+1}(I_{T}^{\varepsilon}-\E\left[I_{T}^{\varepsilon}\right])\}_{T\geq0}
  &\stackrel{Law}{\rightarrow}\{-\Lambda \sum_{j=1}^{d}X_{T}^{j}\}_{T\geq0},
\end{align}
in the space $C[0,\infty)$, endowed with the topology of uniform convergence on compact  sets.
\end{Teorema}
We briefly outline the proof of Theorem \ref{Teo:convergenceHermite}. The convergence \eqref{conv:I2L2} is obtained from the chaotic decomposition of $I^{\varepsilon}_{T}$. It turns out that the chaos of order two completely determines the asymptotic behavior of $\varepsilon^{\frac{d}{2}-\frac{3}{2H}+1}(I_{T}^{\varepsilon}-\E\left[I_{T}^{\varepsilon}\right])$, and consequently, \eqref{conv:I2L2} can be obtained by the characterization of the Hermite processes presented in \cite{NoNuTu}, applied to the second chaotic component of $I_{T}^{\varepsilon}$. Similarly to the case $\frac{3}{2d}<H<\frac{3}{4}$, we show that the sequence $\varepsilon^{\frac{d}{2}-\frac{3}{2H}+1}(I_{T}^{\varepsilon}-\E\left[I_{T}^{\varepsilon}\right])$ is tight, which proves the convergence in law \eqref{conv:I2}.

The technique we use to prove tightness doesn't work for the case $Hd\leq \frac{3}{2}$, so the convergence in law of $\{\log(1/\varepsilon)^{-\frac{1}{2}}(I_{T}^{\varepsilon}-\E\left[I_{T}^{\varepsilon}\right])\}_{T\geq0}$ to a scalar multiple of a Brownian motion for the case $Hd=\frac{3}{2}$ still remains open. Nevertheless, for the critical case $H=\frac{3}{4}$ and  $d\geq 3$, the technique does work, and we prove the following limit theorem
\begin{Teorema}\label{Teo:convergencelog}
Suppose $H=\frac{3}{4}$ and $d\geq3$. Then, 
\begin{align}\label{conv:I1log}
\{\frac{\varepsilon^{\frac{d}{2}-1}}{\sqrt{\log(1/\varepsilon)}}(I_{T}^{\varepsilon}-\E\left[I_{T}^{\varepsilon}\right])\}_{T\geq0}
  &\stackrel{Law}{\rightarrow}\{\rho W_{T}\}_{T\geq0},
\end{align}
in the space $C[0,\infty)$, endowed with the topology of uniform convergence on compact sets, where $W$ is a standard Brownian motion, and the constant $\rho$ is defined by \eqref{eq:rhodeflog}. 
\end{Teorema}
\noindent \textbf{Remark}\\ 
We impose the stronger condition $d\geq3$ instead of $d\geq2$, since the choice $H=\frac{3}{4}$, $d=2$ gives $Hd=\frac{3}{2}$, and as mentioned before, it is not clear how to prove tightness for this case.\\

We briefly outline the proof of Theorem \ref{Teo:convergencelog}. The proof of the tightness property is analogous to the case $\frac{3}{2d}<H<\frac{3}{4}$.
On the other hand, the proof of the convergence of the finite dimensional distributions requires a new approach. First we show that, as in the case $H>\frac{3}{4}$, the chaos of order two determines the asymptotic behavior of $\{I_{T}^{\varepsilon}\}_{T\geq0}$. Then we describe the behavior of the second chaotic component of $I_{T}^{\varepsilon}$, which we denote by  $J_{2}(I_{T}^{\varepsilon})$, and is given by 
\begin{equation}\label{eq:introRiemmanprev}
 J_{2}(I_{T}^{\varepsilon})\\
  =-\frac{(2\pi)^{-\frac{d}{2}}  \varepsilon^{\frac 23-\frac d2}}{2 }\sum_{j=1}^{d}\int_{0}^{T}\int_{0}^{\varepsilon^{-\frac{2}{3}}(T-s)}
  \frac {u^{\frac 32} }{(1+ u^{\frac 32})^{\frac d2+1} }
 H_{2}\left(\frac{B_{s+\varepsilon^{\frac{2}{3}}u}^{(j)}-B_{s}^{(j)}}{\sqrt{\varepsilon}u^{\frac{3}{4}}}\right)duds,
\end{equation}
where $H_{2}$ denotes the Hermite polynomial of order 2. Then we show that we can replace the  domain of integration of $u$ by $[0,\infty)$, and this integral can be approximated 
by Riemann sums of the type
\begin{align}\label{eq:introRiemman}
- \frac 1 {2^M} \sum_{k=2}^{ M2^M} \frac { u(k)^{\frac 32}}{(1+u(k)^{\frac{2}{3}})^{\frac{d}{2}+1}} \int_{0}^{T}H_{2}\left(\frac{B_{s+\varepsilon^{2^M}
u(k)}^{(j)}-B_{s}^{(j)}}{\sqrt{\varepsilon}u(k)^{\frac{3}{4}}}\right)ds,
\end{align} 
where  $u(k)=\frac k {2^M}$, and $M$ is some fixed positive number. By \cite[Equation~(1.4)]{DaNoNu}, we have that, for $k$ fixed, the random variable
\begin{align*}
\xi_{k}^{\varepsilon}(T)
  &:= \frac{\varepsilon^{-\frac{1}{3}}}{\sqrt{\log(1/\varepsilon)}}\int_{0}^{T}H_{2}\left(\frac{B_{s+\varepsilon^{\frac{2}{3}}u(k)}^{(j)}-B_{s}^{(j)}}{\sqrt{\varepsilon}u(k)^{\frac{3}{4}}}\right)ds
\end{align*} 
converges in law to a Gaussian distribution as $\varepsilon \rightarrow 0$. Hence, after a suitable analysis of the covariances of the process 
$\{\xi_{k}^{\varepsilon}(T)\ |\ 2\leq k\leq M2^M,\ \text{ and }\ T\geq0\}$ and an application of the Peccati-Tudor criterion (see \cite{PeTu}), we obtain that the process \eqref{eq:introRiemman} multiplied by the factor $  \frac {(2\pi)^{-\frac d2} \varepsilon ^{-\frac 13 }}{ 2 \sqrt{\log (1/\varepsilon)}}$
converges to a constant multiple of a Brownian motion $\rho_{M}W$, for some $\rho_{M}>0$. The result then follows by proving that the approximations \eqref{eq:introRiemman} to the integrals in the right-hand side of \eqref{eq:introRiemmanprev} are uniform over $\varepsilon\in(0,1/e)$ as $M\rightarrow\infty$, and that $\rho_{M}\rightarrow\rho$ as $M\rightarrow\infty$. 

The paper is organized as follows. In Section 2 we present some preliminary results on the fractional Brownian motion and the chaotic decomposition of $I_{T}^{\varepsilon}$. In Section 3, we compute the asymptotic behavior of the variances of the chaotic components of $I_{T}^{\varepsilon}$ as $\varepsilon\rightarrow0$. The proofs of the main results are presented in Section 4. Finally, in Section 5 we prove some technical lemmas.
\section{Preliminaries and main results}\label{sec:chaos}
\subsection{Some elements of Malliavin calculus for the fractional Brownian motion}\label{subsec:chaos}
Throughout the paper, $B=\{(B_{t}^{(1)},\dots, B_{t}^{(d)})\}_{t\ge 0}$ will denote a $d$-dimensional fractional Brownian motion with Hurst parameter $H\in(0,1)$, defined on a probability space $(\Omega,\Fc,\Pb)$. That is, $B$ is a centered, $\R^{d}$-valued Gaussian process with covariance function 
\begin{align*}
\E\left[B_{t}^{(i)}B_{s}^{(j)}\right]
  &=\frac{\delta_{{i,j}}}{2}(t^{2H}+s^{2H}-|t-s|^{2H}).
\end{align*}
We will denote by $\Hg$ the Hilbert space obtained by taking the completion of the space of step functions on $[0,\infty)$, endowed with the inner product 
\begin{align*}
\Ip{\Indi{[a,b]},\Indi{[c,d]}}_{\Hg}
  &:=  \E\left[\left(B_{b}^{(1)}-B_{a}^{(1)}\right)\left(B_{d}^{(1)}-B_{c}^{(1)}\right)\right],~~ \text{ for }~~0\leq a\leq b,\text{ and }0\leq c\leq d. 
\end{align*}
For every $1\leq j\leq d$ fixed, the mapping $\Indi{[0,t]} \mapsto B_{t}^{(j)}$ can be extended to linear isometry between $\Hg$ and the Gaussian subspace of $L^{2}\left(\Omega\right)$ generated by the process $B^{(j)}$. We will denote this isometry by $B^{(j)}(f)$, for $f\in\Hg$. If $f\in \Hg^{d}$ is of the form $f=(f_{1},\dots, f_{d})$, with $f_{j}\in \Hg$, we set $B(f):= \sum_{j=1} ^d B^{(j)}(f_{j})$.  Then $f\mapsto B(f)$ is a linear isometry between $\Hg^{d}$ and  the Gaussian subspace of   $L^2\left(\Omega \right)$ generated by $B$.
 
For any integer $q\geq1$, we denote by $(\Hg^{d})^{\otimes q}$ and $(\Hg^{d})^{\odot q}$ the $q$th tensor product of $\Hg^{d}$, and the $q$th symmetric tensor product of $\Hg^{d}$, respectively. The $q$th Wiener chaos	of $L^{2}(\Omega)$, denoted by $\Hc_{q}$, is the closed subspace of $L^{2}(\Omega)$ generated by the variables 
\[
\left \{\prod_{j=1}^{d}H_{q_{j}}(B^{(j)}(f_{j}))\ |\ \sum_{j=1}^{d}q_{j}=q,\text{ and } f_{1},\dots, f_{d}\in\Hg,\Norm{f_{j}}_{\Hg}=1 \right\},
\]
 where $H_{q}$ is the $q$th Hermite polynomal, defined by 
\begin{align*}
H_{q}(x)
  &:=(-1)^{q}e^{\frac{x^{2}}{2}}\frac{\text{d}^{q}}{\text{d}x^{q}}e^{-\frac{x^{2}}{2}}.
\end{align*}
For $q\in\N$, with $q\geq1$, and $f\in \Hg^{d}$ of the form $f=(f_{1},\dots, f_{d})$, with $\Norm{f_{j}}_{\Hg}=1$, we can write 
\begin{align*}
f^{\otimes q}
  =\sum_{i_{1},\dots, i_{q}=1}^{d}f_{i_{1}}\otimes\cdots \otimes f_{i_{q}}.
\end{align*}
For such $f$, we define the mapping 
\begin{align*}
I_{q}(f^{\otimes q})
  &:=\sum_{i_{1},\dots, i_{q}=1}^{d}\prod_{j=1}^{d}H_{q_{j}(i_{1},\dots, i_{q})}(B^{(j)}(f_{j})),
\end{align*}
where $q_{j}(i_{1},\dots, i_{q})$ denotes the number of indices in $(i_{1},\dots, i_{q})$ equal to $j$. The range of $I_{q}$ is contained in $\Hc_{q}$. Furthermore, this mapping can be extended to a linear isometry between $\Hg^{\odot q}$ (equipped with the norm $\sqrt{q!}\Norm{\cdot}_{(\Hg^{d})^{\otimes q}}$) and $\Hc_{q}$ (equipped with the $L^{2}(\Omega)$-norm).

Denote by $\Gc$ the $\sigma$-algebra generated by $B$. It is well known that every square integrable random variable $\Gc$-measurable, has a chaos decomposition of the type 
\begin{align*}
F=\E\left[F\right]+\sum_{q=1}^{\infty}I_{q}(f_{q}),
\end{align*}
for some $f_{q}\in(\Hg^{d})^{\odot q}$. In what follows, we will denote by $J_{q}(F)$, for $q\geq1$, the projection of $F$ over the $q$th Wiener chaos $\Hc_{q}$, and by $J_{0}(F)$ the expectation of $F$. 

Let $\Sc$ denote the set  of all cylindrical random variables of the form
\begin{align*}
F= g(B(h_{1}),\dots, B(h_{n})),
\end{align*} 
where $g:\R^{n}\rightarrow\R$ is an infinitely differentiable function with compact support, and $h_{j}\in\Hg^{d}$. The Malliavin derivative of $F$ with respect to $B$, is the element of $L^{2}(\Omega;\Hg^d)$, defined by 
\begin{align*}
DF
  &=\sum_{i=1}^{n}\frac{\partial g}{\partial x_{i}}(B(h_{1}),\dots, B(h_{n}))h_{i}.
\end{align*}
By iteration, one can define the $r$th derivative $D^{r}$ for every $r\geq2$, which is an element of $L^{2}(\Omega;(\Hg^d)^{\otimes r})$.

For $p\geq1$ and $r\geq1$, the space $\D^{r,p}$ denotes the closure of $\Sc$ with respect to the norm $\Norm{\cdot}_{\D^{r,p}}$, defined by 
\begin{align*}
\Norm{F}_{\D^{r,p}}
  &:=\left(\E\left[\Abs{F}^{p}\right]+\sum_{i=1}^{r}\E\left[\Norm{D^{i}F}_{(\Hg^{d})^{\otimes i}}^{p}\right]\right)^{\frac{1}{p}}.
\end{align*}
The operator $D^{r}$ can be consistently extended to the space $\D^{r,p}$. We denote by $\delta$ the adjoint of the operator $D$, also called the divergence operator. A random element $u\in L^{2}(\Omega;\Hg^{d})$ belongs to the domain of $\delta$, denoted by $\mathrm{Dom} \, \delta $, if and only if satisfies
\begin{align*}
\Abs{\E\left[\Ip{DF,u}_{\Hg^{d}}\right]}
  &\leq C_{u}\E\left[F^{2}\right]^{\frac{1}{2}},\ \text{ for every } F\in\D^{1,2},
\end{align*}
where $C_{u}$ is  a constant only depending on $u$. If $u\in \mathrm{Dom} \,\delta$, then the random variable $\delta(u)$ is defined by the duality relationship
\begin{align*}
\E\left[F\delta(u)\right]=\E\left[\Ip{DF,u}_{\Hg^{d}}\right],
\end{align*}
which holds for every $F\in\D^{1,2}$. The operator $L$ is defined on the Wiener chaos by 
$$LF:=\sum_{q=1}^{\infty}-q J_{q}F,\ \text{ for } F\in L^{2}(\Omega),$$
and coincides with the infinitesimal generator of the Ornstein-Uhlenbeck semigroup $\{P_{\theta}\}_{\theta\geq0}$, which is defined by 
\begin{align*}
P_{\theta}
  &:=\sum_{q=0}^{\infty}e^{-q\theta}J_{q}.
\end{align*}
A random variable $F$ belongs to the domain of $L$ if and only if $F\in \D^{1,2}$, and $DF\in \mathrm{Dom} \, \delta$, in which case
\begin{align*}
\delta DF
  &=-LF.
\end{align*}
We also define the operator $L^{-1}$ as 
$$L^{-1}F=\sum_{q=1}^{\infty}-\frac{1}{q}J_{q}F,\ \text{ for } F\in L^{2}(\Omega).$$
Notice that $L^{-1}$ is a bounded operator and satisfies $LL^{-1}F=F-\E\left[F\right]$ for every $F\in L^{2}(\Omega)$, so that $L^{-1}$ acts as a pseudo-inverse of $L$. The operator $L^{-1}$ satisfies the following contraction property for every $F\in L^{2}(\Omega)$ with $\E\left[F\right]=0$, 
\begin{align*}
\E\left[\Norm{DL^{-1}F}_{\Hg^{d}}^2\right]
  &\leq\E\left[F^2\right].
\end{align*}
In addition, by Meyer's inequalities (see \cite[Proposition~1.5.8]{Nualart}), for every $p>1$, there exists a constant $c_{p}>0$ such that the following relation holds for every $F\in \D^{2,p}$, with $\E\left[F\right]=0$
\begin{align}\label{eq:Meyer}
\Norm{\delta(DL^{-1}F)}_{L^{p}(\Omega)}
  &\leq c_{p}(\Norm{D^2L^{-1}F}_{L^{p}(\Omega;(\Hg^d)^{\otimes 2})}+\Norm{\E\left[DL^{-1}F\right]}_{(\Hg)^d}).
\end{align}
Assume that $\widetilde{B}$ is an independent copy of $B$, and such that $B,\widetilde{B}$ are defined in the product space $(\Omega\times \widetilde{\Omega}, \Fc \otimes\widetilde{\Fc},\Pb\otimes\widetilde{\Pb})$. Given a random variable $F\in L^{2}(\Omega)$, measurable with respect to the $\sigma$-algebra generated by $B$, we can write $F=\Psi_{F}(B)$, where $\Psi_{F}$ is a measurable mapping from $\R^{\Hg^{d}}$ to $\R$, determined $\Pb$-a.s. Then, for every $\theta\geq0$ we have the Mehler formula 
\begin{align}\label{eq:Mehler}
P_{\theta}F
  &=\widetilde{\E}\left[\Psi_{F}(e^{-\theta}B+\sqrt{1-e^{-2\theta}}\widetilde{B})\right],
\end{align}
where $\widetilde{\E}$ denotes the expectation with respect to $\widetilde{\Pb}$. The operator $L^{-1}$ can be expressed in terms of $P_{\theta}$, as follows
\begin{align}\label{eq:Mehler2}
L^{-1}F
  &=\int_{0}^{\infty}P_{\theta}Fd\theta,~ \text{ for } F \text{ such that }\E\left[F\right]=0.
\end{align}

\subsection{Hermite process}\label{subsec:Hermiteprocess} When $H>\frac{1}{2}$, the inner product in the space $\Hg$ can be written, for every step functions $\varphi,\vartheta$ on $[0,\infty)$, as
\begin{align}\label{eq:IpHbig}
\Ip{\varphi,\vartheta}_{\Hg}
  &=H(2H-1)\int_{\R_{+}^2}\varphi(\xi)\vartheta(\nu)\Abs{\xi-\nu}^{2H-2}d\xi d\nu.
\end{align}
Following \cite{NoNuTu}, we introduce the Hermite process $\{X_{T}^{j}\}_{T\geq0}$ of order 2, associated to the $j$th component of $B$, $\{B_{t}^{(j)}\}_{t\geq0}$, and describe some of its properties. The family of kernels $\{\varphi_{j,T}^{\varepsilon}\ |\ T\geq0,\varepsilon\in(0,1)\}\subset(\Hg^{d})^{\otimes 2}$, defined, for every multi-index $\textbf{i}=(i_{1},i_{2})$, $1\leq i_{1},i_{2}\leq d$, by 
\begin{align}\label{eq:phidef}
\varphi_{j,T}^{\varepsilon}(\textbf{i},x_{1},x_{2})
  &:=\varepsilon^{-2}\int_{0}^{T}\delta_{j,i_{1}}\delta_{j,i_{2}}\Indi{[s,s+\varepsilon]}(x_{1})\Indi{[s,s+\varepsilon]}(x_{2})ds,
\end{align} 
satisfies the following relation for every $H>\frac{3}{4}$, and $T\geq0$
\begin{align}\label{eq:Ipphi}
\lim_{\varepsilon,\eta\rightarrow0}\Ip{\varphi_{j,T}^{\varepsilon},\varphi_{j,T}^{\eta}}_{(\Hg^{d})^{\otimes 2}}
  &=H^{2}(2H-1)^2\int_{[0,T]^{2}}\Abs{s_{1}-s_{2}}^{4H-4}ds_{1}ds_{2}=c_{H}T^{4H-2},
\end{align}
where $c_{H}:=\frac{H^2(2H-1)}{4H-3}$. This implies that $\varphi_{j,T}^{\varepsilon}$ converges, as $\varepsilon\rightarrow0$, to an element of $(\Hg^{d})^{\otimes 2}$, denoted by $\pi_{T}^{j}$. In particular, for every $K>0$, $\Norm{\varphi_{j,K}^{\varepsilon}}_{(\Hg^{d})^{\otimes 2}}$ is bounded by some constant $C_{K,H}$, only depending on $K$ and $H$. On the other hand, by \eqref{eq:IpHbig} and \eqref{eq:phidef},  we deduce that for every  $T\in[0,K]$, it holds $\Norm{\varphi_{j,T}^{\varepsilon}}_{(\Hg^{d})^{\otimes 2}}\leq \Norm{\varphi_{j,K}^{\varepsilon}}_{(\Hg^{d})^{\otimes 2}}$, and hence
\begin{align}\label{ineq:supIpphi}
\sup_{\substack{T_{1},T_{2}\in(0,K]\\\varepsilon,\eta\in(0,1)}}\Abs{\Ip{\varphi_{j,T_{1}}^{\varepsilon},\varphi_{j,T_{2}}^{\eta}}_{(\Hg^{d})^{\otimes 2}}}
  &\leq \sup_{\substack{T_{1},T_{2}\in(0,K]\\\varepsilon,\eta\in(0,1)}}\Norm{\varphi_{j,T_{1}}^{\varepsilon}}_{(\Hg^{d})^{\otimes 2}}\Norm{\varphi_{j,T_{2}}^{\eta}}_{(\Hg^{d})^{\otimes 2}}\nonumber\\
  &\leq\sup_{\varepsilon\in(0,1)}\Norm{\varphi_{j,K}^{\varepsilon}}_{(\Hg^{d})^{\otimes 2}}^{2}\leq C_{K,H}.
\end{align}
The element $\pi_{T}^{j}$, can be characterized as follows. For any vector of step functions with compact support $f_{i}=(f_{i}^{(1)},\dots, f_{i}^{(d)})\in\Hg^{d}$, $i=1,2$, we have
\begin{align*}
\Ip{\pi_{t}^{j},f_{1}\otimes f_{2}}_{(\Hg^{d})^{\otimes 2}}
  &=\lim_{\varepsilon\rightarrow0}\Ip{\varphi_{j,t}^{\varepsilon},f_{1}\otimes f_{2}}_{(\Hg^{d})^{\otimes 2}}\\
	&=\lim_{\varepsilon\rightarrow0}\varepsilon^{-2}H^{2}(2H-1)^{2}\\
	&\times \int_{0}^{T}\prod_{i=1,2}\int_s^{s+\varepsilon}\int_{0}^{T}\Abs{\xi-\eta}^{2H-2}f_{i}^{(j)}(\eta)d\eta d\xi ds
\end{align*}
and hence 
\begin{align}\label{eq:picharact}
\Ip{\pi_{t}^{j},f_{1}\otimes f_{2}}_{(\Hg^{d})^{\otimes 2}}
	&=H^{2}(2H-1)^{2}\int_{0}^{T}\prod_{i=1,2}\int_{0}^{T}\Abs{s-\eta}^{2H-2}f_{i}^{(j)}(\eta)d\eta ds.
\end{align}
We define the second order Hermite process $\{X_{T}^{j}\}_{T\geq0}$, with respect to $\{B_{t}^{(j)}\}_{t\geq0}$, as $X_{T}^{j}:=I_{2}(\pi_{T}^{j})$.
\subsection{A multivariate central limit theorem}\label{section:4mthm}
In the  seminal paper \cite{NuPe}, Nualart and Peccati established a central limit theorem for sequences of multiple stochastic integrals of a fixed order. In this context, assuming that the variances converge, convergence in distribution to  a centered  Gaussian law is actually equivalent to convergence of just the fourth moment. Shortly afterwards, in \cite{PeTu}, Peccati and Tudor gave a multidimensional version of this characterization. More recent developments on these type of results have been addressed by using Stein's method and Malliavin techniques (see the monograph by Nourdin and Peccati \cite{NoPe} and the references therein). In the sequel, we will use the following  multivariate central limit theorem obtained by Peccati and Tudor in \cite{PeTu} (see also Theorems 6.2.3 and 6.3.1 in \cite{NoPe}).
\begin{Teorema}\label{Fourthmoment}
For $r\in\N$ fixed, consider a sequence $\{F_{n}\}_ {n\geq1 }$ of  random vectors of the form $F_{n}=(F_{n}^{(1)},\dots, F_{n}^{(r)})$. Suppose that for $i=1,\dots, r$ and $n\in\N$, the random variables $F_{n}^{(i)}$ belong to $L^{2}(\Omega)$, and have chaos decomposition 
\begin{align*}
F_{n}^{(i)}
  &=\sum_{q=1}^{\infty}I_{q}(f_{q,i,n}),
\end{align*} 
for some $f_{q,i,n}\in(\Hg^{d})^{\otimes q}$. Suppose, in addition, that for every $q\geq1$, there is a real symmetric non negative definite matrix $C_{q}=\{C_{q}^{i,j}\ |\ 1\leq i,j\leq r\}$, such that the following conditions hold:
\begin{enumerate}
\item[(i)] For every fixed $q\geq1$, and $1\leq i,j\leq r$, we have $q!\Ip{f_{q,i,n},f_{q,j,n}}_{(\Hg^{d})^{\otimes q}}\rightarrow C_{q}^{i,j}$ as $n\rightarrow\infty$.
\item[(ii)] There exists a real symmetric nonnegative definite matrix $C=\{C^{i,j}\ |\ 1\leq i,j\leq r\}$, such that
$C^{i,j}=\lim_{Q\rightarrow\infty}\sum_{q=1}^{Q}C_{q}^{i,j}$.
\item[(iii)] For all $q\geq1$ and $i=1,\dots, r$, the sequence $\{I_{q}(f_{q,i,n})\}_{n\geq1}$ converges in law to a centered Gaussian distribution as $n\rightarrow\infty$.
\item[(iv)] $\lim_{Q\rightarrow\infty}\sup_{n\geq1}\sum_{q=Q}^{\infty}q!\Norm{f_{q,i,n}}_{(\Hg^{d})^{\otimes q}}^{2}=0$, for all $i=1,\dots, r$.
\end{enumerate}
Then, $F_{n}$ converges in law as $n\rightarrow\infty$, to a centered Gaussian vector with covariance matrix $C$.
\end{Teorema}
\subsection{Chaos decomposition for the self-intersection local time}
In this section we describe the chaos decomposition of the variable $I_{T}^{\varepsilon}$ defined by \eqref{eq:I}. Let $\varepsilon\in(0,1)$, and $T\geq0$ be fixed. Define the set
\begin{align*}
\Rc:=\{(s,t)\in\R_{+}^{2}\ |\ s\leq t\leq 1\}.
\end{align*}
For every $\gamma>0$, we will denote by $\gamma \Rc$ the set $\gamma\Rc:=\{\gamma v\ |\ v\in\Rc\}$. First we write
\begin{align}\label{eq:I1}
I_{T}^{\varepsilon}= \int_{\R_{+}^2}\Indi{T\Rc}(s,t)p_{\varepsilon}(B_{t}-B_{s}) dsdt.
\end{align}
We can determine the chaos decomposition of the random variable $p_{\varepsilon}(B_{t}-B_{s})$ appearing in \eqref{eq:I1} as follows. Given a multi-index $\textbf{i}_{n}=(i_{1},\dots,i_{n})$, $n\in\N$, $1\leq i_{j}\leq d$, we set 
\begin{align*}
\alpha(\textbf{i}_{n})
  &:=\E\left[X_{i_{1}}\cdots X_{i_{n}}\right],
\end{align*}
where the $X_{i}$ are independent standard Gaussian random variables. Notice that 
\begin{align}\label{eq:alphaindex}
\alpha(\textbf{i}_{2q})
  &=\frac{(2q_{1})!\cdots(2q_{d})!}{(q_{1})!\cdots(q_{d})!2^{q}},
\end{align}
 if $n=2q$ is even and for each $k=1,\dots,d$, the number of components of $\textbf{i}_{2q}$ equal to $k$, denoted by $2q_{k}$, is also even, and $\alpha(\textbf{i}_{n})=0$ otherwise. Proceeding as in \cite[Lemma~7]{HuNu}, we can prove that 
\begin{align}\label{eq:chaos:1}
p_{\varepsilon}(B_{t}-B_{s})
  &= \E\left[p_{\varepsilon}(B_{t}-B_{s})\right] +\sum_{q=1}^{\infty}I_{2q}\left(f_{2q,s,t}^{\varepsilon}\right),
\end{align}
where $f_{2q,s,t}^{\varepsilon}$ is the element of $(\Hg^{d})^{\otimes 2q}$, given by
\begin{align}\label{eq:kerneldef}
f_{2q,s,t}^{\varepsilon}(\textbf{i}_{2q},x_{1},\dots,x_{2q})
 &:= (-1)^{q}\frac{(2\pi)^{-\frac{d}{2}}\alpha(\textbf{i}_{2q})}{(2q)!}(\varepsilon+(t-s)^{2H})^{-\frac{d}{2}-q}\prod_{j=1}^{2q}\Indi{[s,t]}(x_{j}),
\end{align}
and 
\begin{align}\label{eq:Epinc}
\E\left[p_{\varepsilon}(B_{t}-B_{s})\right]
  &=(2\pi)^{-\frac{d}{2}}(\varepsilon+(t-s)^{2H})^{-\frac{d}{2}}.
\end{align}
By \eqref{eq:I1}, \eqref{eq:chaos:1} and \eqref{eq:Epinc}, it follows that the random variable $I_{T}^{\varepsilon}$ has 	the chaos decomposition
\begin{align}\label{eq:chaos2}
I_{T}^{\varepsilon}
  &=\E\left[I_{T}^{\varepsilon}\right]+\sum_{q=1}^{\infty}I_{2q}(h_{2q,T}^{\varepsilon}),
\end{align}
where 
\begin{align}\label{eq:kerneldef2}
h_{2q,T}^{\varepsilon}(\textbf{i}_{2q},x_{1},\dots,x_{2q})
 &:= \int_{\R_{+}^2}\Indi{T\Rc}(s,t)f_{2q,s,t}^{\varepsilon}(\textbf{i}_{2q},x_{1},\dots,x_{2q})dsdt,
\end{align}
and 
\begin{align}\label{eq:EI}
\E\left[I_{T}^{\varepsilon}\right]
  &=(2\pi)^{-\frac{d}{2}}\int_{\R_{+}^2}\Indi{T\Rc}(s,t)(\varepsilon+(t-s)^{2H})^{-\frac{d}{2}}dsdt.
\end{align}
In Section 3, we will describe the behavior as $\varepsilon\rightarrow0$ of the covariance function of the processes $\{I_{T}^{\varepsilon}\}_{T\geq0}$ and $\{I_{2q}(h_{2q,T}^{\varepsilon})\}_{T\geq0}$. In order to address this problem, we will first introduce some notation that will help us to describe the covariance function of the variables $p_{\varepsilon}(B_{t}-B_{s})$ and its chaotic components, which ultimately will lead to an expresion for the covariance function of $I_{T}^{\varepsilon}$.

First we describe the inner product $\Ip{f_{2q,s_{1},t_{1}}^{\varepsilon},f_{2q,s_{2},t_{2}}^{\varepsilon}}_{(\Hg^{d})^{\otimes2q}}$. From \eqref{eq:kerneldef}, we can prove that for every $0\leq s_{1}\leq t_{1}$ and $0\leq s_{2}\leq t_{2}$, 
\begin{align}\label{eq:Ipfexpanded}
\Ip{f_{2q,s_{1},t_{1}}^{\varepsilon},f_{2q,s_{2},t_{2}}^{\varepsilon}}_{(\Hg^{d})^{\otimes2q}}
  &=\sum_{q_{1}+\cdots+q_{d}=q}(2q_{1},\dots, 2q_{d})!\frac{(2\pi)^{-d}\alpha(\textbf{i}_{2q})^{2}}{((2q)!)^2}(\varepsilon+(t_{1}-s_{1})^{2H})^{-\frac{d}{2}-q}\nonumber\\
	&~~\times(\varepsilon+(t_{2}-s_{2})^{2H})^{-\frac{d}{2}-q}\Ip{\Indi{[s_{1},t_{1}]}^{\otimes2q},\Indi{[s_{2},t_{2}]}^{\otimes2q}}_{\Hg^{\otimes2q}},
\end{align}
where $(2q_{1},\dots, 2q_{d})!$ denotes the multinomial coefficient $(2q_{1},\dots, 2q_{d})!=\frac{(2q)!}{(2q_{1})!\cdots(2q_{d})!}$. To compute the term $\Ip{\Indi{[s_{1},t_{1}]}^{\otimes 2q},\Indi{[s_{2},t_{2}]}^{\otimes 2q}}_{\Hg^{\otimes 2q}}$ appearing in the previous expression, we will introduce the following notation. For every $x,u_{1},u_{2}>0$, define
\begin{align}\label{eq:mudef}
\mu(x,u_{1},u_{2})
  &:= \E\left[B_{u_{1}}^{(1)}\left(B_{x+u_{2}}^{(1)}-B_{x}^{(1)}\right)\right].
\end{align}
Define as well $\mu(x,u_{1},u_{2})$, for $x<0$, by $\mu(x,u_{1},u_{2}):=\mu(-x,u_{2},u_{1})$. Using the property of stationary increments of $B$, we can check that for every $s_{1},s_{2},t_{1},t_{2}\geq0$, such that $s_{1}\leq t_{1}$ and $s_{2}\leq  t_{2}$,	it holds 
\begin{align}\label{eq:covariancemu}
\E\left[\left(B_{t_{1}}^{(1)}-B_{s_{1}}^{(1)}\right)\left(B_{t_{2}}^{(1)}-B_{s_{2}}^{(1)}\right)\right]
  &=\mu(s_{2}-s_{1},t_{1}-s_{1},t_{2}-s_{2}).
\end{align}
 As a consequence, by \eqref{eq:alphaindex} and \eqref{eq:Ipfexpanded},
\begin{align*}
\Ip{f_{2q,s_{1},t_{1}}^{\varepsilon},f_{2q,s_{2},t_{2}}^{\varepsilon}}_{(\Hg^{d})^{\otimes2q}}
  &=\frac{\alpha_{q}}{(2\pi)^{d}(2q)!2^{2q}}(\varepsilon+(t_{1}-s_{1})^{2H})^{-\frac{d}{2}-q}(\varepsilon+(t_{2}-s_{2})^{2H})^{-\frac{d}{2}-q}\\
	&~~\times\mu(s_{2}-s_{1},t_{1}-s_{1},t_{2}-s_{2})^{2q},
\end{align*}
where the constant $\alpha_{q}$ is defined by
\begin{align}\label{eq:alphaqdef}
\alpha_{q}
  &:=\sum_{q_{1}+\cdots+q_{d}=q}\frac{(2q_{1})!\cdots(2q_{d})!}{(q_{1}!)^{2}\cdots(q_{d}!)^2}.
\end{align}
From here we can conclude that  
\begin{align}\label{eq:innerproductG}
\Ip{f_{2q,s_{1},t_{1}}^{\varepsilon},f_{2q,s_{2},t_{2}}^{\varepsilon}}_{(\Hg^{d})^{\otimes2q}}
  &=\frac{\alpha_{q}}{(2\pi)^{d}(2q)!2^{2q}}G_{\varepsilon,s_{2}-s_{1}}^{(q)}(t_{1}-s_{1},t_{2}-s_{2}),
\end{align}
where $G_{\varepsilon,x}^{(q)}(u_{1},u_{2})$ is defined by
\begin{align}\label{eq:Gdef}
G_{\varepsilon,x}^{(q)}(u_{1},u_{2})
  &:=\left(\varepsilon+u_{1}^{2H}\right)^{-\frac{d}{2}-q}\left(\varepsilon+u_{2}^{2H}\right)^{-\frac{d}{2}-q}\mu(x,u_{1},u_{2})^{2q}.
\end{align}

Now we describe the covariance $\mathrm{Cov}\left[p_{\varepsilon}\left(B_{t_{1}}-B_{s_{1}}\right),p_{\varepsilon}\left(B_{t_{2}}-B_{s_{2}}\right)\right]$.  Using the chaos expansion
\eqref{eq:chaos:1} and \eqref{eq:innerproductG}, we obtain
\begin{equation}  \label{eq1}
\mathrm{Cov}\left[p_{\varepsilon}\left(B_{t_{1}}-B_{s_{1}}\right),p_{\varepsilon}\left(B_{t_{2}}-B_{s_{2}}\right)\right] =
\sum_{q=1} ^\infty \frac{\alpha_{q}}{(2\pi)^{d}2^{2q}}G_{\varepsilon,s_{2}-s_{1}}^{(q)}(t_{1}-s_{1},t_{2}-s_{2}).
\end{equation}
On the   other hand, using once more the property of stationary increments of $B$, we can prove that for every $s_{1}\leq t_{1}$, and $s_{2}\leq t_{2}$,
\begin{align}\label{eq:CovpF}
\mathrm{Cov}\left[p_{\varepsilon}\left(B_{t_{1}}-B_{s_{1}}\right),p_{\varepsilon}\left(B_{t_{2}}-B_{s_{2}}\right)\right]
  &= F_{\varepsilon,s_{2}-s_{1}}(t_{1}-s_{1},t_{2}-s_{2}),
\end{align}
where the function $F_{\varepsilon,x}(u_{1},u_{2})$, for $u_{1},u_{2}>0$, is defined by  
\begin{align}\label{eq:Fdef}
F_{\varepsilon,x}(u_{1},u_{2})
  &:= \mathrm{Cov}\left[p_{\varepsilon}(B_{u_{1}}),p_{\varepsilon}\left(B_{x+u_{2}}-B_{x}\right)\right],
\end{align}
in the case $x>0$, and by $F_{\varepsilon,x}(u_{1},u_{2}):=F_{\varepsilon,-x}(u_{2},u_{1})$ in the case $x<0$. Proceeding as in \cite{HuNu}, equations (13)-(14), we can prove that for every $u_{1},u_{2}\geq0$, $x\in\R$, 
\begin{align}\label{eq:Fidentitybasic}
F_{\varepsilon,x}(u_{1},u_{2})
  &= (2\pi)^{-d}\bigg{[}\left((\varepsilon+u_{1}^{2H})(\varepsilon+u_{2}^{2H})-\mu(x,u_{1},u_{2})^{2}\right)^{-\frac{d}{2}}\nonumber\\
	&~~-(\varepsilon+u_{1}^{2H})^{-\frac{d}{2}}(\varepsilon+u_{2}^{2H})^{-\frac{d}{2}}\bigg{]},
\end{align}
and consequently,
\begin{align}\label{eq:Fidentity}
F_{\varepsilon,x}(u_{1},u_{2})
  &= (2\pi)^{-d}(\varepsilon+u_{1}^{2H})^{-\frac{d}{2}}(\varepsilon+u_{2}^{2H})^{-\frac{d}{2}}\nonumber\\
	&~~\times\left(\left(1-\frac{\mu(x,u_{1},u_{2})^{2}}{(\varepsilon+u_{1}^{2H})(\varepsilon+u_{2}^{2H})}\right)^{-\frac{d}{2}}-1\right).
\end{align}
From \eqref{eq1} and  \eqref{eq:CovpF} it follows that the functions $G_{\varepsilon, x}^{(q)}(u_{1},u_{2})$ and $F_{\varepsilon,x}(u_{1},u_{2})$ appearing in \eqref{eq:innerproductG} and \eqref{eq:Fidentity} are related in the following manner:
\begin{align}\label{eq:Fidentitychaos}
F_{\varepsilon,x}(u_{1},u_{2})
  &= \sum_{q=1}^{\infty}\beta_{q}G_{\varepsilon,x}^{(q)}(u_{1},u_{2}),
\end{align}
where $\beta_{q}$ is defined by 
\begin{align}\label{eq:betaq}
\beta_{q}
  &:= \frac{\alpha_{q}}{(2\pi)^{d}2^{2q}}.
\end{align}
The  functions $G_{1,x}^{(q)}(u_{1},u_{2})$ and $F_{1,x}(u_{1},u_{2})$ satisfy the  following useful integrability condition, which was proved in \cite[Lemma~13]{HuNu}, .
\begin{Lema}\label{lem:finiteFGintegral}
Let $\frac{3}{2d}<H<\frac{3}{4}$, and $q\in\N$, $q\geq1$ be fixed. Define  $G_{1,x}^{(q)}(u_{1},u_{2})$ by \eqref{eq:Gdef} and $\beta_{q}$ by \eqref{eq:betaq}. Then, 
\begin{align*}
\beta_{q}\int_{\R_{+}^{3}}G_{1,x}^{(q)}(u_{1},u_{2})\text{d}x\text{d}u_{1}\text{d}u_{2}\leq 
\int_{\R_{+}^{3}}F_{1,x}(u_{1},u_{2})\text{d}x\text{d}u_{1}\text{d}u_{2}<\infty.
\end{align*}
\end{Lema}
\begin{proof}
By \eqref{eq:Fidentitychaos}, it follows that $\beta_{q}G_{1,x}^{(q)}(u_{1},u_{2})\leq F_{1,x}(u_{1},u_{2})$. The integrability of the function $F_{1,x}(u_{1},u_{2})$ over $x,u_{1},u_{2}\geq0$, written as in \eqref{eq:Fidentitybasic}, is proved in \cite[Lemma~13]{HuNu} (see equation (40) for notation reference). 
\end{proof}
With the notation previously introduced, we can compute the covariance functions of the increments of the processes $\{I_{T}^{\varepsilon}\}_{T\geq0}$ and 
$\{I_{2q}(h_{2q,T}^{\varepsilon})\}_{T\geq0}$ as follows. Define the set $\mathcal{K}_{T_{1},T_{2}}$ by 
\begin{align}\label{eq:Sdef}
\mathcal{K}_{T_{1},T_{2}}
  &:=\{(s,t)\in\R_{+}^2\ |\ s\leq t,~\text{ and }~ T_{1}\leq t\leq  T_{2}\}.
\end{align}
By \eqref{eq:I1} and \eqref{eq:kerneldef2}, for every $T_{1}<T_{2},$ we can write
\begin{equation*}
I_{T_{2}}^{\varepsilon}-\E [I_{T_{2}}^{\varepsilon}]-\left(I_{T_{1}}^{\varepsilon}-\E[I_{T_{1}}^{\varepsilon}]\right)
  =\int_{\R_{+}^{2}}\Indi{\mathcal{K}_{T_{1},T_{2}}}(s,t)\left(p_{\varepsilon}(B_{t}-B_{s})-\E\left[p_{\varepsilon}(B_{t}-B_{s})\right]\right)dsdt,
\end{equation*}
and 
\begin{align*}
I_{2q}(h_{2q,T_{2}}^{\varepsilon})-I_{2q}(h_{2q,T_{1}}^{\varepsilon})
  =\int_{\R_{+}^{2}}\Indi{\mathcal{K}_{T_{1},T_{2}}}(s,t)I_{2q}(f_{2q,s,t}^{\varepsilon})dsdt.
\end{align*}
By \eqref{eq:CovpF}, we deduce the following identity for every $T_{1}\leq T_{2}$ and $\widetilde{T}_{1}\leq \widetilde{T}_{2}$, 
\begin{equation}\label{eq:CovIF}
\mathrm{Cov}\left[I_{T_{2}}^{\varepsilon}-I_{T_{1}}^{\varepsilon},I_{\widetilde{T}_{2}}^{\varepsilon}-I_{\widetilde{T}_{1}}^{\varepsilon}\right]
  = \int_{\R_{+}^{4}}\Indi{\mathcal{K}_{T_{1},T_{2}}}(s_{1},t_{1})\Indi{\mathcal{K}_{\widetilde{T}_{1},\widetilde{T}_{2}}}(s_{2},t_{2})F_{\varepsilon,s_{2}-s_{1}}(t_{1}-s_{1},t_{2}-s_{2})ds_{1}ds_{2}dt_{1}dt_{2}.
\end{equation}
Similarly, by \eqref{eq:innerproductG}, 
\begin{multline}\label{eq:covbasicchaos}
\E\left[(I_{2q}(h_{2q,T_{2}}^{\varepsilon})-I_{2q}(h_{2q,T_{1}}^{\varepsilon}))(I_{2q}(h_{2q,\widetilde{T}_{2}}^{\varepsilon})-I_{2q}(h_{2q,\widetilde{T}_{1}}^{\varepsilon}))\right]\\
  =\beta_{q}\int_{\R_{+}^{4}}\Indi{\mathcal{K}_{T_{1},T_{2}}}(s_{1},t_{1})\Indi{\mathcal{K}_{\widetilde{T}_{1},\widetilde{T}_{2}}}(s_{2},t_{2})G_{\varepsilon,s_{2}-s_{1}}^{(q)}(t_{1}-s_{1},t_{2}-s_{2})ds_{1}ds_{2}dt_{1}dt_{2},
\end{multline}
where $\beta_{q}$ is defined by \eqref{eq:betaq}.

We end this section by introducing some notation, which will be used throughout the paper to describe expectations of the form $\E\left[p_{\varepsilon}(B_{t_{1}}-B_{s_{1}})p_{\varepsilon}(B_{t_{2}}-B_{s_{2}})\right]$. For every $n$-dimensional non-negative definite matrix $A$, we will denote by $\phi_{A}$ the density function of a Gaussian vector with mean zero and covariance $A$. In addition, we will denote by $\Abs{A}$ the determinant of $A$, and by $I_{n}$ the identity matrix of dimension $n$. 

Let $\Sigma$ be the covariance matrix of the 2-dimensional random vector $(B_{t_{1}}^{(1)}-B_{s_{1}}^{(1)},B_{t_{2}}^{(1)}-B_{s_{2}}^{(1)})$. Then, the covariance matrix of the $2d$-dimensional random vector $(B_{t_{1}}-B_{s_{1}},B_{t_{2}}-B_{s_{2}})$ can be written as
\begin{align*}
\mathrm{Cov}(B_{t_{1}}-B_{s_{1}},B_{t_{2}}-B_{s_{2}})
  &=I_{d}\otimes \Sigma,
\end{align*}
where in the previous identity $\otimes$ denotes the Kronecker product of matrices. Consider the $2d$-dimensional Gaussian density $\phi_{\varepsilon I_{2d}}(x,y)=p_{\varepsilon}(x)p_{\varepsilon}(y)$, where $x,y\in \R^{d}$, and denote by  $*$  the convolution operation. Then we have that 
\begin{align*}
\E\left[p_{\varepsilon}(B_{t_{1}}-B_{s_{1}})p_{\varepsilon}(B_{t_{2}}-B_{s_{2}})\right]
  &=\int_{\R^{2d}}\phi_{\varepsilon I_{2d}}(x,y)\phi_{I_{d}\otimes \Sigma}(-x,-y)dxdy\\
	&= \phi_{\varepsilon I_{2d}}*\phi_{I_{d}\otimes \Sigma}(0,0)=(2\pi)^{-d}\left|\varepsilon I_{2d} + I_{d}\otimes\Sigma\right|^{-\frac{1}{2}}.
\end{align*}
From the previous equation it follows that
\begin{align}\label{eq:pdet0}
\E\left[p_{\varepsilon}(B_{t_{1}}-B_{s_{1}})p_{\varepsilon}(B_{t_{2}}-B_{s_{2}})\right]
  &= (2\pi)^{-d}\left|\varepsilon I_{2} + \Sigma\right|^{-\frac{d}{2}}.
\end{align}
The right-hand side of the previous identity can be rewritten as follows. Define the function 
\begin{align}\label{eq:Ldef}
\Theta_{\varepsilon}(x,u_{1},u_{2})
  &:= \varepsilon^2+\varepsilon(u_{1}^{2H}+u_{2}^{2H})+u_{1}^{2H}u_{2}^{2H}-\mu(x,u_{1},u_{2})^2.
\end{align}
Then, using \eqref{eq:covariancemu}, we can easily show that 
\begin{align*}
\Abs{\varepsilon I_{2}+\Sigma}
  &=\Theta_{\varepsilon}(s_{2}-s_{1},t_{1}-s_{1},t_{2}-s_{2}),
\end{align*}
which, by \eqref{eq:pdet0}, implies that 
\begin{align}\label{eq:pdet}
\E\left[p_{\varepsilon}(B_{t_{1}}-B_{s_{1}})p_{\varepsilon}(B_{t_{2}}-B_{s_{2}})\right]
  &= (2\pi)^{-d}\Theta_{\varepsilon}(s_{2}-s_{1},t_{1}-s_{1},t_{2}-s_{2})^{-\frac{d}{2}}.
\end{align}
Therefore,  we can write $\E\left[(I_{T}^{\varepsilon})^2\right]$, as 
\begin{align}\label{eq:Im2}
\E\left[(I_{T}^{\varepsilon})^2\right]
  &= (2\pi)^{-d}\int_{(T\Rc)^2}\Theta_{\varepsilon}(s_{2}-s_{1},t_{1}-s_{1},t_{2}-s_{2})^{-\frac{d}{2}}ds_{1}ds_{2}dt_{1}dt_{2}.
\end{align}
Finally, we prove the following inequality, which estimates the function $F_{\varepsilon,x}(u_{1},u_{2})$, defined in \eqref{eq:Fdef}, in terms of $\Theta_{\varepsilon}(x,u_{1},u_{2})$
\begin{align}\label{ineq:FTheta}
F_{\varepsilon,x}(u_{1},u_{2})
  &\leq (2\pi)^{-d}\left(\frac{d}{2}+1\right)\frac{\mu(x,u_{1},u_{2})^{2}}{u_{1}^{2H}u_{2}^{2H}}\Theta_{\varepsilon}(x,u_{1},u_{2})^{-\frac{d}{2}}.
\end{align}  
Indeed, using relation \eqref{eq:Fidentity}, as well as the binomial theorem, we deduce that 
\begin{align*}
F_{\varepsilon,x}(u_{1},u_{2})
  =&(2\pi)^{-d}(\varepsilon+u_{1}^{2H})^{-\frac{d}{2}-1}(\varepsilon+u_{2}^{2H})^{-\frac{d}{2}-1}\mu(x,u_{1},u_{2})^{2}\\
  &\times \sum_{q=0}^{\infty}\frac{(\frac{d}{2})^{\overline{q+1}}}{(q+1)!}\left(\frac{\mu(x,u_{1},u_{2})^{2}}{(\varepsilon+u_{1}^{2H})(\varepsilon+u_{2}^{2H})}\right)^{q},
\end{align*} 
where $a^{\overline{n}}$ denotes the $n$-th raising factorial of $a$. Hence, using the fact that 
\begin{align*}
\frac{(\frac{d}{2})^{\overline{q+1}}}{(q+1)!}
  =\frac{(\frac{d}{2}+q)}{q+1}\frac{(\frac{d}{2})^{\overline{q}}}{q!}
	\leq \left(\frac{d}{2}+1\right)\frac{(\frac{d}{2})^{\overline{q}}}{q!}, 
\end{align*}
we deduce that 
\begin{align*}
F_{\varepsilon,x}(u_{1},u_{2})
  &\leq (2\pi)^{-d}\left(\frac{d}{2}+1\right)(1+u_{1}^{2H})^{-\frac{d}{2}}(1+u_{2}^{2H})^{-\frac{d}{2}}\frac{\mu(x,u_{1},u_{2})^{2}}{(\varepsilon+u_{1}^{2H})(\varepsilon+u_{2}^{2H})}\\
	&\times\sum_{q=0}^{\infty}\frac{(\frac{d}{2})^{\overline{q}}}{q!}\left(\frac{\mu(x,u_{1},u_{2})^{2}}{(\varepsilon+u_{1}^{2H})(\varepsilon+u_{2}^{2H})}\right)^{q},
\end{align*} 
which, by the binomial theorem, implies \eqref{ineq:FTheta}.

Due to relations \eqref{eq:covbasicchaos} and \eqref{eq:pdet}, the integrals 
\begin{align}\label{eq:GandF}
\int_{[0,T]^{3}}G_{\varepsilon}^{(q)}(x,u_{1},u_{2})dxdu_{1}du_{2} \ \ \ \ \text{ and }\ \ \ \int_{[0,T]^{3}}F_{\varepsilon}(x,u_{1},u_{2})dxdu_{1}du_{2}
\end{align}
will frequently appear throughout the paper, and their asymptotic behavior as $\varepsilon\rightarrow0$ will depend on the value Hurst parameter $H$. In order to simplify the study of such integrals, we introduce the following sets
\begin{align}\label{eq:Scdef}
\Sc_{1}
  &:=\{(x,u_{1},u_{2})\in\R_{+}^{3}\ |\ x+u_{2}-u_{1}\geq0, u_{1}-x\geq0\},\nonumber\\
\Sc_{2}
  &:=\{(x,u_{1},u_{2})\in\R_{+}^{3}\ |\ u_{1}-x-u_{2}\geq0\},\nonumber\\
\Sc_{3}
  &:=\{(x,u_{1},u_{2})\in\R_{+}^{3}\ |\ x-u_{1}\geq0\}.
\end{align}
The sets $\Sc_{1},\Sc_{2}$ and $\Sc_{3}$ satisfy $\R_{+}^{3}=\cup_{i=1}^{3}\Sc_{i}$, and $|\Sc_{i}\cap \Sc_{j}|=0$ for $i\neq j$. In addition, they satisfy the property that the integrals 
of    $G^{(q)}_\varepsilon$ and $F_\varepsilon$ over $[0,T]^3\cap \mathcal{S}_i$
are considerably simpler to handle than the integrals \eqref{eq:GandF}. This phenomenon arises from the local nondeterminism property of the factional Brownian motion (see Lemma \ref{lem:local_non_determinism}). 
\section{Behavior of the covariances of \texorpdfstring{$I_{T}^{\varepsilon}$}{TEXT}  and its chaotic components}\label{section:variances}
In this section we describe the behavior as $\varepsilon\rightarrow0$ of the covariance of $I_{T_{1}}^{\varepsilon}$ and $I_{T_{2}}^{\varepsilon}$, as well as the covariance of $I_{2q}(h_{2q,T_{1}}^{\varepsilon})$ and $I_{2q}(h_{2q,T_{2}}^{\varepsilon})$, for $0\leq T_{1}\leq T_{2}$.
\begin{Teorema}\label{teo:covchaos}
Let $T_{1},T_{2}\geq0$ be fixed. Then, if $\frac{3}{2d}<H<\frac{3}{4}$,
\begin{align*}
\lim_{\varepsilon\rightarrow0}\varepsilon^{d-\frac{3}{2H}}\E\left[I_{2q}(h_{2q,T_{1}}^{\varepsilon})I_{2q}(h_{2q,T_{2}}^{\varepsilon})\right]
  &=\sigma_{q}^{2}(T_{1}\wedge T_{2}),
\end{align*}
where 
\begin{align}\label{eq:sigqdef}
\sigma_{q}^{2}
  &:=2\beta_{q}\int_{\R_{+}^{3}}G_{1,x}^{(q)}(u_{1},u_{2})dxdu_{1}du_{2},
\end{align}
$\beta_{q}$ is defined by \eqref{eq:betaq} and $G_{1,x}^{(q)}(u_{1},u_{2})$ by \eqref{eq:Gdef}. Moreover, we have 
\begin{align}\label{eq:sumsigmaq}
\sum_{q=1}^{\infty}\sigma_{q}^2
  &=\sigma^2,
\end{align}
where $\sigma^{2}$ is a finite constant given by 
\begin{align}\label{eq:sigmadef}
\sigma^{2}
  &:=2\int_{\R_{+}^{3}}F_{1,x}(u_{1},u_{2})dxdu_{1}du_{2},
\end{align}
and $F_{1,x}(u_{1},u_{2})$ is defined in \eqref{eq:Fdef}.
\end{Teorema}
\begin{proof}
To prove the result, it suffices to show that for each $a<b<\alpha<\beta$, 
\begin{align}\label{eq:covfddchaosdisjoint}
\lim_{\varepsilon\rightarrow0}\varepsilon^{d-\frac{3}{2H}}\E\left[(I_{2q}(h_{2q,b}^{\varepsilon})-I_{2q}(h_{2q,a}^{\varepsilon}))(I_{2q}(h_{2q,\beta}^{\varepsilon})-I_{2q}(h_{2q,\alpha}^{\varepsilon}))\right]=0,
\end{align}
and
\begin{align}\label{eq:varfddchaos}
\lim_{\varepsilon\rightarrow0}\varepsilon^{d-\frac{3}{2H}}\E\left[(I_{2q}(h_{2q,b}^{\varepsilon})-I_{2q}(h_{2q,a}^{\varepsilon}))^2\right]= \sigma_{q}^{2}(b-a).
\end{align}
First we prove \eqref{eq:covfddchaosdisjoint}. Set
\[
\Phi^\varepsilon =\E\left[(I_{2q}(h_{2q,b}^{\varepsilon})-I_{2q}(h_{2q,a}^{\varepsilon}))(I_{2q}(h_{2q,\beta}^{\varepsilon})-I_{2q}(h_{2q,\alpha}^{\varepsilon}))\right].
\]
Define the set $\mathcal{K}_{T_{1},T_{2}}$ by \eqref{eq:Sdef}, and $\gamma:=\frac{\alpha-b}{2}>0$. We can easily check that for every $(s_{1},t_{1})\in \mathcal{K}_{a,b}$, and $(s_{2},t_{2})\in \mathcal{K}_{\alpha,\beta}$, it holds that either $t_{2}-s_{2}>\gamma$, or $s_{2}-s_{1}\geq\gamma$, and hence, by taking $T_{1}=a$, $T_{2}=b$, $\widetilde{T}_{1}=\alpha$, $\widetilde{T}_{2}=\beta$ in \eqref{eq:covbasicchaos}, we get
\begin{eqnarray}   \nonumber
\Abs{ \Phi^\varepsilon}
 &\le &
  \beta_{q}\int_{[0,\beta]^{4}}\Indi{(\gamma,\infty)}(t_{2}-s_{2})G_{\varepsilon,s_{2}-s_{1}}^{(q)}(t_{1}-s_{1},t_{2}-s_{2})ds_{1}ds_{2}dt_{1}dt_{2}\\
	&& +\beta_{q}\int_{[0,\beta]^{4}}\Indi{(\gamma,\infty)}(s_{2}-s_{1})G_{\varepsilon,s_{2}-s_{1}}^{(q)}(t_{1}-s_{1},t_{2}-s_{2})ds_{1}ds_{2}dt_{1}dt_{2}. \label{eq:incchaosabs}
\end{eqnarray}
Changing the coordinates $(s_{1},s_{2},t_{1},t_{2})$ by $(s:=s_{1},x:=s_{2}-s_{1},u_{1}:=t_{1}-s_{1},u_{2}:=t_{2}-s_{2})$ for $s_{2}\geq s_{2}$, and by $(s:=s_{2},x:=s_{1}-s_{2},u_{1}:=t_{1}-s_{1},u_{2}:=t_{2}-s_{2})$ for $s_{2}\leq s_{1}$, in \eqref{eq:incchaosabs}, using the fact that $G_{\varepsilon,-x}^{(q)}(u_{1},u_{2})=G_{\varepsilon,x}^{(q)}(u_{2},u_{1})$, and integrating the $s_{1}$ variable, we can prove that  
\[
\Abs{ \Phi^\varepsilon}
   \leq  \beta_{q}\beta\int_{[0,\beta]^{3}} \left( \Indi{(\gamma,\infty)}(u_{1})+\Indi{(\gamma,\infty)}(u_{2}) +\Indi{(\gamma,\infty)}(x) \right) G_{\varepsilon,x}^{(q)}(u_{1},u_{2})dxdu_{1}du_{2}.
\]
Next, changing the coordinates $(x,u_{1},u_{2})$ by $(\varepsilon^{-\frac{1}{2H}}x,\varepsilon^{-\frac{1}{2H}}u_{1},\varepsilon^{-\frac{1}{2H}}u_{2})$, and using the fact that $G_{\varepsilon,\varepsilon^{\frac{1}{2H}}x}^{(q)}(\varepsilon^{\frac{1}{2H}}u_{1},\varepsilon^{\frac{1}{2H}}u_{2})=\varepsilon^{-d}G_{1,x}^{(q)}(u_{1},u_{2})$, we get
\begin{align*}
\Abs{ \Phi^\varepsilon}
  \leq & \varepsilon^{\frac{3}{2H}-d}\beta_{q} \beta\int_{[0,\varepsilon^{-\frac{1}{2H}}\beta]^{3}}
  \left(  \Indi{(\gamma,\infty)}(\varepsilon^{\frac{1}{2H}}u_{1})+\Indi{(\gamma,\infty)}(\varepsilon^{\frac{1}{2H}}u_{2})+\Indi{(\gamma,\infty)}(\varepsilon^{\frac{1}{2H}}x)\right)\\
  &\times
  G_{1,x}^{(q)}(u_{1},u_{2})dxdu_{1}du_{2}.
\end{align*}
  Since $\gamma>0$, the arguments in the previous integrals converge to zero pointwise, and are dominated by the function $3\beta_{q}\beta G_{1,	x}^{(q)}(u_{1},u_{2})$, which is integrable by Lemma \ref{lem:finiteFGintegral} due to the condition $\frac{3}{2d}<H<\frac{3}{4}$. Hence, by the dominated convergence theorem, 
\begin{align*}
\lim_{\varepsilon\rightarrow0}\varepsilon^{d-\frac{3}{2H}}\Abs{ \Phi^\varepsilon}
  &=0,
\end{align*}
as required. Next we prove \eqref{eq:varfddchaos}. By taking $T_{1}=\widetilde{T}_{1}=a$, and $T_{2}=\widetilde{T}_{2}=b$ in \eqref{eq:covbasicchaos}, we deduce that 
\begin{align*}
\E\left[(I_{2q}(h_{2q,b}^{\varepsilon})-I_{2q}(h_{2q,a}^{\varepsilon}))^2\right]
  &=2\beta_{q}\int_{[0, b]^{4}}\Indi{\{s_{1}\leq s_{2}\}}\Indi{\mathcal{K}_{a,b}}(s_{1},t_{1})\Indi{\mathcal{K}_{a,b}}(s_{2},t_{2})\\
	&\times G_{\varepsilon,s_{2}-s_{1}}^{(q)}(t_{1}-s_{1},t_{2}-s_{2})ds_{1}ds_{2}dt_{1}dt_{2}.
\end{align*}
Changing the coordinates $(s_{1},s_{2},t_{1},t_{2})$ by $(s_{1},x:=s_{2}-s_{1},u_{1}:=t_{1}-s_{1},u_{2}:=t_{2}-s_{2})$, we get
\begin{multline}\label{eq:InckerG}
\E\left[(I_{2q}(h_{2q,b}^{\varepsilon})-I_{2q}(h_{2q,a}^{\varepsilon}))^2\right]\\
\begin{aligned}
  &=2\beta_{q}\int_{[0,b]^{4}}\Indi{\mathcal{K}_{a,b}}(s_{1},s_{1}+u_{1})\Indi{\mathcal{K}_{a,b}}(s_{1}+x,s_{1}+x+u_{2}) G_{\varepsilon,x}^{(q)}(u_{1},u_{2})ds_{1}dxdu_{1}du_{2}\\
	&=2\beta_{q}\int_{[0,b]^{3}}\int_{(a-u_{1})_{+}\vee (a-x-u_{2})_{+}}^{(b-u_{1})_{+}\wedge(b-x-u_{2})_{+}}ds_{1}G_{\varepsilon,x}^{(q)}(u_{1},u_{2})dxdu_{1}du_{2}.
\end{aligned}
\end{multline}
Notice that $G_{\varepsilon,\varepsilon^{\frac{1}{2H}}x}^{(q)}(\varepsilon^{\frac{1}{2H}}u_{1},\varepsilon^{\frac{1}{2H}}u_{2})=\varepsilon^{-d}G_{1,x}(u_{1},u_{2})$. Therefore, integrating the variable $s_{1}$, and changing the coordinates $(x,u_{1},u_{2})$ by $(\varepsilon^{-\frac{1}{2H}}x,\varepsilon^{-\frac{1}{2H}}u_{1},\varepsilon^{-\frac{1}{2H}}u_{2})$ in \eqref{eq:InckerG}, we conclude that 
\begin{eqnarray}   \nonumber    \label{eq:VarIincfinalchaos}
 &&\varepsilon^{d-\frac{3}{2H}}\E\left[(I_{2q}(h_{2q,b}^{\varepsilon})-I_{2q}(h_{2q,a}^{\varepsilon}))^2\right]
  = 2\beta_{q}\int_{[0,\varepsilon^{-\frac{1}{2H}}b]^{3 }}G_{1,x}^{(q)}(u_{1},u_{2})\\  \nonumber
	&& \qquad \qquad  \qquad \qquad \qquad   \times  \Big[  (b-\varepsilon^{\frac{1}{2H}}u_{1})_{+}\wedge (b-\varepsilon^{\frac{1}{2H}}(x+u_{2}))_{+} \\
	&&  \qquad \qquad  \qquad  \qquad \qquad  \qquad -(a-\varepsilon^{\frac{1}{2H}}u_{1})_{+}\vee(a-\varepsilon^{\frac{1}{2H}}(x+u_{2}))_{+} \Big]dxdu_{1}du_{2}.
\end{eqnarray}	
The integrand in \eqref{eq:VarIincfinalchaos} converges increasingly to $2(b-a)G_{1,x}^{(q)}(u_{1},u_{2})$ as $\varepsilon\rightarrow0$, which is integrable by Lemma \ref{lem:finiteFGintegral}. Identity \eqref{eq:varfddchaos} then follows by applying the dominated convergence theorem in \eqref{eq:VarIincfinalchaos}. 

Relation \eqref{eq:sumsigmaq} is obtained by integrating both sides of relation \eqref{eq:Fidentitychaos} over the variables $x,u_{1},u_{2}\geq0$, for $\varepsilon=1$, and then using the monotone convergence theorem. The constant $\sigma^2$ is finite  by Lemma \ref{lem:finiteFGintegral}. 
The proof is now complete.
\end{proof}

In order to determine the behavior of the covariances of $I_{T}^{\varepsilon}$ for the case $H=\frac{3}{4}$, we will first prove that the second chaotic component $I_{2}(h_{2,T}^{\varepsilon})$ characterizes the asymptotic behavior of $I_{T}^{\varepsilon}-\E\left[I_{T}^{\varepsilon}\right]$ as $\varepsilon\rightarrow\infty$, for every $H\geq\frac{3}{4}$. 

We start by showing that, after a suitable rescaling, the sequence $I_{2}(h_{2,T}^{\varepsilon})$ approximates $I_{T}^{\varepsilon}-\E\left[I_{T}^{\varepsilon}\right]$ in $L^{2}(\Omega)$ for  $H>\frac{3}{4}$. This result will be latter used in the proof of Theorem \ref{Teo:convergenceHermite}. 
\begin{Lema}\label{lem:chaos4tozero}
Let $\frac{3}{4}<H<1$ be fixed. Then, 
\begin{align*}
\lim_{\varepsilon\rightarrow0}\varepsilon^{\frac{d}{2}-\frac{3}{2H}+1}\Norm{I_{T}^{\varepsilon}-\E\left[I_{T}^{\varepsilon}\right]-J_{2}(I_{T}^{\varepsilon})}_{L^{2}(\Omega)}
  &=0.
\end{align*}
\end{Lema}
\begin{proof}
For $T>0$ fixed, define the quantity 
\begin{align*}
Q_{\varepsilon}
  &:=\Norm{I_{T}^{\varepsilon}-\E\left[I_{T}^{\varepsilon}\right]-J_{2}(I_{T}^{\varepsilon})}_{L^{2}(\Omega)}^2.
\end{align*}
From the chaos decomposition \eqref{eq:chaos2}, we get 
\begin{align}\label{eq:auxfinal}
Q_{\varepsilon}
  &= \E\left[(I_{T}^{\varepsilon})^2\right]
	-\E\left[I_{T}^{\varepsilon}\right]^2
	-\E\left[J_{2}(I_{T}^{\varepsilon})^2\right]\nonumber\\
	&= \E\left[(I_{T}^{\varepsilon})^2\right]
	-\E\left[I_{T}^{\varepsilon}\right]^2
	-2\Norm{h_{2,T}^{\varepsilon}}_{(\Hg^{d})^{\otimes 2}}^2\nonumber\\
		&= \E\left[(I_{T}^{\varepsilon})^2\right]
	-\E\left[I_{T}^{\varepsilon}\right]^2
	-2\Norm{\int_{T\Rc}f_{2,s,t}^{\varepsilon}dsdt}_{(\Hg^{d})^{\otimes 2}}^2.
\end{align}
By \eqref{eq:EI} and \eqref{eq:Im2}, the first two terms in the right-hand side of the previous identity can be written as  
\begin{align}\label{eq:Qdec1}
\E\left[(I_{T}^{\varepsilon})^2\right]
	  &=(2\pi)^{-d}\int_{(T\Rc)^2}\Theta_{\varepsilon}(s_{2}-s_{1},t_{1}-s_{1},t_{2}-s_{2})^{-\frac{d}{2}}ds_{1}ds_{2}dt_{1}dt_{2},
\end{align}
and 
\begin{align}\label{eq:Qdec2}
\E\left[I_{T}^{\varepsilon}\right]^2
	  &=(2\pi)^{-d}\int_{(T\Rc)^2}G_{\varepsilon,s_{2}-s_{1}}^{(0)}(t_{1}-s_{1},t_{2}-s_{2})ds_{1}ds_{2}dt_{1}dt_{2},
\end{align}
where $G_{\varepsilon,x}^{(q)}(u_{1},u_{2})$ and $\Theta_{\varepsilon}(x,u_{1},u_{2})$ are given by \eqref{eq:Gdef} and \eqref{eq:Ldef}, respectively. To handle the third term in \eqref{eq:auxfinal}, recall that the constants $\alpha_{q}$ are given by \eqref{eq:alphaqdef}, and notice that $\alpha_{1}=2d$. Hence, from \eqref{eq:innerproductG}, we deduce that
\begin{align}\label{eq:Qdec3}
\Norm{\int_{T\Rc }f_{2,s,t}^{\varepsilon}dsdt}_{(\Hg^{d})^{\otimes 2}}^2
  &=\frac{d(2\pi)^{-d}}{4}\int_{(T\Rc)^2}G_{\varepsilon,s_{2}-s_{1}}^{(1)}(t_{1}-s_{1},t_{2}-s_{2})ds_{1}ds_{2}dt_{1}dt_{2}.
\end{align}
From equations \eqref{eq:auxfinal}-\eqref{eq:Qdec3}, we conclude that 
\begin{align}\label{eq:Qrewritten}
Q_{\varepsilon}
	&=(2\pi)^{-d}\int_{(T\Rc)^2}\bigg(\Theta_{\varepsilon}(s_{2}-s_{1},t_{1}-s_{1},t_{2}-s_{2})^{-\frac{d}{2}}\nonumber\\
	&-G_{\varepsilon,s_{2}-s_{1}}^{(0)}(t_{1}-s_{1},t_{2}-s_{2})-\frac{d}{2}G_{\varepsilon,s_{2}-s_{1}}^{(1)}(t_{1}-s_{1},t_{2}-s_{2})\bigg)ds_{1}ds_{2}dt_{1}dt_{2}.
\end{align}
The integrand appearing in the right-hand side is positive. Indeed, if we define 
\begin{align*}
\rho_{\varepsilon}(x,u_{1},u_{2})
  &:=\mu(x,u_{1},u_{2})^2(\varepsilon+u_{1}^{2H})^{-1}(\varepsilon+u_{2}^{2H})^{-1},
\end{align*}
then,  applying relations \eqref{eq:Gdef}, \eqref{eq:Ldef} we obtain
\begin{align} 
\Theta_{\varepsilon}(x,u_{1},u_{2})^{-\frac{d}{2}}   \notag
	-&G_{\varepsilon,x}^{(0)}(u_{1},u_{2})-\frac{d}{2}G_{\varepsilon,x}^{(1)}(u_{1},u_{1})
	= 2(2\pi)^{-d}(\varepsilon+u_{1}^{2H})^{-\frac{d}{2}}(\varepsilon+u_{2}^{2H})^{-\frac{d}{2}}\\
	&\times\bigg((1-\rho_{\varepsilon}(x,u_{1},u_{2}))^{-\frac{d}{2}}-1-\frac{d}{2}\rho_{\varepsilon}(x,u_{1},u_{2})\bigg) \label{ineq:Taylorchaosgeq4}
\end{align}
and the right-hand side of the previous identity is positive by the binomial theorem. As a consequence, by changing the coordinates $(s_{1},s_{2},t_{1},t_{2})$ by $(s_{1},x:=s_{2}-s_{1},u_{1}:=t_{1}-s_{1},u_{2}:=t_{2}-s_{2})$, and integrating the variable $s_{1}$ in \eqref{eq:Qrewritten}, we get
\begin{align*}
Q_{\varepsilon}
  &\leq2(2\pi)^{-d}T\int_{[0,T]^{3}}\bigg(\Theta_{\varepsilon}(x,u_{1},u_{2})^{-\frac{d}{2}}
	-G_{\varepsilon,x}^{(0)}(u_{1},u_{2})-\frac{d}{2}G_{\varepsilon,x}^{(1)}(u_{1},u_{2})\bigg)dxdu_{1}du_{2}.
\end{align*}
In addition, by the binomial theorem, we have that for every $0<y<1$,
\begin{align*}
(1-y)^{-\frac{d}{2}}-1-\frac{d}{2}y
  &=\sum_{q=2}^{\infty}(-1)^{q}\Comb{-\frac{d}{2}\\q}y^{q}=y^{2}\sum_{q=0}^{\infty}\frac{(\frac{d}{2})^{\overline{q+2}}}{(q+2)!}y^{q},
\end{align*}
where $(x)^{\overline{q}}$ denotes the raising factorial $(x)^{\overline{q}}:=x(x+1)\dots(x+q-1)$. Hence, by \eqref{ineq:Taylorchaosgeq4}, 
\begin{align}\label{ineq:L2chaosgeq4}
Q_{\varepsilon}
	&\leq 2(2\pi)^{-d}T\int_{[0,T]^{3}}(\varepsilon+u_{1}^{2H})^{-\frac{d}{2}}(\varepsilon+u_{2}^{2H})^{-\frac{d}{2}}\nonumber\\
	&\times\rho_{\varepsilon}(x,u_{1},u_{2})^{2}\sum_{q=0}^{\infty}\frac{(\frac{d}{2})^{\overline{q+2}}}{(q+2)!}\rho_{\varepsilon}(x,u_{1},u_{2})^qdxdu_{1}du_{2}.
\end{align}
 Since 
\begin{align*}
\frac{(\frac{d}{2})^{\overline{q+2}}}{(q+2)!}
  &=\frac{(\frac{d}{2})^{\overline{q}}}{q!}\frac{(\frac{d}{2}+q)(\frac{d}{2}+q+1)}{(q+1)(q+2)}
  \leq \left(\frac{d}{2}+1\right)^2\frac{(\frac{d}{2})^{\overline{q}}}{q!},
\end{align*}
then, by \eqref{ineq:L2chaosgeq4}, 
\begin{align*}
Q_{\varepsilon}
	&\leq 2(2\pi)^{-d}T\left(\frac{d}{2}+1\right)^2\int_{[0,T]^{3}}(\varepsilon+u_{1}^{2H})^{-\frac{d}{2}}(\varepsilon+u_{2}^{2H})^{-\frac{d}{2}}\nonumber\\
	&\times\rho_{\varepsilon}(x,u_{1},u_{2})^{2}\sum_{q=0}^{\infty}\frac{(\frac{d}{2})^{\overline{q}}}{q!}\rho_{\varepsilon}(x,u_{1},u_{2})^qdxdu_{1}du_{2},
\end{align*}
which, by the binomial theorem, implies that there exists a constant $C>0$ only depending on $T$ and $d$, such that
\begin{align}\label{eq:muThetaprev}
Q_{\varepsilon}
	&\leq C\int_{[0,T]^{3}}\frac{\mu(x,u_{1},u_{2})^4}{(\varepsilon+u_{1}^{2H})^2(\varepsilon+u_{2}^{2H})^2}\Theta_{\varepsilon}(x,u_{1},u_{2})^{-\frac{d}{2}}dxdu_{1}du_{2}.
\end{align}
Hence, to prove the lemma it suffices to show that 
\begin{align}\label{eq:muTheta}
\lim_{\varepsilon\rightarrow0}\varepsilon^{d-\frac{3}{H}+2}\int_{[0,T]^{3}}\Psi_{\varepsilon}(x,u_{1},u_{2})dxdu_{1}du_{2}=0,
\end{align}
where 
\begin{align}\label{eq:psideflogtheta}
\Psi_{\varepsilon}(x,u_{1},u_{2})
  &:=\frac{\mu(x,u_{1},u_{2})^4}{(\varepsilon+u_{1}^{2H})^2(\varepsilon+u_{2}^{2H})^2}\Theta_{\varepsilon}(x,u_{1},u_{2})^{-\frac{d}{2}}.
\end{align}
In order to prove \eqref{eq:muTheta}, we proceed as follows. First we decompose the domain of integration of \eqref{eq:muTheta} as $[0,T]^{3}=\widetilde{\Sc}_{1}\cup\widetilde{\Sc}_{2}\cup\widetilde{\Sc}_{3}$, where 
\begin{align}\label{eq:Stildedef}
\widetilde{\Sc}_{1}
  &:=\{(x,u_{1},u_{2})\in[0,T]^{3}\ |\ x+u_{2}-u_{1}\geq0, u_{1}-x\geq0\},\nonumber\\
\widetilde{\Sc}_{2}
  &:=\{(x,u_{1},u_{2})\in[0,T]^{3}\ |\ u_{1}-x-u_{2}\geq0\},\nonumber\\
\widetilde{\Sc}_{3}
  &:=\{(x,u_{1},u_{2})\in[0,T]^{3}\ |\ x-u_{1}\geq0\}.
\end{align}
Then, it suffices to show that 
\begin{align}\label{conv:chaos4decomp}
\lim_{\varepsilon\rightarrow0}\varepsilon^{d-\frac{3}{H}+2}\int_{\widetilde{\Sc}_{i}}\Psi_{\varepsilon}(x,u_{1},u_{2})dxdu_{1}du_{2}
  &=0,
\end{align} 
for $i=1,2,3$.\\~\\
First prove \eqref{conv:chaos4decomp} in the cases $i=1,2$. Changing the coordinates $(x,u_{1},u_{2})$ by $(\varepsilon^{-\frac{1}{2H}}x,\varepsilon^{-\frac{1}{2H}}u_{1},\varepsilon^{-\frac{1}{2H}}u_{2})$, and using the fact that 
$\Psi_{\varepsilon}(\varepsilon^{\frac{1}{2H}}x,\varepsilon^{\frac{1}{2H}}u_{1},\varepsilon^{\frac{1}{2H}}u_{2})=\varepsilon^{-d}\Psi_{1}(x,u_{1},u_{2})$,
we get 
\begin{align*}
\varepsilon^{d-\frac{3}{H}+2}\int_{\widetilde{\Sc}_{i}}\Psi_{\varepsilon}(x,u_{1},u_{2})dxdu_{1}du_{2}
  &\leq\varepsilon^{2-\frac{3}{2H}}\int_{\Sc_{i}}\Psi_{1}(x,u_{1},u_{2})dxdu_{1}du_{2},
\end{align*}
where the sets $\Sc_{i}$ are defined by \eqref{eq:Scdef}.  Therefore, using the inequality $\mu(x,u_{1},u_{2})^{2}\leq (u_{1}u_{2})^{2H}$, we obtain 
\begin{align}\label{conv:chaos4R1p}
\varepsilon^{d-\frac{3}{H}+2}\int_{\widetilde{\Sc}_{i}}\Psi_{\varepsilon}(x,u_{1},u_{2})dxdu_{1}du_{2}
  &\leq\varepsilon^{2-\frac{3}{2H}}\int_{\Sc_{i}}\frac{\mu(x,u_{1},u_{2})^2}{(u_{1}u_{2})^{2H}}\Theta_{1}(x,u_{1},u_{2})^{-\frac{d}{2}}dxdu_{1}du_{2}.
\end{align}
The integral appearing in the right-hand side of the previous inequality is finite by Lemma \ref{lem:inegraltightness} (see equation \eqref{eq:integralexpressiontight2p} for $p=2$ and $i=1,2$).
Relation \eqref{conv:chaos4decomp} for $i=1,2$ is then obtained by taking $\varepsilon\rightarrow0$ in \eqref{conv:chaos4R1p}. 

It then remains to prove \eqref{conv:chaos4decomp} for $i=3.$ Changing the coordinates $(x,u_{1},u_{2})$ by $(a:=u_{1},b:=x-u_{1},c:=u_{2})$, we get 
\begin{align}\label{conv:chaos4R3v2}
\int_{\widetilde{\Sc}_{3}}\Psi_{\varepsilon}(x,u_{1},u_{2})dxdu_{1}du_{2}
  &\leq\int_{[0,T]^3}\Psi_{\varepsilon}(a+b,a,c)dadbdc.
\end{align}
We bound the right-hand side of the previous inequality as follows. First we write
\begin{align}\label{eq:muintegralrepR3}
\mu(a+b,a,c)
  &=\frac{1}{2}((a+b+c)^{2H}+b^{2H}-(b+c)^{2H}-(a+b)^{2H})\nonumber\\
	&=H(2H-1)ac\int_{[0,1]^2}(b+av_{1}+cv_{2})^{2H-2}dv_{1}dv_{2}.
\end{align}
Notice that if $a>c$, then $b+av_{1}+cv_{2}\geq v_{1}(b+a)\geq v_{1}(b+\frac{a}{2}+\frac{c}{2})$, and if $c>a$, then $b+av_{1}+cv_{2}\geq v_{2}(b+c)\geq v_{2}(b+\frac{a}{2}+\frac{c}{2})$. Therefore, since $H>\frac{3}{4}$, by \eqref{eq:muintegralrepR3} we deduce that there exists a constant $K>0$, such that
\begin{align}\label{eq:muboundR3}
\mu(a+b,a,c)
  &\leq Kac(a+b+c)^{2H-2}.
\end{align}
On the other hand, if $\Sigma$ denotes the covariance matrix of $(B_{a},B_{a+b+c}-B_{a+b})$, we can write
$$\Theta_{\varepsilon}(a+b,a,c)=\varepsilon^2+\varepsilon(a^{2H}+c^{2H})+|\Sigma|.$$ 
As a consequence, by part $(3)$ of Lemma \ref{lem:local_non_determinism}, we deduce that $\Theta_{\varepsilon}(a+b,a,c)\geq \varepsilon^{2}+\delta (ac)^{2H}$ for some constant $\delta\in(0,1)$. Hence, by \eqref{eq:psideflogtheta} and \eqref{eq:muboundR3}, that there exists a constant $C>0$, such that
\begin{align}\label{ineq:PsiR3v4}
\Psi_{\varepsilon}(a+b,a,c)
  &\leq C(ac)^{4-4H}(a+b+c)^{8H-8}(\varepsilon^{2}+(ac)^{2H})^{-\frac{d}{2}}.
\end{align}
Next we bound the right-hand side of \eqref{ineq:PsiR3v4} by using Young's inequality. Since $H>\frac{3}{4}$ and $Hd>\frac{3}{2}$, then 
\begin{align}\label{ineq:gammaapprox}
0<\frac{3-2H}{Hd}<\frac{3}{2Hd}<1.
\end{align}
Using the relation \eqref{ineq:gammaapprox}, as well as the fact that $\frac{3}{4}<H<1$, we deduce that there exists a constant $y>0$, such that 
\begin{align}
4H-4+4Hdy<0\label{ineq:expauxR31},\\
4H-3-4Hdy>0\label{ineq:expauxR32},\\
\frac{3-2H}{Hd}+y<1.\label{ineq:expauxR3}
\end{align}
By \eqref{ineq:expauxR3}, the constant $\gamma:=\frac{3-2H}{Hd}+y$ belongs to $(0,1)$, and hence, by Young's inequality, we have 
\begin{align}\label{ineq:PsiR3v2}
(1-\gamma)\varepsilon^{2}+\gamma(ac)^{2H}
  &\geq \varepsilon^{2(1-\gamma)}(ac)^{2H\gamma}.
\end{align}
In addition, by \eqref{ineq:expauxR31}, we have 
\begin{align}
(a+b+c)^{8H-8}
  &= (a+b+c)^{4H-4-4Hdy}(a+b+c)^{4H-4+4Hdy}\nonumber\\
	&\leq b^{4H-4-4Hdy}(a+c)^{4H-4+4Hdy}\nonumber\\
	&\leq b^{4H-4-4Hdy}(2\sqrt{ac})^{4H-4+4Hdy},\label{ineq:varepsilonbac2}
\end{align}
where the last inequality follows from the arithmetic mean-geometric mean inequality. Hence, by \eqref{ineq:PsiR3v4}, \eqref{ineq:PsiR3v2} and \eqref{ineq:varepsilonbac2}, we obtain
\begin{align}\label{ineq:PsiR3v5}
\varepsilon^{d-\frac{3}{H}+2}\int_{[0,T]^{3}}\Psi_{\varepsilon}(a+b,a,c)
  &\leq \varepsilon^{d-\frac{3}{H}+2-d(1-\gamma)}C\int_{[0,T]^3}b^{4H-4-4Hdy}(ac)^{2-2H+2Hdy-Hd\gamma}dadbdc\nonumber\\
	&= \varepsilon^{dy}C\int_{[0,T]^3}b^{4H-4-4Hdy}(ac)^{-1+Hdy}dadbdc.
\end{align}
The integral in the right-hand side is finite by \eqref{ineq:expauxR32}. Relation \eqref{conv:chaos4decomp} for $i=3$ then follows from \eqref{conv:chaos4R3v2} and \eqref{ineq:PsiR3v5}.
\end{proof}

\noindent The next result extends Lemma \ref{lem:chaos4tozero} to the case $H=\frac{3}{4}$.
\begin{Lema}\label{lem:chaos4tozero2}
Let $d\geq 3$ be fixed. Then, if $H=\frac{3}{4}$,
\begin{align}\label{eq:chaos2critical}
\lim_{\varepsilon\rightarrow0}\frac{\varepsilon^{\frac{d}{2}-1}}{\sqrt{\log(1/\varepsilon)}}\Norm{I_{T}^{\varepsilon}-\E\left[I_{T}^{\varepsilon}\right]-J_{2}(I_{T}^{\varepsilon})}_{L^{2}(\Omega)}
  &=0.
\end{align}
\end{Lema}
\begin{proof}
For $T>0$ fixed, define the quantity 
\begin{align*}
Q_{\varepsilon}
  &:=\Norm{I_{T}^{\varepsilon}-\E\left[I_{T}^{\varepsilon}\right]-J_{2}(I_{T}^{\varepsilon})}_{L^{2}(\Omega)}^2.
\end{align*}
As in the proof of equation \eqref{eq:muThetaprev} in Lemma \ref{lem:chaos4tozero}, we can show that there exists a constant $C>0$ such that 
\begin{align}\label{ineq:Qpsi}
Q_{\varepsilon}
  &\leq C\int_{[0,T]^{3}}\Psi_{\varepsilon}(x,u_{1},u_{2})dxdu_{1}du_{2},
\end{align}
where 
\begin{align}\label{eq:psilogdef1}
\Psi_{\varepsilon}(x,u_{1},u_{2})
  &:=\frac{\mu(x,u_{1},u_{2})^4}{(\varepsilon+u_{1}^{\frac{3}{2}})^2(\varepsilon+u_{2}^{\frac{3}{2}})^2}\Theta_{\varepsilon}(x,u_{1},u_{2})^{-\frac{d}{2}}.
\end{align}
Hence, by splitting the domain of integration in \eqref{ineq:Qpsi} as $[0,T]^3=\bigcup_{i=1}^{3}\widetilde{\Sc}_{i}$, where the sets $\widetilde{\Sc}_{i}$ are defined by \eqref{eq:Stildedef}, we deduce that the relation \eqref{eq:chaos2critical} holds, provided that
\begin{align}\label{conv:chaos4decomp2}
\lim_{\varepsilon\rightarrow0}\frac{\varepsilon^{d-2}}{\log(1/\varepsilon)}\int_{\widetilde{\Sc}_{i}}\Psi_{\varepsilon}(x,u_{1},u_{2})dxdu_{1}du_{2}
  &=0,
\end{align}  
for $i=1,2,3$. To prove \eqref{conv:chaos4decomp2} for $i=1,2$, we change the coordinates $(x,u_{1},u_{2})$ by $(\varepsilon^{-\frac{2}{3}}x,\varepsilon^{-\frac{2}{3}}u_{1},\varepsilon^{-\frac{2}{3}}u_{2})$ and use the fact that $\Psi_{\varepsilon}(\varepsilon^{\frac{2}{3}}x,\varepsilon^{\frac{2}{3}} u_{1},\varepsilon^{\frac{2}{3}} u_{2})=\varepsilon^{-d}\Psi_{1}(x,u_{2},u_{2})$, in order to get 
\begin{align}\label{eq:Psilogthetaprime}
\frac{\varepsilon^{d-2}}{\log(1/\varepsilon)}\int_{\widetilde{\Sc_{i}}}\Psi_{\varepsilon}(x,u_{1},u_{2})dxdu_{1}du_{2}
  &\leq\frac{1}{\log(1/\varepsilon)}\int_{\Sc_{i}}\Psi_{1}(x,u_{1},u_{2})dxdu_{1}du_{2},
\end{align}
where the sets $\Sc_{i}$ are defined by \eqref{eq:Scdef}. As a consequence, by applying the inequality $\mu(x,u_{1},u_{2})^2\leq(u_{1}u_{2})^{\frac{3}{2}}$, we get 
\begin{align}\label{eq:Psilogtheta}
\frac{\varepsilon^{d-2}}{\log(1/\varepsilon)}\int_{\widetilde{\Sc_{i}}}\Psi_{\varepsilon}(x,u_{1},u_{2})dxdu_{1}du_{2}
  &\leq\frac{1}{\log(1/\varepsilon)}\int_{\Sc_{i}}\frac{\mu(x,u_{1},u_{2})^2}{(u_{1}u_{2})^{\frac{3}{2}}}\Theta_{1}(x,u_{1},u_{2})^{-\frac{d}{2}}dxdu_{1}du_{2}.
\end{align}
The integral appearing the right-hand side of the previous inequality is finite for $i=1,2$ by Lemma \ref{lem:inegraltightness} (see equation \eqref{eq:integralexpressiontight2p} for $p=2$). Relation \eqref{conv:chaos4decomp2} for $i=1,2$ is then obtained by taking $\varepsilon\rightarrow0$ in \eqref{eq:Psilogtheta}.

It then suffices to handle the case $i=3$.  Define the function $K(x,u_{1},u_{2})$ by 
\begin{align}\label{eq:Kdef}
K(x,u_{1},u_{2})
  &:=\frac{\mu(x,u_{1},u_{2})^4}{(u_{1}u_{2})^{3}}\Theta_{1}(x,u_{1},u_{2})^{-\frac{d}{2}}.
\end{align} 
Notice that 
\begin{align}\label{eq:PsiKrel}
\frac{1}{\log(1/\varepsilon)}\int_{\Sc_{3}}\Psi_{1}(x,u_{1},u_{2})dxdu_{1}du_{2}
  &\leq\frac{1}{\log(1/\varepsilon)}\int_{\Sc_{3}}K(x,u_{1},u_{2})dxdu_{1}du_{2}.
\end{align}
Using the representation
\begin{align*}
\mu(a+b,a,c)
  &=H(2H-1)ac\int_{[0,1]^2}(b+a\xi+c\eta)^{2H-2}d\xi d\eta,
\end{align*}
we get 
\begin{align*}
\mu(a+b,a,c)
  &\leq \frac{3ac}{8}\int_{[0,1]^2}(b\xi\eta +a\xi\eta+c\xi\eta)^{-\frac{1}{2}}d\xi d\eta= \frac{3ac}{2}(a+b+c)^{-\frac{1}{2}}.
\end{align*}
As a consequence, 
\begin{align*}
K(a+b,a,c)
  &\leq \frac{3^4}{2^4}ac(a+b+c)^{-2}\Theta_{1}(a+b,a,c)^{-\frac{d}{2}}.
\end{align*}
Notice that $\Theta_{1}(a+b,a,c)=1+a^{\frac{3}{2}}+c^{\frac{3}{2}}+|\Sigma|$, where $\Sigma$ denotes the covariance matrix of $(B_{a},B_{a+b+c}-B_{a+b})$. Therefore, by part (3) of Lemma \ref{lem:local_non_determinism}, we deduce that 
$$\Theta_{1}(a+b,a,c)\geq 1+a^{\frac{3}{2}}+c^{\frac{3}{2}}+\delta (ac)^{\frac{3}{2}},$$ 
for some constant $\delta\in(0,1)$. From here, it follows that there exists a constant $C>0$, such that
\begin{align*}
K(a+b,a,c)
  &\leq Cac(a+b+c)^{-2}\left(1+a^{\frac{3}{2}}+c^{\frac{3}{2}}+ a^{\frac{3}{2}}c^{\frac{3}{2}}\right)^{-\frac{d}{2}}.
\end{align*}
From here it follows that there exists a constant $C>0$ such that the following inequalities hold
\begin{align*}
K(a+b,a,c)
&\leq Cac^{-1}\left(1+c^{\frac{3}{2}}+ a^{\frac{3}{2}}c^{\frac{3}{2}}\right)^{-\frac{d}{2}}~~~~~~~~~~~~~\text{\ \ \ \ \ \ \ \ \ \ \ \ \ \ \ \ \ if }~~a\leq b\leq c,\\
K(a+b,a,c)
&\leq Ca^{-1}c\left(1+a^{\frac{3}{2}}+ a^{\frac{3}{2}}c^{\frac{3}{2}}\right)^{-\frac{d}{2}}~~~~~~~~~~~~~\text{\ \ \ \ \ \ \ \ \ \ \ \ \ \ \ \ \ if }~~c\leq b\leq a,\\
K(a+b,a,c)
&\leq Cacb^{-2}\left(1+(a\vee c)^{\frac{3}{2}}+ a^{\frac{3}{2}}c^{\frac{3}{2}}\right)^{-\frac{d}{2}}~~~~~~~~~~~~~\text{\ \ \ \ \ \ \ \ \ if }~~a, c\leq b,\\
K(a+b,a,c)
&\leq Cac^{-1}\left(1+c^{\frac{3}{2}}+ a^{\frac{3}{2}}c^{\frac{3}{2}}\right)^{-\frac{d}{2}}~~~~~~~~~~~~~\text{\ \ \ \ \ \ \ \ \ \ \ \ \ \ \ \ \ if }~~b\leq a\leq c,\\
K(a+b,a,c)
&\leq Ca^{-1}c\left(1+a^{\frac{3}{2}}+ a^{\frac{3}{2}}c^{\frac{3}{2}}\right)^{-\frac{d}{2}}~~~~~~~~~~~~~\text{\ \ \ \ \ \ \ \ \ \ \ \ \ \ \ \ \ if }~~b\leq c\leq a.
\end{align*}
Using the previous inequalities, as well as the condition $d\geq 3$, we can easily check that $K(a+b,a,c)$ is integrable
in $\R_{+}^{3}$, which in turn implies that $K(x,u_{1},u_{2})$ is integrable in $\Sc_{3}$. Using this observation, as well as relations \eqref{eq:Psilogthetaprime} and \eqref{eq:PsiKrel}, we obtain 
\begin{align*}
\lim_{\varepsilon\rightarrow0}\frac{\varepsilon^{d-2}}{\log(1/\varepsilon)}\int_{\widetilde{\Sc}_{3}}\Psi_{\varepsilon}(x,u_{1},u_{2})dxdu_{1}du_{2}
  &=0,
\end{align*}
as required. The proof is now complete.
\end{proof}

The next result provides a useful approximation for  $I_{2}(h_{2,T}^{\varepsilon})$.
\begin{Lema}\label{lem:finiteFGintegrallogaux}
Assume that $H=\frac{3}{4}$ and $d\geq 3$. Let $h_{2,T}^{\varepsilon}$ be defined as in \eqref{eq:kerneldef2} and consider the following approximation of $I_{2}(h_{2,T}^{\varepsilon})$
\begin{align}\label{eq:Jtildedef}
\widetilde{J}_{T}^{\varepsilon}
  &:=-\frac{(2\pi)^{-\frac{d}{2}}\varepsilon^{-\frac{d}{2}+1}}{2}\sum_{j=1}^{d}\int_{0}^{T}\int_{0}^{\infty}
\frac{u^{\frac{3}{2}}}{\varepsilon^{\frac{1}{3}}(1+u^{\frac{3}{2}})^{\frac{d}{2}+1}}	
	H_{2}\left(\frac{B_{s+\varepsilon^{\frac{2}{3}}u}^{(j)}-B_{s}^{(j)}}{\sqrt{\varepsilon}u^{\frac{3}{2}}}\right)duds.
\end{align}
Then we have that 
\begin{align*}
\lim_{\varepsilon\rightarrow0}\frac{\varepsilon^{\frac{d}{2}-1}}{\sqrt{\log(1/\varepsilon)}}\Norm{I_{2}(h_{2,T}^{\varepsilon})-\widetilde{J}_{T}^{\varepsilon}}_{L^{2}(\Omega)}
  &=0.
\end{align*}
\end{Lema}
\begin{proof}
Using \eqref{eq:kerneldef}, we can easily check that
\begin{align*}
I_{2}(h_{2,T}^{\varepsilon})
  &=-\frac{(2\pi)^{-\frac{d}{2}}}{2}\sum_{j=1}^{d}\int_{0}^{T}\int_{0}^{T-u}\frac{u^{\frac{3}{2}}}{(\varepsilon+u^{\frac{3}{2}})^{\frac{d}{2}+1}}H_{2}\left(\frac{B_{s+u}^{(j)}-B_{s}^{(j)}}{u^{\frac{3}{4}}}\right)dsdu.
\end{align*}
Making the change of variables $v:=\varepsilon^{-\frac{2}{3}}u$, we get 
\begin{align*}
I_{2}(h_{2,T}^{\varepsilon})
  &=-\frac{(2\pi)^{-\frac{d}{2}}\varepsilon^{-\frac{d}{2}+\frac{2}{3}}}{2}\sum_{j=1}^{d}\int_{0}^{T}\int_{0}^{\varepsilon^{-\frac{2}{3}}(T-s)} 
  \frac{v^{\frac{3}{2}}}{(1+v^{\frac{3}{2}})^{\frac{d}{2}+1}}
  H_{2}\left(\frac{B_{s+\varepsilon^{\frac{2}{3}}v}^{(j)}-B_{s}^{(j)}}{\sqrt{\varepsilon}v^{\frac{3}{4}}}\right)dvds,
\end{align*}
and hence, 
\begin{align}\label{eq:JTtildeminusJ}
\widetilde{J}_{T}^{\varepsilon}-I_{2}(h_{2,T}^{\varepsilon})
  &=-\frac{(2\pi)^{-\frac{d}{2}}\varepsilon^{-\frac{d}{2}+\frac{2}{3}}}{2}\sum_{j=1}^{d}\int_{0}^{T}\int_{\varepsilon^{-\frac{2}{3}}(T-s)}^{\infty} \frac{v^{\frac{3}{2}}}{(1+v^{\frac{3}{2}})^{\frac{d}{2}+1}}H_{2}\left(\frac{B_{s+\varepsilon^{\frac{2}{3}}u}^{(j)}-B_{s}^{(j)}}{\sqrt{\varepsilon}u^{\frac{3}{4}}}\right)duds.
\end{align}
Set
\begin{align*}
\Phi^\varepsilon 
  &=\varepsilon^{d-2}\Norm{\widetilde{J}_{T}^{\varepsilon}-I_{2}(h_{2,T}^{\varepsilon})}_{L^{2}(\Omega)}^2.
\end{align*}
Using \eqref{eq:JTtildeminusJ}, as well as the fact that 
\begin{align}\label{eq:Hermiteisometry}
\E\left[H_{2}\left(\frac{B_{s_{1}+v_{1}}-B_{s_{1}}}{v_{1}^{H}}\right)H_{2}\left(\frac{B_{s_{2}+v_{2}}-B_{s_{2}}}{v_{2}^{H}}\right)\right]
  &=2(v_{1}v_{2})^{-2H}\mu(s_{2}-s_{1},v_{1},v_{2})^2,
\end{align}
for all $s_{1},s_{2},v_{1},v_{2}\geq0$, we can easily check that
\begin{align*}
\Phi^\varepsilon
  &= \frac{d(2\pi)^{-d}}{2}\int_{[0,T]^2}\int_{\R_{+}^{2}}\Indi{[T,\infty)}(s_{1}+\varepsilon^{\frac{2}{3}}u_{1})\Indi{[T,\infty)}(s_{2}+\varepsilon^{\frac{2}{3}}u_{2})V_{\varepsilon,s_{2}-s_{1}}(u_{1},u_{2})du_{1}du_{2}ds_{1}ds_{2},
\end{align*}
where 
\begin{align*}
V_{\varepsilon,x}(u_{1},u_{2})
  &:=		\varepsilon^{-\frac{8}{3}}\psi(u_{1},u_{2})\mu(x,\varepsilon^{\frac{2}{3}}u_{1},\varepsilon^{\frac{2}{3}}u_{2})^{2},
\end{align*}
and  
\begin{align} \label{psi}
\psi(u_{1},u_{2})
  &:= (1+u_{1}^{\frac{3}{2}})^{-\frac{d}{2}-1}(1+u_{2}^{\frac{3}{2}})^{-\frac{d}{2}-1}.
\end{align}
Hence, using the fact that $\mu(x,v_{1},v_{2})=\mu(-x,v_{2},v_{1})$, we can write
\begin{align}\label{eq:Vcovdisjoint2p}
\Phi^\varepsilon
  &= d(2\pi)^{-d}\int_{0}^{T}\int_{0}^{s_{2}}\int_{\R_{+}^{2}}\Indi{[T,\infty)}(s_{1}+\varepsilon^{\frac{2}{3}}u_{1})\Indi{[T,\infty)}(s_{2}+\varepsilon^{\frac{2}{3}}u_{2})V_{\varepsilon,s_{2}-s_{1}}(u_{1},u_{2})du_{1}du_{2}ds_{1}ds_{2}.
\end{align}
Changing the coordinates $(s_{1},s_{2},u_{1},u_{2})$ by $(s:=s_{1},x:=s_{2}-s_{1},u_{1},u_{2})$ in the expression \eqref{eq:Vcovdisjoint2p}, and then integrating the variable $s$, we obtain
\begin{align*}
\Abs{ \Phi^\varepsilon}
   &=  d(2\pi)^{-d}\int_{0}^{T}\int_{\R_{+}^{2}}(T-(T-\varepsilon^{\frac{2}{3}}u_{1})_{+}\vee(T-x-\varepsilon^{\frac{2}{3}}u_{2})_{+})V_{\varepsilon,x}(u_{1},u_{2})du_{1}du_{2}dx,
\end{align*}
and consequently, there exists a constant $C>0$ such that 
\begin{align}\label{eq:Vcovdisjointp}
\Abs{ \Phi^\varepsilon}
   \leq  C \int_{0}^{T}\int_{\R_{+}^{2}}r_{\varepsilon^{\frac{2}{3}}}(u_{1})V_{\varepsilon,x}(u_{1},u_{2})du_{1}du_{2}dx,
\end{align}
where $r_{\delta}(u_{1}):=T-(T-\delta u_{1})_{+}$. Making the change of variable $v:=\varepsilon^{-\frac{2}{3}}x$ in \eqref{eq:Vcovdisjointp} and using the fact that  $V_{\varepsilon,\varepsilon^{\frac{2}{3}}v}(u_{1},u_{2})=\varepsilon^{-\frac{2}{3}}G_{1,v}^{(1)}(u_{1},u_{2})$, we get  
\begin{align*}
\Abs{ \Phi^\varepsilon}
  &\leq C \int_{0}^{\varepsilon^{-\frac{2}{3}}T}\int_{\R_{+}^2}r_{\varepsilon^{\frac{2}{3}}}(u_{1})G_{1,v}^{(1)}(u_{1},u_{2})du_{1}du_{2}dv.
\end{align*}
Therefore, defining $N:=\varepsilon^{-\frac{2}{3}}$, so that $\log(1/\varepsilon)=\frac{3\log N}{2}$, we obtain 
\begin{align*}
\frac{\Abs{ \Phi^\varepsilon}}{\log(1/\varepsilon)}
  &\leq\frac{2C}{3\log N}\int_{0}^{NT}\int_{\R_{+}^2}r_{\frac{1}{N}}(u_{1})G_{1,x}^{(1)}(u_{1},u_{2})du_{1}du_{2}dx.
\end{align*}
To bound the right-hand side of the previous relation we split the domain of integration as follows. Define the sets $\Sc_{i}$, for $i=1,2,3$, by \eqref{eq:Scdef}. Then
\begin{align}\label{eq:epstoNlogp}
\limsup_{\varepsilon\rightarrow0}\frac{\Abs{ \Phi^\varepsilon}}{\log(1/\varepsilon)}
  &\leq\limsup_{N\rightarrow\infty}\frac{2C }{3\log N}\int_{0}^{NT}\int_{\R_{+}^2}r_{\frac{1}{N}}(x,u_{1})G_{1,x}^{(1)}(u_{1},u_{2})du_{1}du_{2}dx\\
	&\leq\frac{2C}{3}\sum_{i=1}^3\limsup_{N\rightarrow\infty}\frac{1}{\log N}\int_{0}^{NT}\int_{\R_{+}^2}\Indi{\Sc_{i}}(x,u_{1},u_{2})r_{\frac{1}{N}}(u_{1})G_{1,x}^{(1)}(u_{1},u_{2})du_{1}du_{2}dx\nonumber.
\end{align}
By relations \eqref{eq:Fidentitychaos} and \eqref{ineq:FTheta}, there exists a constant $C>0$, such that
\begin{align}\label{eq:G1Theta}
G_{1,x}^{(1)}(u_{1},u_{2})\leq C \frac { \mu^2(x,u_1,u_2)}{(u_1u_2)^{\frac 32}}\Theta_{1}(x,u_{1},u_{2})^{-\frac d2}.
\end{align}
Hence, by Lemma \ref{lem:inegraltightness}, the terms with $i=1$ and $i=2$ in the sum in the right-hand side of \eqref{eq:epstoNlogp} converge to zero. From this observation, we conclude that there exists a constant $C>0$, such that 
\begin{align}\label{eq:Glogdisjoint}
\limsup_{\varepsilon\rightarrow0}\frac{\Abs{ \Phi^\varepsilon}}{\log(1/\varepsilon)}
	\leq \limsup_{N\rightarrow\infty}\frac{C}{\log N}\int_{0}^{NT}\int_{\R_{+}^2}\Indi{\Sc_{3}}(x,u_{1},u_{2})r_{\frac{1}{N}}(u_{1})G_{1,x}^{(1)}(u_{1},u_{2})du_{1}du_{2}dx.
\end{align}
Using Lemma \ref{ineq:techR}, we can easily show that there exists a constant $C>0$, such for every $(x,u_{1},u_{2})\in\Sc_{3}$, the following inequality holds
\begin{align}\label{eq:Glogdisjoint2p}
G_{1,x}^{(1)}(u_{1},u_{2})
  &=\psi(u_1,u_2)\mu(x,u_{1},u_{2})^2\leq C\psi(u_1,u_2)(x+u_{1}+u_{2})^{-1}(u_{1}u_{2})^2,
\end{align}
where $ \psi(u_1, u_2)$ is defined in  (\ref{psi}).
From \eqref{eq:Glogdisjoint} and \eqref{eq:Glogdisjoint2p}, it follows that 
\begin{align*}
\limsup_{\varepsilon\rightarrow0}\frac{\Abs{ \Phi^\varepsilon}}{\log(1/\varepsilon)}
	&\leq \limsup_{N\rightarrow\infty}\frac{C}{\log N}\int_{0}^{NT}\int_{\R_{+}^2}r_{\frac{1}{N}}(u_{1})(x+u_{1}+u_{2})^{-1}(u_{1}u_{2})^2\psi(u_1,u_2)du_{1}du_{2}dx.
\end{align*}
In addition, we have  that
\begin{multline*}
\limsup_{N\rightarrow\infty}\frac{1}{\log N}\int_{0}^{1}\int_{\R_{+}^2}r_{\frac{1}{N}}(u_{1})(x+u_{1}+u_{2})^{-1}(u_{1}u_{2})^2\psi(u_1,u_2)du_{1}du_{2}dx\\
\leq\limsup_{N\rightarrow\infty}\frac{T}{\log N}\int_{0}^{1}\int_{\R_{+}^2}(u_{1}+u_{2})^{-1}(u_{1}u_{2})^2\psi(u_1,u_2)du_{1}du_{2}=0,
\end{multline*}
and consequently, 
\begin{align*}
\limsup_{\varepsilon\rightarrow0}\frac{\Abs{ \Phi^\varepsilon}}{\log(1/\varepsilon)}
	&\leq \limsup_{N\rightarrow\infty}\frac{C}{\log N}\int_{1}^{NT}\int_{\R_{+}^2}r_{\frac{1}{N}}(u_{1})x^{-1}(u_{1}u_{2})^2\psi(u_1,u_2)du_{1}du_{2}dx.
\end{align*}
For $\delta>0$ fixed, let $M>1$ be such that 
\begin{align}\label{ineq:reminderdelta}
\int_{M}^{\infty}\int_{0}^{\infty}(u_{1}u_{2})^2\psi(u_1,u_2)du_{2}du_{1}<\delta.
\end{align}
Using \eqref{ineq:reminderdelta}, as well as the fact that $r_{\frac{1}{N}}(u)$ is increasing on $u$, we obtain
\begin{align*}
\frac{1}{\log N}\int_{1}^{NT}\int_{M}^{\infty}\int_{0}^{\infty}x^{-1}(u_{1}u_{2})^2\psi(u_1,u_2)du_{1}du_{2}dx\leq\delta\left(1+\frac{\log(T)}{\log N}\right),
\end{align*}
and 
\begin{multline*}
\limsup_{N\rightarrow\infty}\frac{1}{\log N}\int_{1}^{NT}\int_{0}^{M}\int_{0}^{\infty}r_{\frac{1}{N}}(u_{1})x^{-1}(u_{1}u_{2})^2\psi(u_1,u_2)du_{1}du_{2}dx\\
\leq\limsup_{N\rightarrow\infty}\left(1+\frac{\log(T)}{\log N}\right)\int_{\R_{+}^2}r_{\frac{1}{N}}(M)(u_{1}u_{2})^2\psi(u_1,u_2)du_{1}du_{2}=0.
\end{multline*}
As a consequence, 
\begin{align*}
\limsup_{\varepsilon\rightarrow0}\frac{\Abs{ \Phi^\varepsilon}}{\log(1/\varepsilon)}\leq C\delta.
\end{align*}
Hence, taking $\delta\rightarrow0$, we get 
\begin{align*}
\lim_{\varepsilon\rightarrow0}\frac{\Phi^\varepsilon}{\log(1/\varepsilon)}
	&=0,
\end{align*}
as required.
\end{proof}

Finally, we describe the behavior of the covariance function of $I_{2}(h_{2,T}^{\varepsilon})$ for the case $H=\frac{3}{4}$.
\begin{Teorema}\label{teo:covchaoslog}
Let $T_{1},T_{2}\geq0$ be fixed. Then, if $d\geq3 $ and $H=\frac{3}{4}$,
\begin{align*}
\lim_{\varepsilon\rightarrow0}\frac{\varepsilon^{d-2}}{\log(1/\varepsilon)}\E\left[I_{2}(h_{2,T_{1}}^{\varepsilon})I_{2}(h_{2,T_{2}}^{\varepsilon})\right]
  &=\rho^{2}(T_{1}\wedge T_{2}),
\end{align*}
where $\rho$ is a finite constant defined by
\begin{align}\label{eq:rhodeflog}
\rho
  &:=\frac{\sqrt{3d}}{2^{\frac{d+5}{2}}\pi^{\frac{d}{2}}}\int_{0}^{\infty}(1+u^{\frac{3}{2}})^{-\frac{d}{2}-1}u^{2}du.
\end{align}
\end{Teorema}
\begin{proof}
Consider the approximation $\widetilde{J}_{T}^{\varepsilon}$ of $I_{2}(h_{2,T}^{\varepsilon})$, introduced in  \eqref{eq:Jtildedef}.  By Lemma \ref{lem:finiteFGintegrallogaux}, 
$$\lim_{\varepsilon\rightarrow0}\frac{\varepsilon^{d-2}}{\log(1/\varepsilon)}\Norm{\widetilde{J}_{T}^{\varepsilon}-I_{2}(h_{2,T}^{\varepsilon})}_{L^{2}(\Omega)}^2\rightarrow0.$$ 
Therefore, it suffices to show that 
\begin{align}\label{eq:limJtildeaux}
\lim_{\varepsilon\rightarrow0}\frac{\varepsilon^{d-2}}{\log(1/\varepsilon)}\E\left[\widetilde{J}_{T_{1}}^{\varepsilon}\widetilde{J}_{T_{2}}^{\varepsilon}\right]
  &=\rho^{2}(T_{1}\wedge T_{2}).
\end{align}
As in Lemma \ref{section:variances}, to prove \eqref{eq:limJtildeaux}, it suffices to show that for each $a<b<\alpha<\beta$, 
\begin{align}\label{eq:covfddchaosdisjointlog}
\lim_{\varepsilon\rightarrow0}\frac{\varepsilon^{d-2}}{\log(1/\varepsilon)}\E\left[(\widetilde{J}_{b}^{\varepsilon}-\widetilde{J}_{a}^{\varepsilon})(\widetilde{J}_{\beta}^{\varepsilon}-\widetilde{J}_{\alpha}^{\varepsilon})\right]=0,
\end{align}
and
\begin{align}\label{eq:varfddchaoslog}
\lim_{\varepsilon\rightarrow0}\frac{\varepsilon^{d-2}}{\log(1/\varepsilon)}\E\left[(\widetilde{J}_{b}^{\varepsilon}-\widetilde{J}_{a}^{\varepsilon})^2\right]= \rho^{2}(b-a).
\end{align}
First we prove \eqref{eq:covfddchaosdisjointlog}. Set
\begin{align*}
\Phi^\varepsilon 
  &=\varepsilon^{d-2}\E\left[(\widetilde{J}_{b}^{\varepsilon}-\widetilde{J}_{a}^{\varepsilon})(\widetilde{J}_{\beta}^{\varepsilon}-\widetilde{J}_{\alpha}^{\varepsilon})\right].
\end{align*}
Using \eqref{eq:Hermiteisometry}\ and \eqref{eq:Jtildedef}, we can easily check that
\begin{align}\label{eq:Vcovdisjoint2prime}
\Phi^\varepsilon
  &= \frac{d(2\pi)^{-d}}{2}\int_{\alpha}^{\beta}\int_{a}^{b}\int_{\R_{+}^{2}}V_{\varepsilon,s_{2}-s_{1}}(u_{1},u_{2})du_{1}du_{2}ds_{1}ds_{2},
\end{align}
where 
\begin{align*}
V_{\varepsilon,x}(u_{1},u_{2})
	&:=\varepsilon^{-\frac{8}{3}}\psi(u_1,u_2)\mu(x,\varepsilon^{\frac{2}{3}}u_{1},\varepsilon^{\frac{2}{3}}u_{2})^{2},
\end{align*}
and $\psi(u_1,u_2)$ is defined by \eqref{eq:Kdef}. 
Changing the coordinates $(s_{1},s_{2},u_{1},u_{2})$ by $(s:=s_{1},x:=s_{2}-s_{1},u_{1},u_{2})$ in \eqref{eq:Vcovdisjoint2prime},
and then integrating the variable $s$, we can show that 
\begin{align}\label{eq:Vcovdisjoint}
\Abs{ \Phi^\varepsilon}
   \leq  d(2\pi)^{-d}\beta \int_{\gamma}^{\beta}\int_{\R_{+}^{2}}V_{\varepsilon,x}(u_{1},u_{2})du_{1}du_{2}dx,
\end{align}
where the constant $\gamma$ is defined by $\gamma:=\alpha-b$. 
Making the change of variable $v:=\varepsilon^{-\frac{2}{3}}x$ and using the fact that  
$$V_{\varepsilon,\varepsilon^{\frac{2}{3}}v}(u_{1},u_{2})=\varepsilon^{-\frac{2}{3}}G_{1,v}^{(1)}(u_{1},u_{2}),$$ 
we get  
\begin{align*}
\Abs{\Phi^{\varepsilon}}
  &\leq d(2\pi)^{-d}\beta\int_{\varepsilon^{-\frac{2}{3}}\gamma}^{\varepsilon^{-\frac{2}{3}}\beta}\int_{\R_{+}^2}G_{1,v}^{(1)}(u_{1},u_{2})\text{d}u_{1}\text{d}u_{2}\text{d}v.
\end{align*}
Therefore, defining $N:=\varepsilon^{-\frac{2}{3}}$, so that $\log(1/\varepsilon)=\frac{3\log N}{2}$, we obtain 
\begin{align*}
\frac{\Abs{\Phi^{\varepsilon}}}{\log(1/\varepsilon)}
  &\leq\frac{2d(2\pi)^{-d}\beta}{3\log N}\int_{N\gamma}^{N\beta}\int_{\R_{+}^2}G_{1,x}^{(1)}(u_{1},u_{2})\text{d}u_{1}\text{d}u_{2}\text{d}x.
\end{align*}
To bound the right-hand side of the previous relation we split the domain of integration as follows. Define the sets $\Sc_{i}$, for $i=1,2,3$, by \eqref{eq:Scdef}. Then, there exists a constant $C>0$, such that
\begin{align}\label{eq:epstoNlog}
\limsup_{\varepsilon\rightarrow0}\frac{\Abs{\Phi^{\varepsilon}}}{\log(1/\varepsilon)}	
  &\leq\limsup_{N\rightarrow\infty}\frac{C}{\log N}\int_{N\gamma}^{N\beta}\int_{\R_{+}^2}G_{1,x}^{(1)}(u_{1},u_{2})\text{d}u_{1}\text{d}u_{2}\text{d}x\nonumber\\
	&\leq\sum_{i=1}^3\limsup_{N\rightarrow\infty}\frac{C}{\log N}\int_{N\gamma}^{N\beta}\int_{\R_{+}^2}\Indi{\Sc_{i}}(x,u_{1},u_{2})G_{1,x}^{(1)}(u_{1},u_{2})du_{1}\text{d}u_{2}dx.
\end{align}
Taking into account (\ref{eq:G1Theta}), by Lemma \ref{lem:inegraltightness}, the terms with $i=1$ and $i=2$ in the sum in the right-hand side of \eqref{eq:epstoNlog} converge to zero. From this observation, we conclude that
\begin{align}\label{eq:Glogdisjoint1}
\limsup_{\varepsilon\rightarrow0}\frac{\Abs{\Phi^{\varepsilon}}}{\log(1/\varepsilon)}
	\leq\limsup_{N\rightarrow\infty}\frac{C}{\log N}\int_{N\gamma}^{N\beta}\int_{\R_{+}^2}\Indi{\Sc_{3}}(x,u_{1},u_{2})G_{1,x}^{(1)}(u_{1},u_{2})\text{d}u_{1}du_{2}dx.
\end{align}
By Lemma \ref{ineq:techR}, there exists a constant $C>0$, such for every $(x,u_{1},u_{2})\in\Sc_{3}$, the following inequality holds
\begin{align}\label{eq:Glogdisjoint2}
G_{1,x}^{(1)}(u_{1},u_{2})
  =\psi(u_1,u_2)\mu(x,u_{1},u_{2})^2\leq C\psi(u_1,u_2)x^{-1}(u_{1}u_{2})^2.
\end{align}
From \eqref{eq:Glogdisjoint1} and \eqref{eq:Glogdisjoint2}, we obtain 
\begin{align*}
\limsup_{\varepsilon\rightarrow0}\frac{\Abs{\Phi^{\varepsilon}}}{\log(1/\varepsilon)}
	\leq C\limsup_{N\rightarrow\infty}\frac{\log(N\beta)-\log(N\gamma)}{\log N}\int_{\R_{+}^2}\psi(u_1,u_2)(u_{1}u_{2})^2du_{1}du_{2},
\end{align*}
for some constant $C>0$. The function $(1+u^{\frac{3}{2}})^{-\frac{d}{2}-1}u^2$ is integrable for $u$ in $\R_{+}$ due to the condition $d\geq 3$, and hence, from the previous inequality we conclude that
\begin{align}\label{eq:covlogconclusion}
\limsup_{\varepsilon\rightarrow0}\frac{\Abs{\Phi^{\varepsilon}}}{\log(1/\varepsilon)}=0.
\end{align}
Relation \eqref{eq:covfddchaosdisjointlog} then follows from \eqref{eq:covlogconclusion}.\\

Next we prove \eqref{eq:varfddchaoslog}. By taking $\alpha=a$ and $\beta=b$ in relation \eqref{eq:Vcovdisjoint2prime}, we obtain
\begin{align*}
\varepsilon^{d-2}\E\left[(\widetilde{J}_{b}^{\varepsilon}-\widetilde{J}_{a}^{\varepsilon})^2\right]
	&=d(2\pi)^{-d}\int_{a}^{b}\int_{a}^{s_{2}}\int_{\R_{+}^{2}}V_{\varepsilon,s_{2}-s_{1}}(u_{1},u_{2})du_{1}du_{2}ds_{1}ds_{2}.
\end{align*}
Changing the coordinates $(s_{1},s_{2},t_{1},t_{2})$ by $(s_{1},x:=\varepsilon^{-\frac{2}{3}}(s_{2}-s_{1}),u_{1}:=t_{1}-s_{1},u_{2}:=t_{2}-s_{2})$, integrating the variable $s_{1}$ and using the fact that $V_{\varepsilon,\varepsilon^{\frac{2}{3}}\hat{x}}(u_{1},u_{2})=\varepsilon^{-\frac{2}{3}}G_{1,\hat{x}}^{(1)}(u_{1},u_{2})$, we deduce that
\begin{align*}
\varepsilon^{d-2}\E\left[(\widetilde{J}_{b}^{\varepsilon}-\widetilde{J}_{a}^{\varepsilon})^2\right]
  &=d(2\pi)^{-d}\int_{0}^{\varepsilon^{-\frac{2}{3}}(b-a)}\int_{\R_{+}^{2}}(b-\varepsilon^{\frac{2}{3}}x-a)\varepsilon^{\frac{2}{3}}V_{\varepsilon,\varepsilon^{\frac{2}{3}}x}(u_{1},u_{2})du_{1}du_{2}dx\\
	&=d(2\pi)^{-d}\int_{0}^{\varepsilon^{-\frac{2}{3}}(b-a)}\int_{\R_{+}^{2}}(b-\varepsilon^{\frac{2}{3}}x-a)G_{1,x}^{(1)}(u_{1},u_{2})du_{1}du_{2}dx.
\end{align*}
Therefore, defining $N:=\varepsilon^{-\frac{2}{3}}$, so that $\log(1/\varepsilon)=\frac{3\log N}{2}$, we obtain 
\begin{multline}\label{eq:InckerG11prime}
\frac{\varepsilon^{d-2}}{\log(1/\varepsilon)}\E\left[(\widetilde{J}_{b}^{\varepsilon}-\widetilde{J}_{a}^{\varepsilon})^2\right]\\
\begin{aligned}
	&=\frac{2d(2\pi)^{-d}}{3\log N}\int_{0}^{N(b-a)}\int_{\R_{+}^{2}}(b-\frac{x}{N}-a)G_{1,x}^{(1)}(u_{1},u_{2})du_{1}du_{2}dx\\
	&=\frac{2d(2\pi)^{-d}}{3\log N}\sum_{i=1}^{3}\int_{0}^{N(b-a)}\int_{\R_{+}^{2}}(b-\frac{x}{N}-a)\Indi{\Sc_{i}}(x,u_{1},u_{2})G_{1,x}^{(1)}(u_{1},u_{2})du_{1}du_{2}dx.
\end{aligned}
\end{multline}
By inequality \eqref{eq:G1Theta} and Lemma \ref{lem:inegraltightness}, the terms with $i=1$ and $i=2$ in the sum in the right-hand side of \eqref{eq:InckerG11prime} converge to zero. From this observation, it follows that 
\begin{multline}\label{eq:InckerG11}
\lim_{\varepsilon\rightarrow0}\frac{\varepsilon^{d-2}}{\log(1/\varepsilon)}\E\left[(\widetilde{J}_{b}^{\varepsilon}-\widetilde{J}_{a}^{\varepsilon})^2\right]\\
\begin{aligned}
	&=\lim_{N\rightarrow\infty}\frac{2d(2\pi)^{-d}}{3\log N}\int_{0}^{N(b-a)}\int_{\R_{+}^{2}}(b-\frac{x}{N}-a)\Indi{\Sc_{3}}(x,u_{1},u_{2})G_{1,x}^{(1)}(u_{1},u_{2})du_{1}du_{2}dx\\
	&=\lim_{N\rightarrow\infty}\frac{2d(2\pi)^{-d}}{3\log N}\int_{0}^{N(b-a)}\int_{\R_{+}^{2}}(b-a)\Indi{\Sc_{3}}(x,u_{1},u_{2})G_{1,x}^{(1)}(u_{1},u_{2})du_{1}du_{2}dx\\
	&-\lim_{N\rightarrow\infty}\frac{2d(2\pi)^{-d}}{3N\log N}\int_{0}^{N(b-a)}\int_{\R_{+}^{2}}\Indi{\Sc_{3}}(x,u_{1},u_{2})xG_{1,x}^{(1)}(u_{1},u_{2})du_{1}du_{2}dx,
\end{aligned}
\end{multline}
provided that the limits in the right-hand side exist. By \eqref{eq:Glogdisjoint2}, there exists a constant $C>0$ such that 
\begin{multline*}
\frac{1}{N\log N}\int_{0}^{N(b-a)}\int_{\R_{+}^{2}}\Indi{\Sc_{3}}(x,u_{1},u_{2})xG_{1,x}^{(1)}(u_{1},u_{2})du_{1}du_{2}dx\\
\begin{aligned}
 &\leq \frac{C}{N\log N}\int_{0}^{N(b-a)}\int_{\R_{+}^{2}}\psi(u_1,u_2)(u_{1}u_{2})^2du_{1}du_{2}dx\\
&= \frac{C(b-a)}{\log N}\int_{\R_{+}^{2}}\psi(u_1,u_2)(u_{1}u_{2})^2du_{1}du_{2}.
\end{aligned}
\end{multline*}
Since $d\geq 3$, the integral in the right-hand side is finite, and hence 
\begin{align*}
\lim_{N\rightarrow\infty}\frac{1}{N\log N}\int_{0}^{N(b-a)}\int_{\R_{+}^{2}}\Indi{\Sc_{3}}(x,u_{1},u_{2})xG_{1,x}^{(1)}(u_{1},u_{2})du_{1}du_{2}dx=0.
\end{align*}
Therefore, by \eqref{eq:InckerG11}, 
\begin{multline}\label{eq:InckerG111}
\lim_{\varepsilon\rightarrow0}\frac{\varepsilon^{d-2}}{\log(1/\varepsilon)}\E\left[(\widetilde{J}_{b}^{\varepsilon}-\widetilde{J}_{a}^{\varepsilon})^2\right]\\
\begin{aligned}
	&=\lim_{N\rightarrow\infty}\frac{2d(2\pi)^{-d}}{3\log N}\int_{0}^{N(b-a)}\int_{\R_{+}^{2}}(b-a)\Indi{\Sc_{3}}(x,u_{1},u_{2})G_{1,x}^{(1)}(u_{1},u_{2})du_{1}du_{2}dx.
\end{aligned}
\end{multline}
Applying L'\^Hopital's rule in \eqref{eq:InckerG111}, we get 
\begin{multline}\label{eq:InckerG12}
\lim_{\varepsilon\rightarrow0}\frac{\varepsilon^{d-2}}{\log(1/\varepsilon)}\E\left[(\widetilde{J}_{b}^{\varepsilon}-\widetilde{J}_{a}^{\varepsilon})^2\right]\\
\begin{aligned}
	&=\lim_{N\rightarrow\infty}\frac{2d(2\pi)^{-d}}{3}\int_{\R_{+}^{2}}N(b-a)^2\Indi{\Sc_{3}}(N(b-a),u_{1},u_{2})G_{1,N(b-a)}^{(1)}(u_{1},u_{2})du_{1}du_{2}dx.
\end{aligned}
\end{multline}
By \eqref{eq:Glogdisjoint2}, the integrand in the right-hand side is bounded by the function 
$$C\psi(u_1,u_2)(u_{1}u_{2})^2$$ 
for some constant $C>0$. On the other hand, using \eqref{eq:IpHbig}, we can easily check that  
\begin{align*}
\Abs{\mu(x,v_{1},v_{2})}
  &=\Abs{\Ip{\Indi{[0,v_{1}]},\Indi{[x,x+v_{2}]}}_{\Hg}}\\
	&=H(2H-1)v_{1}v_{2}\int_{[0,1]^{2}}\Abs{x+v_{2}w_{2}-v_{1}w_{1}}^{2H-2}dw_{1}dw_{2}\\
	&=\frac{3v_{1}v_{2}}{8}\int_{[0,1]^{2}}\Abs{x+v_{2}w_{2}-v_{1}w_{1}}^{-\frac{1}{2}}dw_{1}dw_{2},
\end{align*}
so that
\begin{align*}
\lim_{N\rightarrow\infty}N(b-a)\mu(N(b-a),u_{1},u_{2})^2
  &=\frac{3^2(u_{1}u_{2})^2}{2^{6}},
\end{align*}
and hence, 
\begin{align*}
\lim_{N\rightarrow\infty}N(b-a)\Indi{\Sc_{3}}(N(b-a),u_{1},u_{2})G_{1,N(b-a)}^{(1)}(u_{1},u_{2})=\frac{3^2}{2^6}\psi(u_1,u_2)(u_{1}u_{2})^2.
\end{align*}
Therefore, by applying the dominated convergence theorem to \eqref{eq:InckerG12}, we get 
\begin{align*}
\lim_{\varepsilon\rightarrow0}\frac{\varepsilon^{d-2}}{\log(1/\varepsilon)}\E\left[(\widetilde{J}_{b}^{\varepsilon}-\widetilde{J}_{a}^{\varepsilon})^2\right]
	&=(b-a)\frac{3d}{2^{d+5}\pi^{d}}\left(\int_{\R_{+}}(1+u^{\frac{3}{2}})^{-\frac{d}{2}-1}u^2du\right)^2.
\end{align*}
\end{proof}
Relation \eqref{eq:varfddchaoslog} follows from the previous inequality. The proof is now complete.
\section{Proof of Theorems \ref{Teo:convergence}, \ref{Teo:convergenceHermite} and \ref{Teo:convergencelog}}\label{section:limit_theorems}
In the sequel, $W=\{W_{t}\}_{t\geq0}$ will denote a standard one-dimensional Brownian motion independent of $B$, and $X^{j}=\{X_{t}^{j}\}_{t\geq0}$ will denote the second order Hermite process introduced in Section \ref{sec:chaos}.\\

\noindent\textbf{Proof of Theorem \ref{Teo:convergence}}\\
We start with the proof of Theorem \ref{Teo:convergence}, which will be done in two steps.

\medskip
\noindent
{\it Step 1.} First we prove the convergence of the finite dimensional distributions, namely, we will show that for every $r\in\N$, and $T_{1},\dots, T_{r}\geq0$ fixed, it holds 
\begin{align}\label{eq:fdd}
\varepsilon^{\frac{d}{2}-\frac{3}{4H}}\left((I_{T_{1}}^{\varepsilon},\dots, I_{T_{r}}^{\varepsilon})-\E\left[(I_{T_{1}}^{\varepsilon},\dots, I_{T_{r}}^{\varepsilon})\right]\right)\stackrel{Law}{\rightarrow} \sigma (W_{T_{1}},\dots, W_{T_{r}}),
\end{align}
as $\varepsilon\rightarrow0$, where $\sigma$ is the finite constant defined by \eqref{eq:sigmadef}. To this end, define the kernels $h_{2q,T_{i}}^{\varepsilon}$ by \eqref{eq:kerneldef2}, and the constants $\sigma_{q}^2$ by \eqref{eq:sigqdef}, for $q\in\N$. Notice that the constants $\sigma_{q}^2$ are well defined due to the condition $\frac{3}{2d}<H<\frac{3}{4}$. 
Define as well the matrices $C_{q}=\{C_{q}^{i,j}\ |\ 1\leq i,j\leq r\}$ and $C=\{C^{i,j}\ |\ 1\leq i,j\leq r\}$, by $C_{q}^{i,j}:=\sigma_{q}^{2}(T_{i}\wedge T_{j})$, and $C^{i,j}:=\sigma^{2}(T_{i}\wedge T_{j})$. Since $I_{T_{i}}^{\varepsilon}$ has chaos decomposition \eqref{eq:chaos2}, by Theorem \ref{Fourthmoment}, we deduce that in order to prove the convergence \eqref{eq:fdd}, it suffices to show the following properties:
\begin{enumerate}
\item[(i)] For every fixed $q\geq1$, and $1\leq i,j\leq r$, we have 
$$\varepsilon^{d-\frac{3}{2H}}(2q)!\Ip{h_{2q,T_{i}}^{\varepsilon},h_{2q,T_j}^{\varepsilon}}_{(\Hg^{d})^{\otimes 2q}}\rightarrow \sigma_{q}^{2}(T_{i}\wedge T_{j}),~~~~\text{ as }~~\varepsilon\rightarrow0.$$ 
\item[(ii)] The constants $\sigma_{q}^{2}$ satisfy $\sum_{q=1}^{\infty}\sigma_{q}^{2}=\sigma^{2}$. In particular, $C^{i,j}=\lim_{Q\rightarrow\infty}\sum_{q=1}^{Q}C_{q}^{i,j}$,
\item[(iii)] For all $q\geq1$ and $i=1,\dots, r$, the random variables $\varepsilon^{\frac{d}{2}-\frac{3}{4H}}I_{2q}(h_{2q,T_{i}}^{\varepsilon})$ converge in law to a centered Gaussian distribution as $\varepsilon\rightarrow0$,
\item[(iv)] $\lim_{Q\rightarrow\infty}\sup_{\varepsilon\in(0,1)}\varepsilon^{d-\frac{3}{2H}}\sum_{q=Q}^{\infty}(2q)!\Norm{h_{2q,T_i}^{\varepsilon}}_{(\Hg^{d})^{\otimes 2q}}^{2}=0$, for every $i=1,\dots, r$.
\end{enumerate}
Part (i) follows from Theorem \ref{teo:covchaos}. Condition (ii) follows from equation \eqref{eq:sumsigmaq}. In \cite[Theorem~2]{HuNu}, it was proved that for $T>0$ fixed, 
$\varepsilon^{\frac{d}{2}-\frac{3}{4H}}I_{2q}(h_{2q,T}^{\varepsilon})$ converges in law to a centered Gaussian random variable when $\varepsilon\rightarrow0$, and 
$$\lim_{Q\rightarrow\infty}\sup_{\varepsilon\in(0,1)}\varepsilon^{d-\frac{3}{2H}}\sum_{q=Q}^{\infty}(2q)!\Norm{h_{2q,T}^{\varepsilon}}_{(\Hg^{d})^{\otimes 2q}}^{2}=0,$$
which proves conditions (iii) and (iv). This finishes the proof of \eqref{eq:fdd}. 
 
 \medskip
 \noindent
 {\it Step 2.} We are going to show the tightness  of  the sequence of processes $\{\varepsilon^{\frac{d}{2}-\frac{3}{4H}}(I_{T}^{\varepsilon}-\E\left[I_{T}^{\varepsilon}\right])\}_{T\geq0}$. To this end, we will prove that there exists a sufficiently small $p>2$, depending only on $d$ and $H$, such that for every $0\leq T_{1}\leq T_{2}$, it holds
\begin{align}\label{eq:billbound}
\sup_{\varepsilon\in(0,1)}\E\left[\left|\varepsilon^{\frac{d}{2}-\frac{3}{4H}}\left(I_{T_{2}}^{\varepsilon}-\E\left[I_{T_{2}}^{\varepsilon}\right]-\left(I_{T_{1}}^{\varepsilon}-\E\left[I_{T_{1}}^{\varepsilon}\right]\right)\right)\right|^{p}\right]\leq C\Abs{T_{2}-T_{1}}^{\frac{p}{2}}, 
\end{align}
for some constant $C>0$ only  depending on $d$, $p$ and $H$. The tightness property for $\{\varepsilon^{\frac{d}{2}-\frac{3}{4H}}(I_{T}^{\varepsilon}-\E\left[I_{T}^{\varepsilon}\right])\}_{T\geq0}$ then follows from the Billingsley criterion (see \cite[Theorem~12.3]{Billin}). 
 
In order to prove \eqref{eq:billbound} we proceed as follows. Define, for $0\le T_{1} \le T_{2}$ fixed, the random variable $Z_{\varepsilon}=Z_{\varepsilon}(T_{1},T_{2})$, by 
\begin{align}\label{eq:Zdef}
Z_{\varepsilon}
  &:=I_{T_{2}}^{\varepsilon}-\E\left[I_{T_{2}}^{\varepsilon}\right]-\left(I_{T_{1}}^{\varepsilon}-\E\left[I_{T_{1}}^{\varepsilon}\right]\right).
\end{align}
From the chaos decomposition \eqref{eq:chaos2}, we can easily check that $J_{0}(L^{-1}Z_{\varepsilon})=J_{1}(L^{-1}Z_{\varepsilon})=0$, which in turn implies that 
$$\E\left[DL^{-1}Z_{\varepsilon}\right]=J_{0}(DL^{-1}Z_{\varepsilon})=DJ_{1}(L^{-1}Z_{\varepsilon})=0.$$ 
Hence, by \eqref{eq:Meyer}, there exists a constant $c_{p}>0$ such that 
\begin{align}\label{eq:ZMeyer}
\Norm{Z_{\varepsilon}}_{L^{p}(\Omega)}
  &\leq c_{p}\Norm{D^{2}L^{-1}Z_{\varepsilon}}_{L^{p}(\Omega;(\Hg^{d})^{\otimes 2})}.
\end{align}
The right-hand side of the previous inequality can be estimated as follows. From \eqref{eq:Mehler2}, we can easily check that
\begin{align}\label{eq:D2Linv}
D^2L^{-1} Z_{\varepsilon} 
  &=\int_0^\infty\int_{\mathcal{K}_{T_{1},T_{2}}} D^2P_{\theta}[p_{\varepsilon}(B_{t}-B_{s})]dsdtd\theta,
\end{align}
where $\mathcal{K}_{T_{1},T_{2}}$ is defined by \eqref{eq:Sdef}. Let $\widetilde{B}$ be an independent copy of $B$. Using  Mehler's formula \eqref{eq:Mehler} and the semigroup property of the heat kernel, we obtain  
\begin{eqnarray}\label{eq:PGker}
P_{\theta}[p_{\varepsilon}(B_{t}-B_{s})]
  &=&\widetilde{\E}\left[p_{\varepsilon}(e^{-\theta}(B_{t}-B_{s})+\sqrt{1-e^{-2\theta}}(\widetilde{B}_{t}-\widetilde{B}_{s}))\right]\\
  &=& p_{\lambda_{\varepsilon}(\theta,s,t)}(e^{-\theta}(B_{t}-B_{s})),\nonumber
  \end{eqnarray}
where the function $\lambda_{\varepsilon}=\lambda_{\varepsilon}(\theta,s,t)$ is defined by 
\begin{align}\label{eq:lamdadef}
\lambda_{\varepsilon}(\theta,s,t)
  &:=\varepsilon+(1-e^{-2\theta})(t-s)^{2H}.
\end{align}
This implies that for every multi-index $\textbf{i}=(i_{1},i_{2})$, with $1\leq i_{1},i_{2}\leq d$, we have
\begin{eqnarray}   \nonumber
&&D^2 P_\theta [p_\varepsilon (B_t-B_s)](\textbf{i},x_{1},x_{2}) =
   e^{-2\theta}\Indi{[s,t]}(x_{1})\Indi{[s,t]}(x_{2})\\
	& & \qquad \times  \lambda_{\varepsilon}(\theta,s,t)^{-1}p_{\lambda_{\varepsilon}(\theta,s,t)}(e^{-\theta}(B_{t}-B_{s}))g_{\textbf{i},\lambda_{\varepsilon}(\theta,s,t)}(e^{-\theta}(B_{t}-B_{s})), \label{eq:D2Ornstein1}
\end{eqnarray}
where the function $g_{\textbf{i},\lambda}$, for $\lambda>0$, is defined by 
\begin{align*}
g_{\textbf{i},\lambda}(x_{1},\dots, x_{d})
  &=\Multicase{\lambda^{-1}x_{i_{1}}^2-1}{i_{1}=i_{2}}{\lambda^{-1}x_{i_{1}}x_{i_{2}}}{i_{1}\neq i_{2}.}
\end{align*}
From \eqref{eq:D2Linv} and \eqref{eq:D2Ornstein1}, we deduce that 
\begin{align}\label{ineq:Zp1arg}
\Norm{D^2L^{-1}Z_{\varepsilon}}_{(\Hg^{d})^{\otimes 2}}^{2}
	&= \int_{\R_{+}^{2}}\int_{\mathcal{K}_{T_1,T_2}^2}e^{-2\theta-2\beta}\mu(s_{2}-s_{1},t_{1}-s_{1},t_{2}-s_{2})^2\nonumber\\
	&\times (\lambda_{\varepsilon}(\theta,s_{1},t_{1})\lambda_{\varepsilon}(\beta,s_{2},t_{2}))^{-1}p_{\lambda_{\varepsilon}(\theta,s_{1},t_{1})}(e^{-\theta} (B_{t_1}-B_{s_1}))\nonumber\\
	&\times p_{\lambda_{\varepsilon}(\beta,s_{2},t_{2})} (e^{-\beta} (B_{t_2}-B_{s_2})) \sum_{\textbf{i}}g_{\textbf{i},\lambda_{\varepsilon}(\theta,s_{1},t_{1})}\left(e^{-\theta}(B_{t_{1}}-B_{s_{1}})\right)\nonumber\\
	&\times g_{\textbf{i},\lambda_{\varepsilon}(\beta,s_{2},t_{2})}\left(e^{-\beta}(B_{t_{2}}-B_{s_{2}})\right)ds_1 dt_1 ds_2 dt_2d\theta d\beta,
\end{align}
where the sum runs over all the possible muti-indices $\textbf{i}=(i_{1},i_{2})$, with $1\leq i_{1},i_{2}\leq d$. Using Minkowski inequality, as well as \eqref{eq:ZMeyer} and \eqref{ineq:Zp1arg}, we deduce that 
\begin{align}\label{ineq:Zp1}
\Norm{Z_{\varepsilon}}_{L^{p}(\Omega)}^{2}
	&\leq c_{p}^{2}\Norm{D^{2}L^{-1}Z_{\varepsilon}}_{L^{p}(\Omega;(\Hg^{d})^{\otimes2})}^{2}
	=c_{p}^{2}\Norm{\Norm{D^{2}L^{-1}Z_{\varepsilon}}_{(\Hg^{d})^{\otimes 2}}^{2}}_{L^{\frac{p}{2}}(\Omega)}\nonumber\\
  &\leq c_{p}^2\int_{\R_{+}^{2}}\int_{\mathcal{K}_{T_1,T_2}^2}e^{-2\theta-2\beta}\mu(s_{2}-s_{1},t_{1}-s_{1},t_{2}-s_{2})^2\nonumber\\
	&\times (\lambda_{\varepsilon}(\theta,s_{1},t_{1})\lambda_{\varepsilon}(\beta,s_{2},t_{2}))^{-1}\big\|p_{\lambda_{\varepsilon}(\theta,s_{1},t_{1})}(e^{-\theta} (B_{t_1}-B_{s_1}))\nonumber\\
	&\times p_{\lambda_{\varepsilon}(\beta,s_{2},t_{2})} (e^{-\beta} (B_{t_2}-B_{s_2})) \sum_{\textbf{i}}g_{\textbf{i},\lambda_{\varepsilon}(\theta,s_{1},t_{1})}\left(e^{-\theta}(B_{t_{1}}-B_{s_{1}})\right)\nonumber\\
	&\times g_{\textbf{i},\lambda_{\varepsilon}(\beta,s_{2},t_{2})}\left(e^{-\beta}(B_{t_{2}}-B_{s_{2}})\right) \big\|_{L^{\frac{p}{2}}(\Omega)}ds_1 dt_1 ds_2 dt_2d\theta d\beta.
\end{align}
Next we bound the $L^{\frac{p}{2}}(\Omega)$-norm in the right-hand side of the previous inequality. Let $y\in(0,1)$ be fixed. We can easily check that there exists a constant $C>0$ only depending on $y$, such that for every $\lambda_{1},\lambda_{2}>0$ and $\eta,\xi\in\R^{d}$, and every multi-index $\textbf{i}=(i_{1},i_{2})$, with $1\leq i_{1},i_{2}\leq d$, 
\begin{align}\label{ineq:gbound}
\Abs{g_{\textbf{i},\lambda_{1}}(\eta)g_{\textbf{i},\lambda_{2}}(\xi)}
  &\leq (1+\lambda_{1}^{-1}\Norm{\eta}^2)(1+\lambda_{2}^{-1}\Norm{\xi}^2)\leq Ce^{\frac{y}{2}(\lambda_{1}^{-1}\Norm{\eta}^{2}+\lambda_{2}^{-1}\Norm{\xi}^{2})}.
\end{align}
From \eqref{ineq:Zp1} and \eqref{ineq:gbound}, it follows that there exists a constant $C>0$, not depending on $\varepsilon,T_{1},T_{2}$, such that
 \begin{multline}\label{ineq:Zp2}
\Norm{Z_{\varepsilon}}_{L^{p}(\Omega)}^{2}\\
\begin{aligned} 
  &\leq C \int_{\R_{+}^{2}}\int_{\mathcal{K}_{T_1,T_2}^2}e^{-2\theta-2\beta}\mu(s_{2}-s_{1},t_{1}-s_{1},t_{2}-s_{2})^2\\
	&\times(\lambda_{\varepsilon}(\theta,s_{1},t_{1})\lambda_{\varepsilon}(\beta,s_{2},t_{2}))^{-1}\\
	&\times\left\|p_{\frac{\lambda_{\varepsilon}(\theta,s_{1},t_{1})}{1-y}}(e^{-\theta}(B_{t_{1}}-B_{s_{1}}))p_{\frac{\lambda_{\varepsilon}(\beta,s_{2},t_{2})}{1-y}}(e^{-\beta}(B_{t_{2}}-B_{s_{2}}))\right\|_{L^{\frac{p}{2}}(\Omega)}ds_1 dt_1 ds_2 dt_2d\theta d\beta.
\end{aligned}
\end{multline}
Proceeding as in the proof of \eqref{eq:pdet0}, we can easily check that 
\begin{multline*}
\E\left[p_{\frac{\lambda_{\varepsilon}(\theta,s_{1},t_{1})}{1-y}}(e^{-\theta}(B_{t_{1}}-B_{s_{1}}))^{\frac{p}{2}}p_{\frac{\lambda_{\varepsilon}(\beta,s_{2},t_{2})}{1-y}}(e^{-\beta}(B_{t_{2}}-B_{s_{2}}))^{\frac{p}{2}}\right]\\
\begin{aligned} 
  &=(2\pi)^{-\frac{d(p-2)}{2}}\left(\frac{\lambda_{\varepsilon}(\theta,s_{1},t_{1})\lambda_{\varepsilon}(\beta,s_{2},t_{2})}{(1-y)^2}\right)^{-\frac{dp}{4}+\frac{d}{2}}\frac{2^{d}}{p^{d}}e^{d(\theta+\beta)}\\
	&\times\E\left[p_{\frac{2\lambda_{\varepsilon}(\theta,s_{1},t_{1})e^{2\theta}}{p(1-y)}}(B_{t_{1}}-B_{s_{1}})p_{\frac{2\lambda_{\varepsilon}(\beta,s_{2},t_{2})e^{2\beta}}{p(1-y)}}(B_{t_{2}}-B_{s_{2}})\right]\\
  &=(2\pi)^{-\frac{d(p-2)}{2}}\left(\frac{\lambda_{\varepsilon}(\theta,s_{1},t_{1})\lambda_{\varepsilon}(\beta,s_{2},t_{2})}{(1-y)^2}\right)^{-\frac{dp}{4}+\frac{d}{2}}\frac{2^{d}}{p^{d}}e^{d(\theta+\beta)}\\
	&\times\Abs{\frac{2}{p(1-y)}\Sqmatrix{\lambda_{\varepsilon}(\theta,s_{1},t_{1})e^{2\theta}}{0}{0}{\lambda_{\varepsilon}(\beta,s_{2},t_{2})e^{2\beta}}+\Sigma}^{-\frac{d}{2}},
\end{aligned}
\end{multline*}
where $\Sigma=\{\Sigma_{i,j}\}_{1\leq i,j\leq2}$, denotes the covariance matrix of $(B_{t_{1}}^{(1)}-B_{s_{1}}^{(1)},B_{t_{2}}^{(1)}-B_{s_{2}}^{(1)})$, whose components are given by $\Sigma_{1,1}=(t_{1}-s_{1})^{2H}$, $\Sigma_{1,2}=\Sigma_{2,1}=\mu(s_{2}-s_{1},t_{1}-s_{1},t_{2}-s_{2})$, and $\Sigma_{2,2}=(t_{2}-s_{2})^{2H}$. Therefore, there exists a constant $C>0$ only depending on $p$ and $d$, such that
\begin{multline*}
\E\left[p_{\frac{\lambda(\theta,s_{1},t_{1})}{1-y}}(e^{-\theta}(B_{t_{1}}-B_{s_{1}}))^{\frac{p}{2}}p_{\frac{\lambda(\beta,s_{2},t_{2})}{1-y}}(e^{-\beta}(B_{t_{2}}-B_{s_{2}}))^{\frac{p}{2}}\right]\\
\begin{aligned} 
  &\leq C(\lambda_{\varepsilon}(\theta,s_{1},t_{1})\lambda_{\varepsilon}(\beta,s_{2},t_{2}))^{-\frac{dp}{4}+\frac{d}{2}}\\
	&\times e^{d(\theta+\beta)}\Abs{\frac{2}{p(1-y)}\Sqmatrix{\lambda_{\varepsilon}(\theta,s_{1},t_{1})e^{2\theta}}{0}{0}{\lambda_{\varepsilon}(\beta,s_{2},t_{2})e^{2\beta}}+\Sigma}^{-\frac{d}{2}}.
\end{aligned}
\end{multline*}
Choosing $y<1-\frac{2}{p}$, so that $\frac{p(1-y)}{2}\Sigma\geq\Sigma$, we deduce that there exists a constant $C>0$ only depending on $p,y$ and $d$, such that 
\begin{multline*}
\E\left[p_{\frac{\lambda_{\varepsilon}(\theta,s_{1},t_{1})}{1-y}}(e^{-\theta}(B_{t_{1}}-B_{s_{1}}))^{\frac{p}{2}}p_{\frac{\lambda_{\varepsilon}(\beta,s_{2},t_{2})}{1-y}}(e^{-\beta}(B_{t_{2}}-B_{s_{2}}))^{\frac{p}{2}}\right]\\
\begin{aligned} 
 &\leq C(\lambda_{\varepsilon}(\theta,s_{1},t_{1})\lambda_{\varepsilon}(\beta,s_{2},t_{2}))^{-\frac{dp}{4}+\frac{d}{2}}\\
 &\times e^{d(\theta+\beta)}\Abs{\Sqmatrix{\lambda_{\varepsilon}(\theta,s_{1},t_{1})e^{2\theta}+(t_{1}-s_{1})^{2H}}{\mu(s_{2}-s_{1},t_{1}-s_{2},t_{2}-s_{2})}{\mu(s_{2}-s_{1},t_{1}-s_{2},t_{2}-s_{2})}{\lambda_{\varepsilon}(\beta,s_{2},t_{2})e^{2\beta}+(t_{2}-s_{2})^{2H}}}^{-\frac{d}{2}}.
\end{aligned}
\end{multline*}
Hence, by the multilinearity of the determinant function, 
\begin{multline}\label{ineq:prodp}
\E\left[p_{\frac{\lambda_{\varepsilon}(\theta,s_{1},t_{1})}{1-y}}(e^{-\theta}(B_{t_{1}}-B_{s_{1}}))^{\frac{p}{2}}p_{\frac{\lambda_{\varepsilon}(\beta,s_{2},t_{2})}{1-y}}(e^{-\beta}(B_{t_{2}}-B_{s_{2}}))^{\frac{p}{2}}\right]\\
\begin{aligned} 
 &\leq C(\lambda_{\varepsilon}(\theta,s_{1},t_{1})\lambda_{\varepsilon}(\beta,s_{2},t_{2}))^{-\frac{dp}{4}+\frac{d}{2}}\\
 &\times \Abs{\Sqmatrix{\lambda_{\varepsilon}(\theta,s_{1},t_{1})+e^{-2\theta}(t_{1}-s_{1})^{2H}}{e^{-2\beta}\mu(s_{2}-s_{1},t_{1}-s_{2},t_{2}-s_{2})}{e^{-2\theta}\mu(s_{2}-s_{1},t_{1}-s_{2},t_{2}-s_{2})}{\lambda_{\varepsilon}(\beta,s_{2},t_{2})+e^{-2\beta}(t_{2}-s_{2})^{2H}}}^{-\frac{d}{2}}.
\end{aligned}
\end{multline}
By relation \eqref{eq:lamdadef}, we have that $\lambda_{\varepsilon}(\theta,s,t)+e^{-2\theta}(t-s)^{2H}=\varepsilon+(t-s)^{2H}$ for every $\theta,s,t>0$. As a consequence,  relation \eqref{ineq:prodp} can be written as
\begin{multline}\label{eq:pbound}
\E\left[p_{\frac{\lambda_{\varepsilon}(\theta,s_{1},t_{1})}{1-y}}(e^{-\theta}(B_{t_{1}}-B_{s_{1}}))^{\frac{p}{2}}p_{\frac{\lambda_{\varepsilon}(\beta,s_{2},t_{2})}{1-y}}(e^{-\beta}(B_{t_{2}}-B_{s_{2}}))^{\frac{p}{2}}\right]\\
\begin{aligned} 
  &\leq C(\lambda_{\varepsilon}(\theta,s_{1},t_{1})\lambda_{\varepsilon}(\beta,s_{2},t_{2}))^{-\frac{dp}{4}+\frac{d}{2}}\\
	&\times \bigg(\varepsilon^{2}+\varepsilon((t_{1}-s_{1})^{2H}+(t_{2}-s_{2})^{2H})+(t_{1}-s_{1})^{2H}(t_{2}-s_{2})^{2H}-e^{-2\beta-2\theta}\mu^{2}\bigg)^{-\frac{d}{2}}\\
	&\leq C(\lambda_{\varepsilon}(\theta,s_{1},t_{1})\lambda_{\varepsilon}(\beta,s_{2},t_{2}))^{-\frac{dp}{4}+\frac{d}{2}}\Theta_{\varepsilon}(s_{2}-s_{1},t_{1}-s_{1},t_{2}-s_{2})^{-\frac{d}{2}},
\end{aligned}
\end{multline}
where $\Theta_{\varepsilon}(x,u_{1},u_{2})$ is defined by \eqref{eq:Ldef}. From \eqref{eq:lamdadef}, \eqref{ineq:Zp2} and \eqref{eq:pbound}, it follows that 
\begin{align}\label{ineq:Zp3}
\Norm{Z_{\varepsilon}}_{L^{p}(\Omega)}^{2}
  &\leq C \int_{\R_{+}^{2}}\int_{\Sc_{T_1,T_2}^2}e^{-2\theta-2\beta}\mu(s_{2}-s_{1},t_{1}-s_{1},t_{2}-s_{2})^2\nonumber\\
	&\times((\varepsilon+(1-e^{-2\theta})(t_{1}-s_{1})^{2H})(\varepsilon+(1-e^{-2\beta})(t_{2}-s_{2})^{2H}))^{-1-\frac{d}{2}+\frac{d}{p}}\nonumber\\
	&\times\Theta_{\varepsilon}(s_{2}-s_{1},t_{1}-s_{1},t_{2}-s_{2})^{-\frac{d}{p}}ds_1 dt_1 ds_2 dt_2d\theta d\beta.
\end{align}
Changing the coordinates $(s_{1},t_{1},s_{2},t_{2})$ by $(s_{1},x:=s_{2}-s_{1},u_{1}:=t_{1}-s_{1},u_{2}:=t_{2}-s_{2})$ in \eqref{ineq:Zp3}, we get
\begin{align*}
\Norm{Z_{\varepsilon}}_{L^{p}(\Omega)}^{2}
  &\leq 2C\int_{\R_{+}^{2}}e^{-2\theta-2\beta}\int_{[0,T_{2}]^{3}}\int_{(T_{1}-u_{1})_{+}\vee(T_{1}-x-u_{2})_{+}}^{(T_{2}-u_{1})_{+}\wedge(T_{2}-x-u_{2})_{+}}ds_{1}\nonumber\\
	&\times\mu(x,u_{1},u_{2})^2((\varepsilon+(1-e^{-2\theta})u_{1}^{2H})(\varepsilon+(1-e^{-2\beta})u_{2}^{2H}))^{-1-\frac{d}{2}+\frac{d}{p}}\nonumber\\
	&\times\Theta_{\varepsilon}(x,u_{1},u_{2})^{-\frac{d}{p}}dx du_{1}du_{2}d\theta d\beta.
\end{align*}
Integrating the variable $s_{1}$, and making the change of variables $\eta:=1-e^{-2\theta}$, and $\xi:=1-e^{-2\beta}$, we deduce that there exists a constant $C>0$, such that
\begin{align}\label{ineq:Ze}
\Norm{Z_{\varepsilon}}_{L^{p}(\Omega)}^{2}
  &\leq C(T_{2}-T_{1})\int_{[0,T_{2}]^{3}}\mu(x,u_{1},u_{2})^2\Theta_{\varepsilon}(x,u_{1},u_{2})^{-\frac{d}{p}}\nonumber\\
	&\times\int_{[0,1]^{2}}((\varepsilon+\eta u_{1}^{2H})(\varepsilon+\xi u_{2}^{2H}))^{-1-\frac{d}{2}+\frac{d}{p}}d\eta d\xi dx du_{1}du_{2}.
\end{align}
Changing the coordinates $(x,u_{1},u_{2})$ by $(\varepsilon^{-\frac{1}{2H}}x,\varepsilon^{-\frac{1}{2H}}u_{1},\varepsilon^{-\frac{1}{2H}}u_{2})$ in \eqref{ineq:Ze}, and using the fact that $\Theta_{\varepsilon}(\varepsilon^{-\frac{1}{2H}}x,\varepsilon^{-\frac{1}{2H}}u_{1},\varepsilon^{-\frac{1}{2H}}u_{2})=\varepsilon^{2}\Theta_{1}(x,u_{1},u_{2})$, we get
\begin{align*}
\Norm{\varepsilon^{\frac{d}{2}-\frac{3}{4H}}Z_{\varepsilon}}_{L^{p}(\Omega)}^{2}
  &\leq C(T_{2}-T_{1})\int_{\R_{+}^{3}}\mu(x,u_{1},u_{2})^2\Theta_{1}(x,u_{1},u_{2})^{-\frac{d}{p}}\nonumber\\
	&\times \int_{[0,1]^2}((1+\eta u_{1}^{2H})(1+\xi u_{2}^{2H}))^{-1-\frac{d}{2}+\frac{d}{p}}d\eta d\xi dx du_{1}du_{2}.
\end{align*}
Integrating the variables $\eta$ and $\xi$, we obtain 
\begin{align*}
\Norm{\varepsilon^{\frac{d}{2}-\frac{3}{4H}}Z_{\varepsilon}}_{L^{p}(\Omega)}^{2}
  &\leq C(1+\frac{d}{2}-\frac{d}{p})(T_{2}-T_{1})\int_{\R_{+}^{3}}\frac{\mu(x,u_{1},u_{2})^2}{u_{1}^{2H}u_{2}^{2H}}\Theta_{1}(x,u_{1},u_{2})^{-\frac{d}{p}}\nonumber\\
	&\times (1-(1+u_{1}^{2H})^{-\frac{d}{2}+\frac{d}{p}})(1-(1+u_{2}^{2H})^{-1-\frac{d}{2}+\frac{d}{p}})dx du_{1}du_{2}.
\end{align*}
Hence, choosing $p>2$, we deduce that there exists a constant $C$ only depending on $H,d$ and $p$, such that
\begin{align}\label{ineq:Zp6}
\Norm{\varepsilon^{\frac{d}{2}-\frac{3}{4H}}Z_{\varepsilon}}_{L^{p}(\Omega)}^{2}
  &\leq C(T_{2}-T_{1})\int_{\R_{+}^{3}}\frac{\mu(x,u_{1},u_{2})^2}{u_{1}^{2H}u_{2}^{2H}}\Theta_{1}(x,u_{1},u_{2})^{-\frac{d}{p}}dx du_{1}du_{2}.
\end{align}
Since $Hd>\frac{3}{2}$, we can choose $p$ so that $2<p<\frac{4Hd}{3}$. For this choice of $p$, the integral in the right-hand side of \eqref{ineq:Zp6} is finite  by Lemma \ref{lem:inegraltightness}. Therefore, from \eqref{ineq:Zp6}, it follows that there exists a constant $C>0$, independent of $T_{1},T_{2}$ and $\varepsilon$, such that $\Norm{\varepsilon^{\frac{d}{2}-\frac{3}{4H}}Z_{\varepsilon}}_{L^{p}(\Omega)}^{2}\leq C(T_{2}-T_{1})$, which in turn implies that 
\begin{align}\label{ineq:Zpfinal}
\E\left[\Abs{\varepsilon^{\frac{d}{2}-\frac{3}{4H}}Z_{\varepsilon}}^{p}\right]
  &\leq C(T_{2}-T_{1})^{\frac{p}{2}}.
\end{align} 
Relation \eqref{eq:billbound} then follows from \eqref{ineq:Zpfinal}. This finishes the proof of Theorem \ref{Teo:convergence}.\\~\\

\noindent\textbf{Proof of Theorem \ref{Teo:convergenceHermite}}\\
\noindent Now we proceed with the proof of Theorem \ref{Teo:convergenceHermite}, in which we will prove \eqref{conv:I2L2} and \eqref{conv:I2} in the case $H>\frac{3}{4}$. In order to prove \eqref{conv:I2L2}, it suffices to show that for every $T>0$,
\begin{align}\label{conv:L22}
\varepsilon^{\frac{d}{2}-\frac{3}{2H}+1}(I_{T}^{\varepsilon}-\E\left[I_{T}^{\varepsilon}\right]-J_{2}(I_{T}^{\varepsilon}))
  &\stackrel{L^{2}(\Omega)}{\rightarrow}0,
\end{align}
 and 
\begin{align}\label{conv:fdd2}
\varepsilon^{\frac{d}{2}-\frac{3}{2H}+1}J_{2}(I_{T}^{\varepsilon})
  &\stackrel{L^{2}(\Omega)}{\rightarrow}-\Lambda \sum_{j=1}^{d}X_{T}^{j},
\end{align}
as $\varepsilon\rightarrow0$. Relation \eqref{conv:L22} follows from Lemma \ref{lem:chaos4tozero}. In order to prove the convergence \eqref{conv:fdd2} we proceed as follows. Using \eqref{eq:kerneldef}, we can easily check that
\begin{align*}
J_{2}(I_{T}^{\varepsilon})
  &=-\frac{(2\pi)^{-\frac{d}{2}}}{2}\sum_{j=1}^{d}\int_{0}^{T}\int_{0}^{T-u}(\varepsilon+u^{2H})^{-\frac{d}{2}-1}u^{2H}H_{2}\left(\frac{B_{s+u}^{(j)}-B_{s}^{(j)}}{u^{H}}\right)dsdu.
\end{align*}
Making the change of variable $v:=\varepsilon^{-\frac{1}{2H}}u$, we get 
\begin{align}\label{eq:J2}
\varepsilon^{\frac{d}{2}-\frac{3}{2H}+1}J_{2}(I_{T}^{\varepsilon})
  &=-\frac{(2\pi)^{-\frac{d}{2}}}{2}\sum_{j=1}^{d}\int_{0}^{\varepsilon^{-\frac{1}{2H}}T}\int_{0}^{T-\varepsilon^{\frac{1}{2H}}v}(1+v^{2H})^{-\frac{d}{2}-1}v^{2H}\varepsilon^{1-\frac{1}{H}}H_{2}\left(\frac{B_{s+\varepsilon^{\frac{1}{2H}}v}^{(j)}-B_{s}^{(j)}}{\sqrt{\varepsilon}v^{H}}\right)dv\nonumber\\
	&=-\frac{(2\pi)^{-\frac{d}{2}}}{2}\sum_{j=1}^{d}\int_{0}^{\varepsilon^{-\frac{1}{2H}}T}(1+u^{2H})^{-\frac{d}{2}-1}u^{2}I_{2}(\varphi_{j,T-\varepsilon^{\frac{1}{2H}}u}^{\varepsilon^{\frac{1}{2H}}u})du,
\end{align}
where the kernel $\varphi_{j,T-\varepsilon^{\frac{1}{2H}}u}^{\varepsilon^{\frac{1}{2H}}u}$ is defined by \eqref{eq:phidef}. 
From \eqref{eq:J2}, it follows that for every $\varepsilon,\eta>0$, 
\begin{multline}\label{eq:J2Cauchy}
\E\left[\varepsilon^{\frac{d}{2}-\frac{3}{2H}-1}J_{2}(I_{T}^{\varepsilon})\eta^{\frac{d}{2}-\frac{3}{2H}-1}J_{2}(I_{T}^{\eta})\right]\\
\begin{aligned}
  &=\frac{(2\pi)^{-d}}{2}\sum_{j=1}^{d}\int_{0}^{\varepsilon^{-\frac{1}{2H}}T}\int_{0}^{\eta^{-\frac{1}{2H}}T}(1+u_{1}^{2H})^{-\frac{d}{2}-1}(1+u_{2}^{2H})^{-\frac{d}{2}-1}\\
	&\times (u_{1}u_{2})^{2}\Ip{\varphi_{j,T-\varepsilon^{\frac{1}{2H}}u_{1}}^{\varepsilon^{\frac{1}{2H}}u_{1}},\varphi_{j,T-\eta^{\frac{1}{2H}}u_{1}}^{\eta^{\frac{1}{2H}}u_{2}}}_{(\Hg^{d})^{\otimes 2}}d u_{1} du_{2}.
\end{aligned}
\end{multline}
By \eqref{eq:Ipphi}, 
\begin{align}\label{conv:phiinnerproduct}
\lim_{\varepsilon\rightarrow0}\Ip{\varphi_{j,T-\varepsilon^{\frac{1}{2H}}u_{1}}^{\varepsilon^{\frac{1}{2H}}u_{1}},\varphi_{j,T-\eta^{\frac{1}{2H}}u_{1}}^{\eta^{\frac{1}{2H}}u_{2}}}_{(\Hg^{d})^{\otimes 2}}
  &=H^{2}(2H-1)^2\int_{[0,T]^{2}}\Abs{s_{1}-s_{2}}^{4H-4}ds_{1}ds_{2}\nonumber\\
	&=\frac{H^2(2H-1)}{4H-3}T^{4H-2}.
\end{align}
On the other hand, by  \eqref{ineq:supIpphi}, there exists a constant $C_{H,T}>0$, only depending on $H$ and $T$, such that 
\begin{align*}
0\leq \Ip{\varphi_{j,T-\varepsilon^{\frac{1}{2H}}u_{1}}^{\varepsilon^{\frac{1}{2H}}u_{1}},\varphi_{j,T-\eta^{\frac{1}{2H}}u_{1}}^{\eta^{\frac{1}{2H}}u_{2}}}_{(\Hg^{d})^{\otimes 2}}
  &\leq C_{H,K}.
\end{align*}
Hence, using the pointwise convergence \eqref{conv:phiinnerproduct}, we can apply the dominated convergence theorem to \eqref{eq:J2Cauchy}, in order to obtain
\begin{align*}
\lim_{\varepsilon,\nu\rightarrow0}\E\left[\varepsilon^{\frac{d}{2}-\frac{3}{2H}-1}J_{2}(I_{T}^{\varepsilon})\eta^{\frac{d}{2}-\frac{3}{2H}-1}J_{2}(I_{T}^{\eta})\right]
  &=\frac{d(2\pi)^{-d}\Lambda^{2}H^2(2H-1)T^{4H-2}}{2(4H-3)},
\end{align*}
where the constant $\Lambda$ is defined by \eqref{eq:Lambdadef}. From the previous identity, it follows that 
$\varepsilon^{\frac{d}{2}-\frac{3}{2H}-1}J_{2}(I_{T}^{\varepsilon})$ 
converges to some $\widetilde{h}_{T}\in (\Hg^{d})^{\otimes 2}$, as $\varepsilon\rightarrow0$.\\~\\
Recall that the element $\pi_{T}^{j}\in(\Hg^{d})^{\otimes d}$, is defined as the limit in $(\Hg^{d})^{\otimes 2}$, as $\varepsilon\rightarrow0$, of $\varphi_{j,T}^{\varepsilon}$, and is characterized by  relation \eqref{eq:picharact}. In order to prove \eqref{conv:fdd2}, it suffices to show that $\widetilde{h}_{T}=\Lambda\sum_{j=1}^{d}\pi_{T}^{j}$, or equivalently, that
$$\Ip{\widetilde{h}_{T},f_{1}\otimes f_{2}}_{(\Hg^{d})^{\otimes 2}}=-\Lambda\sum_{j=1}^{d}\Ip{\pi_{T}^{j},f_{1}\otimes f_{2}}_{(\Hg^{d})^{\otimes 2}},$$ 
for vectors of step functions with compact support $f_{i}=(f_{i}^{(1)},\dots, f_{i}^{(d)})\in\Hg^{d}$, $i=1,2$. By \eqref{eq:J2}, 
\begin{align}\label{eq:actionh}
\lim_{\varepsilon\rightarrow0}\Ip{\widetilde{h}_{T},f_{1}\otimes f_{2}}_{(\Hg^{d})^{\otimes 2}}
  &=\lim_{\varepsilon\rightarrow0}-\frac{(2\pi)^{-\frac{d}{2}}}{2}\int_{0}^{\varepsilon^{-\frac{1}{2H}}T}(1+u^{2H})^{-\frac{d}{2}}u^{2}\Ip{\varphi_{j,T-\varepsilon^{\frac{1}{2H}u}}^{\varepsilon^{\frac{1}{2H}}u},f_{1}\otimes f_{2}}_{(\Hg^{d})^{\otimes 2}}du.
\end{align}
Proceeding as in the proof of \eqref{conv:phiinnerproduct}, we can easily check that 
\begin{align*}
\lim_{\varepsilon\rightarrow0}\Ip{\varphi_{j,T-\varepsilon^{\frac{1}{2H}u}}^{\varepsilon^{\frac{1}{2H}}u},f_{1}\otimes f_{2}}_{(\Hg^{d})^{\otimes 2}}
  &=-H^{2}(2H-1)^{2}\sum_{j=1}^{d}\int_{0}^{T}\prod_{i=1,2}\int_{0}^{T}\Abs{s-\eta}^{2H-2}f_{i}^{(j)}(\eta)d\eta ds.
\end{align*}
Moreover, by \eqref{ineq:supIpphi}, 
\begin{align*}
\Abs{\Ip{\varphi_{j,T-\varepsilon^{\frac{1}{2H}u}}^{\varepsilon^{\frac{1}{2H}}u},f_{1}\otimes f_{2}}_{(\Hg^{d})^{\otimes 2}}}
  &\leq \Norm{\varphi_{j,T-\varepsilon^{\frac{1}{2H}u}}^{\varepsilon^{\frac{1}{2H}}u}}_{(\Hg^{d})^{\otimes 2}}\Norm{f_{1}}_{\Hg^{d}}\Norm{f_{2}}_{\Hg^{d}}\leq C_{H,T}\Norm{f_{1}}_{\Hg^{d}}\Norm{f_{2}}_{\Hg^{d}},
\end{align*}
for some constant $C_{H,T}>0$ only depending on $T$ and $H$. Therefore, applying the dominated convergence theorem in \eqref{eq:actionh}, we get
\begin{align}\label{eq:actionh2}
\lim_{\varepsilon\rightarrow0}\Ip{\widetilde{h}_{T},f_{1}\otimes f_{2}}_{(\Hg^{d})^{\otimes 2}}
  &=-\Lambda H^{2}(2H-1)^{2}\sum_{j=1}^{d}\int_{0}^{T}\prod_{i=1,2}\int_{0}^{T}\Abs{s-\eta}^{2H-2}f_{i}^{(j)}(\eta)d\eta ds,
\end{align}
and from the characterization \eqref{eq:picharact}, we conclude that $\widetilde{h}_{T}=-\Lambda\sum_{j=1}^{d}\pi_{T}^{j}$, as required. This finishes the proof of \eqref{conv:fdd2}, which, by \eqref{conv:L22}, implies that the convergence \eqref{conv:I2L2}.\\

It only remains to prove \eqref{conv:I2}. By \eqref{conv:I2L2}, it suffices to show the tightness property for $\varepsilon^{\frac{d}{2}-\frac{3}{2H}+1}(I_{T}^{\varepsilon}-\E\left[I_{T}^{\varepsilon}\right])$, which, as in the proof of \eqref{conv:I1}, can be reduced to proving that there exists $p>2$, such that for every $0\leq T_{1}\leq T_{2}\leq K$, 
\begin{align}\label{ineq:Zpfinal2}
\E\left[\Abs{\varepsilon^{\frac{d}{2}-\frac{3}{2H}+1}Z_{\varepsilon}}^{p}\right]
  &\leq C(T_{2}-T_{1})^{\frac{p}{2}},
\end{align} 
where $Z_{\varepsilon}$ is defined by \eqref{eq:Zdef}, and $C$ is some constant only depending on $d,H,K$ and $p$. Changing the coordinates $(x,u_{1},u_{2})$ by $(x,\varepsilon^{-\frac{1}{2H}}u_{1},\varepsilon^{-\frac{1}{2H}}u_{2})$ in \eqref{ineq:Ze}, and using the fact that 
$$\Theta_{\varepsilon}(x,\varepsilon^{\frac{1}{2H}}u_{1},\varepsilon^{\frac{1}{2H}}u_{2})=\varepsilon^{2}\Theta_{1}(\varepsilon^{-\frac{1}{2H}}x,u_{1},u_{2}),$$
we can easily check that 
\begin{align*}
\Norm{\varepsilon^{\frac{d}{2}-\frac{3}{2H}+1}Z_{\varepsilon}}_{L^{p}(\Omega)}^{2}
  &\leq C(T_{2}-T_{1})\int_{\R_{+}^2}\int_{0}^{T_{2}}\varepsilon^{-\frac{2}{H}}\mu(x,\varepsilon^{\frac{1}{2H}}u_{1},\varepsilon^{\frac{1}{2H}}u_{2})^2\nonumber\\
	&\times\Theta_{1}(\varepsilon^{-\frac{1}{2H}}x,u_{1},u_{2})^{-\frac{d}{p}}\int_{[0,1]^{2}}((1+\eta u_{1}^{2H})(1+\xi u_{2}^{2H}))^{-1-\frac{d}{2}+\frac{d}{p}}d\eta d\xi dx du_{1}du_{2},
\end{align*}
and hence, if $p>2$, we obtain 
\begin{align}\label{ineq:Ze2}
\Norm{\varepsilon^{\frac{d}{2}-\frac{3}{2H}+1}Z_{\varepsilon}}_{L^{p}(\Omega)}^{2}
  &\leq C(T_{2}-T_{1})\int_{\R_{+}^2}\int_{0}^{T_{2}}\varepsilon^{-\frac{2}{H}}\mu(x,\varepsilon^{\frac{1}{2H}}u_{1},\varepsilon^{\frac{1}{2H}}u_{2})^2(u_{1}u_{2})^{-2H}\nonumber\\
	&\times\Theta_{1}(\varepsilon^{-\frac{1}{2H}}x,u_{1},u_{2})^{-\frac{d}{p}}dx du_{1}du_{2}.
\end{align}
By Lemma \ref{lem:inegraltightness2}, if $T_{1},T_{2}\in[0,K]$, for some $K>0$, the integral in the right-hand side of the previous inequality is bounded by a constant only depending on $H,d,p$ and $K$. Relation \eqref{ineq:Zpfinal2} then follows from \eqref{ineq:Ze2}. This finishes the proof of the tightness property for $\varepsilon^{\frac{d}{2}-\frac{3}{2H}+1}(I_{T}^{\varepsilon}-\E\left[I_{T}^{\varepsilon}\right])$ in the case $H>\frac{3}{4}$.\\

\noindent\textbf{Proof of Theorem \ref{Teo:convergencelog}}\\
\noindent Finally we prove Theorem \ref{Teo:convergencelog}. First we show the convergence of the finite dimensional distributions, namely,   that for every $r\in\N$ and $T_{1},\dots, T_{r}\geq0$ fixed, it holds 
\begin{align}\label{eq:fddlog}
\frac{\varepsilon^{\frac{d}{2}-1}}{\sqrt{\log(1/\varepsilon)}}\left((I_{T_{1}}^{\varepsilon},\dots, I_{T_{r}}^{\varepsilon})-\E\left[(I_{T_{1}}^{\varepsilon},\dots, I_{T_{r}}^{\varepsilon})\right]\right)\stackrel{Law}{\rightarrow} \rho (W_{T_{1}},\dots, W_{T_{r}}),
\end{align}
where $\rho$ is defined by \eqref{eq:rhodeflog}.  Consider the random variable $\widetilde{J}_{T}^{\varepsilon}$ introduced in \eqref{eq:Jtildedef}.
By Lemma \ref{lem:chaos4tozero2}, we have 
\begin{align}\label{conv:secondchaosapprox}
\lim_{\varepsilon\rightarrow0}\frac{\varepsilon^{\frac{d}{2}-1}}{\sqrt{\log(1/\varepsilon)}}\Norm{I_{T}^{\varepsilon}-\E\left[I_{T}^{\varepsilon}\right]-I_{2}(h_{2,T}^{\varepsilon})}_{L^{2}(\Omega)}=0,
\end{align}
and by Lemma \ref{lem:finiteFGintegrallogaux}  
\begin{align}\label{conv:truncatedsecondchaos}
\lim_{\varepsilon\rightarrow0}\frac{\varepsilon^{\frac{d}{2}-1}}{\sqrt{\log(1/\varepsilon)}}\Norm{I_{2}(h_{2,T}^{\varepsilon})-\widetilde{J}_{T}^{\varepsilon}}_{L^{2}(\Omega)}=0.
\end{align} 
Consequently, 
$$\lim_{\varepsilon\rightarrow0}\frac{\varepsilon^{\frac{d}{2}-1}}{\sqrt{\log(1/\varepsilon)}}\Norm{I_{T}^{\varepsilon}-\E\left[I_{T}^{\varepsilon}\right]-\widetilde{J}_{T}^{\varepsilon}}_{L^{2}(\Omega)}=0,$$ 
and hence,  relation \eqref{eq:fddlog} is equivalent to 
\begin{align}\label{eq:fddlogp}
\frac{\varepsilon^{\frac{d}{2}-1}}{\sqrt{\log(1/\varepsilon)}}\left(\widetilde{J}_{T_{1}}^{\varepsilon},\dots, \widetilde{J}_{T_{r}}^{\varepsilon}\right)\stackrel{Law}{\rightarrow} \rho (W_{T_{1}},\dots, W_{T_{r}}).
\end{align}
By the Peccati-Tudor criterion, the convergence \eqref{eq:fddlogp} holds provided that $\widetilde{J}_{t}^{\varepsilon}$ satisfies the following conditions:
\begin{enumerate}
\item[(i)] For every $1\leq i,j\leq r$, 
$$\frac{\varepsilon^{d-2}}{\log(1/\varepsilon)}\E\left[\widetilde{J}_{T_{i}}^{\varepsilon}\widetilde{J}_{T_{j}}^{\varepsilon}\right]\rightarrow \rho^2(T_{i}\wedge T_{j}),~~~~\text{ as }~~\varepsilon\rightarrow0.$$
\item[(ii)] For all $i=1,\dots, r$, the random variables $\frac{\varepsilon^{\frac{d}{2}-1}}{\sqrt{\log(1/\varepsilon)}}\widetilde{J}_{T_{i}}^{\varepsilon}$ converge in law to a centered Gaussian distribution as $\varepsilon\rightarrow0$.
\end{enumerate}
Relation (i) follows from relation \eqref{conv:truncatedsecondchaos}, as well as Theorem \ref{teo:covchaoslog}. Hence, it suffices to check (ii). To this end, consider the following Riemann sum approximation for  $\widetilde{J}_{T}^{\varepsilon}$
\begin{align}\label{eq:RTMdef}
R_{T,M}^{\varepsilon}
  &:=-\frac{c_{log}\varepsilon^{\frac 23-\frac{d}{2}}}{2^{M}}\sum_{k=2}^{M2^{M}}
	\int_{0}^{T}
	\sum_{j=1}^{d}  \frac{u(k)^{\frac 32}} {(1+u(k)^{\frac 32}}
	H_{2}\left(\frac{B_{s+\varepsilon^{\frac{2}{3}}u(k)}^{(j)}-B_{s}^{(j)}}{\sqrt{\varepsilon}u(k)^{\frac{3}{4}}}\right)ds,
\end{align}
where $c_{log}:=\frac{(2\pi)^{-\frac{d}{2}}}{2}$ and  $u(k):=\frac{k}{2^{M}}$, for $k=2,\dots, M2^{M}$. We will prove that $\frac{\varepsilon^{\frac{d}{2}-1}}{\sqrt{\log(1/\varepsilon)}}(R_{T,M}^{\varepsilon}-\widetilde{J}_{T}^{\varepsilon})$ converges to zero, uniformly in $\varepsilon\in(0,1/e)$, and $\frac{\varepsilon^{\frac{d}{2}-1}}{\sqrt{\log(1/\varepsilon)}}R_{T,M}^{\varepsilon}\stackrel{Law}{\rightarrow} T\mathcal{N}(0,\widetilde{\rho}_{M}^2)$ as $\varepsilon\rightarrow0$ for some constant $\widetilde{\rho}_{M}^2$ satisfying 
$\widetilde{\rho}_{M}^2\rightarrow\rho^{2}$ as $M\rightarrow\infty$. The result will then follow by a standard approximation argument. We will separate the argument in the following steps.\\

\noindent \textit{Step I}\\
We prove that $\frac{\varepsilon^{\frac{d}{2}-1}}{\sqrt{\log(1/\varepsilon)}}(R_{T,M}^{\varepsilon}-\widetilde{J}_{T}^{\varepsilon})\rightarrow 0$ in $L^{2}(\Omega)$ as $M\rightarrow\infty$ uniformly in $\varepsilon\in(0,1/e)$, namely, 
\begin{align}\label{limunifR}
\lim_{M\rightarrow\infty}\sup_{\varepsilon\in(0,1/e)}\frac{\varepsilon^{\frac{d}{2}-1}}{\sqrt{\log(1/\varepsilon)}}\Norm{R_{T,M}^{\varepsilon}-\widetilde{J}_{T}^{\varepsilon}}_{L^{2}(\Omega)}=0.
\end{align}
For $\varepsilon\in(0,1/e)$ fixed, we decompose the term $\widetilde{J}_{T}^{\varepsilon}$ as
\begin{align}\label{eq:JT12}
\widetilde{J}_{T}^{\varepsilon}
  &=\widetilde{J}_{T,1}^{\varepsilon,M}+\widetilde{J}_{T,2}^{\varepsilon,M},
\end{align}
where 
\begin{align*}
\widetilde{J}_{T,1}^{\varepsilon,M}
  &:=-c_{log}\varepsilon^{\frac 32-\frac{d}{2}}\int_{0}^{T}\int_{2^{-M}}^{M}
	\sum_{j=1}^{d}\frac{u^{\frac{3}{2}}}{(1+u^{\frac{3}{2}})^{\frac{d}{2}+1}}H_{2}\left(\frac{B_{s+\varepsilon^{\frac{2}{3}}u}^{(j)}-B_{s}^{(j)}}{\sqrt{\varepsilon}u^{\frac{3}{4}}}\right)duds
\end{align*}
and 
\begin{align*}
\widetilde{J}_{T,2}^{\varepsilon,M}
  &:=-c_{log}\varepsilon^{\frac 32-\frac{d}{2}}\int_{0}^{T}\int_{0}^{\infty}\Indi{(0,2^{-M})\cup(M,\infty)}(u) \sum_{j=1}^{d}\frac{u^{\frac{3}{2}}}{(1+u^{\frac{3}{2}})^{\frac{d}{2}+1}}H_{2}\left(\frac{B_{s+\varepsilon^{\frac{2}{3}}u}^{(j)}-B_{s}^{(j)}}{\sqrt{\varepsilon}u^{\frac{3}{4}}}\right)duds.
\end{align*}
From \eqref{eq:JT12}, we deduce that the relation \eqref{limunifR} is equivalent to 
\begin{align}\label{eq:limunif012}
\lim_{M\rightarrow\infty}\sup_{\varepsilon\in(0,1/e)}\frac{\varepsilon^{\frac{d}{2}-1}}{\sqrt{\log(1/\varepsilon)}}\Norm{R_{T,M}^{\varepsilon}-\widetilde{J}_{T,1}^{\varepsilon,M}}_{L^{2}(\Omega)}=0,
\end{align}
provided that 
\begin{align}\label{eq:limunif01}
\lim_{M\rightarrow\infty}\sup_{\varepsilon\in(0,1/e)}\frac{\varepsilon^{\frac{d}{2}-1}}{\sqrt{\log(1/\varepsilon)}}\Norm{\widetilde{J}_{T,2}^{\varepsilon,M}}_{L^{2}(\Omega)}=0.
\end{align}
To prove \eqref{eq:limunif01} we proceed as follows. First we use the relation \eqref{eq:Hermiteisometry}
to write
\begin{align*}
\frac{\varepsilon^{d-2}}{\log(1/\varepsilon)}\Norm{\widetilde{J}_{T,2}^{\varepsilon,M}}_{L^{2}(\Omega)}^2
  &=\frac{2dc_{log}^2}{\log(1/\varepsilon)}\int_{[0,T]^2}\int_{[0,\varepsilon^{-\frac{2}{3}}T]}\prod_{i=1,2}\Indi{(0,2^{-M})\cup(M,\infty)}(u_{i})\\
	&\times \psi(u_1,u_2)\varepsilon^{-8/3}\mu(s_{2}-s_{1},\varepsilon^{\frac{2}{3}}u_{1},\varepsilon^{\frac{2}{3}}u_{1})^2ds_{1}ds_{2}du_{1}du_{2},
\end{align*}
where $\psi(u_1,u_2)$ is defined by \eqref{eq:Kdef}. Changing the coordinates $(s_{1},s_{2},u_{1},u_{2})$ by $(s:=s_{1},x:=\varepsilon^{-\frac{2}{3}}(s_{2}-s_{1}),u_{1},u_{2})$ when $s_{1}\leq s_{2}$, and by $(s:=s_{2},x:=\varepsilon^{-\frac{2}{3}}(s_{1}-s_{2}),u_{1},u_{2})$ when $s_{1}\geq  s_{2}$, integrating the variable $s$, and using the identity $\mu(\varepsilon^{\frac{2}{3}}x,\varepsilon^{\frac{2}{3}}u_{1},\varepsilon^{\frac{2}{3}}u_{2})^2=\varepsilon^2\mu(x,u_{1},u_{2})$, we get 
\begin{align}\label{eq:renormJtilde}
\frac{\varepsilon^{d-2}}{\log(1/\varepsilon)}\Norm{\widetilde{J}_{T,2}^{\varepsilon,M}}_{L^{2}(\Omega)}^2
	&\leq\frac{4Tdc_{log}^2}{\log(1/\varepsilon)}\int_{[0,\varepsilon^{-\frac{2}{3}}T]^3}\prod_{i=1,2}\Indi{(0,2^{-M})\cup(M,\infty)}(u_{i})
	G_{1,x}^{(1)}(u_{1},u_{2})dx du_{1}du_{2},
\end{align}
where the function $G_{1,x}^{(1)}(u_{1},u_{2})$ is defined by \eqref{eq:Gdef}. Define the regions $\Sc_{i}$ by \eqref{eq:Scdef}. Splitting the domain of integration of the right-hand side of \eqref{eq:renormJtilde} into $[0,T]^{3}=\bigcup_{i=1}^{3}([0,\varepsilon^{-\frac{2}{3}}T]^3\cap \Sc_{i})$,  we obtain
\begin{align*}
\frac{\varepsilon^{d-2}}{\log(1/\varepsilon)}\Norm{\widetilde{J}_{T,2}^{\varepsilon,M}}_{L^{2}(\Omega)}^2
	&\leq\frac{4Tdc_{log}^2}{\log(1/\varepsilon)}\sum_{i=1}^{3}\int_{[0,\varepsilon^{-\frac{2}{3}}T]^3}\Indi{\Sc_{i}}(x,u_{1},u_{2})\\
	&\times \prod_{i=1,2}\Indi{(0,2^{-M})\cup(M,\infty)}(u_{i})G_{1,x}^{(1)}(u_{1},u_{2})dx du_{1}du_{2},
\end{align*}
and hence, dropping the normalization term $\frac{1}{\log(1/\varepsilon)}$ in the regions $\Sc_{1},\Sc_{2}$, we obtain
\begin{align*}
\frac{\varepsilon^{d-2}}{\log(1/\varepsilon)}\Norm{\widetilde{J}_{T,2}^{\varepsilon,M}}_{L^{2}(\Omega)}^2
  &\leq \frac{4Tdc_{log}^2}{\log(1/\varepsilon)}\int_{[0,\varepsilon^{-\frac{2}{3}}T]^3}\Indi{\Sc_{3}}(x,u_{1},u_{2})\\
	&\times \prod_{i=1,2}\Indi{(0,2^{-M})\cup(M,\infty)}(u_{i})G_{1,x}^{(1)}(u_{1},u_{2})dx du_{1}du_{2}\\
	&+4Tdc_{log}^2\sum_{i=1}^{2}\int_{[0,\varepsilon^{-\frac{2}{3}}T]^3}\Indi{\Sc_{i}}(x,u_{1},u_{2})\\
	&\times \prod_{i=1,2}\Indi{(0,2^{-M})\cup(M,\infty)}(u_{i})G_{1,x}^{(1)}(u_{1},u_{2})dx du_{1}du_{2}.
\end{align*}
The integrands corresponding to  $i=1,2$ converge pointwise to zero as $M\rightarrow\infty$, and are bounded by the functions $\Indi{\Sc_{i}}(x,u_{1},u_{2})G_{1,x}^{(1)}(u_{1},u_{2})$, which, by relations \eqref{eq:Fidentitychaos} and \eqref{ineq:FTheta}, are in turn bounded by 
\begin{align}\label{eq:boundforGtheta}
\Indi{\Sc_{i}}(x,u_{1},u_{2})C\frac{\mu(x,u_{1},u_{2})^2}{(u_{1}u_{2})^{2H}}\Theta_{1}(x,u_{1},u_{2})^{-\frac{d}{2}},
\end{align}
for some constant $C>0$. In addition, by Lemma \ref{lem:inegraltightness}, the function \eqref{eq:boundforGtheta} is integrable for $i=1,2$, and hence, by the dominated convergence theorem, 
\begin{align}\label{eq:unifboundlogR3M}
\limsup_{M\rightarrow\infty}\sup_{\varepsilon\in(0,1/e)}\frac{\varepsilon^{d-2}}{\log(1/\varepsilon)}\Norm{\widetilde{J}_{T,2}^{\varepsilon,M}}_{L^{2}(\Omega)}^2
  &\leq \limsup_{M\rightarrow\infty}\sup_{\varepsilon\in(0,1/e)}\frac{4Tdc_{log}^2}{\log(1/\varepsilon)}\int_{[0,\varepsilon^{-\frac{2}{3}}T]^3}\Indi{\Sc_{3}}(x,u_{1},u_{2})\\
		&\times \prod_{i=1,2}\Indi{(0,2^{-M})\cup(M,\infty)}(u_{i})G_{1,x}^{(1)}(u_{1},u_{2})dx du_{1}du_{2}.\nonumber
\end{align}
On the other hand, by equation \eqref{eq:inequR321} in Lemma \ref{ineq:techR}, 
we deduce that there exists a constant $C>0$, such that for every $(x,u_{1},u_{2})\in\Sc_{3}$, 
\begin{align}\label{eq:G2boundlog}
G_{1,x}^{(1)}(u_{1},u_{2})
  &\leq C(x+u_{1}+u_{2})^{-1}(u_{1}u_{2})^{2}\psi(u_1,u_2). 
\end{align}
Therefore, from \eqref{eq:unifboundlogR3M} we deduce that
\begin{multline*}
\limsup_{M\rightarrow\infty}\sup_{\varepsilon\in(0,1/e)}\frac{\varepsilon^{d-2}}{\log(1/\varepsilon)}\Norm{\widetilde{J}_{T,2}^{\varepsilon,M}}_{L^{2}(\Omega)}^2\\
\begin{aligned}
  &\leq \limsup_{M\rightarrow\infty}\sup_{\varepsilon\in(0,1/e)}\frac{4Cdc_{log}^2T}{\log(1/\varepsilon)}\int_{0}^{\varepsilon^{-\frac{2}{3}}T}\int_{\R_{+}^{2}}\prod_{i=1,2}\Indi{(0,2^{-M})\cup(M,\infty)}(u_{i})\\
	&\times (x+u_{1}+u_{2})^{-1}(u_{1}u_{2})^{2}\psi(u_1,u_2) du_{1}du_{2}dx,
\end{aligned}
\end{multline*}
so that there exists a constant $C>0$ such that
\begin{multline*}
\limsup_{M\rightarrow\infty}\sup_{\varepsilon\in(0,1/e)}\frac{\varepsilon^{d-2}}{\log(1/\varepsilon)}\Norm{\widetilde{J}_{T,2}^{\varepsilon,M}}_{L^{2}(\Omega)}^2\\
\begin{aligned}
  &\leq \limsup_{M\rightarrow\infty}\sup_{\varepsilon\in(0,1/e)}CT\int_{\R_{+}^{2}}\prod_{i=1,2}\Indi{(0,2^{-M})\cup(M,\infty)}(u_{i})\psi(u_1,u_2)\\
	&\times \left(\frac{\log(\varepsilon^{-\frac{2}{3}}T+u_{1}+u_{2})-\log(u_{1}+u_{2})}{\log(1/\varepsilon)}\right)(u_{1}u_{2})^{2}du_{1}du_{2}=0,
\end{aligned}
\end{multline*}
where the last equality easily follows from the dominated convergence theorem. This finishes the proof of \eqref{eq:limunif01}.

To prove \eqref{eq:limunif012} we proceed as follows.  Define the intervals $I_{k}:=(\frac{k-1}{2^{M}},\frac{k}{2^{M}}]$. Then, we can write $R_{T,M}^{\varepsilon}$ and $\widetilde{J}_{T,1}^{M}$, as 
\begin{align*}
R_{T,M}^{\varepsilon}
  &=-\sum_{k=2}^{M2^{M}}c_{log}\varepsilon^{\frac 32 -\frac{d}{2}}
	\int_{0}^{T}\int_{\R_{+}}
	\sum_{j=1}^{d}\Indi{I_{k}}(u)  \frac {u(k)^{\frac 32}}{(1+u(k)^{\frac 32})^{\frac d2+1}}
	 H_{2}\left(\frac{B_{s+\varepsilon^{\frac{2}{3}}u(k)}^{(j)}-B_{s}^{(j)}}{\sqrt{\varepsilon}u(k)^{\frac{3}{4}}}\right)duds,
\end{align*}
and 
\begin{align*}
\widetilde{J}_{T,1}^{\varepsilon,M}
  &=	-\sum_{k=2}^{M2^{M}}c_{log}\varepsilon^{\frac 32-\frac{d}{2}}
	\int_{0}^{T}\int_{\R_{+}}
	\sum_{j=1}^{d}\Indi{I_{k}}(u)    \frac {u(k)^{\frac 32}}{(1+u(k)^{\frac 32})^{\frac d2+1}}H_{2}\left(\frac{B_{s+\varepsilon^{\frac{2}{3}}u}^{(j)}-B_{s}^{(j)}}{\sqrt{\varepsilon}u^{\frac{3}{4}}}\right)duds.
\end{align*}
Notice that by \eqref{eq:Hermiteisometry}, 
\begin{align*}
\E\left[H_{2}\left(\frac{B_{s_{1}+\varepsilon^{\frac{2}{3}}v_{1}}-B_{s_{1}}}{\sqrt{\varepsilon}v_{1}^{\frac{3}{4}}}\right)H_{2}\left(\frac{B_{s_{1}+\varepsilon^{\frac{2}{3}}v_{2}}-B_{s_{2}}}{\sqrt{\varepsilon}v_{2}^{\frac{3}{4}}}\right)\right]
	&=2(v_{1}v_{2})^{-\frac{3}{2}}\mu(\varepsilon^{-\frac{2}{3}}(s_{2}-s_{1}),v_{1},v_{2})^2,
\end{align*}
and hence,
\begin{align*}
\frac{\varepsilon^{d-2}}{\log(1/\varepsilon)}\Norm{\widetilde{J}_{T,1}^{\varepsilon,M}-R_{T,M}^{\varepsilon}}_{L^{2}(\Omega)}^2
  &=\frac{2dc_{log}^2}{\log(1/\varepsilon)}\int_{[0,T]^{2}}\int_{\R_{+}^{2}}\sum_{k_{1},k_{2}=2}^{M2^{M}}\Indi{I_{k_{1}}}(u_{1})\Indi{I_{k_{2}}}(u_{2})\\
	&\times \varepsilon^{-\frac{2}{3}}A_{k_{1},k_{2}}^{M}(\varepsilon^{-\frac{2}{3}}(s_{2}-s_{1}),u_{1},u_{2})ds_{1}ds_{2}du_{1}du_{2}, 
\end{align*}
where the function $A_{k_{1},k_{2}}^{M}(x,u_{1},u_{2})$ is defined by 
\begin{align*}
A_{k_{1},k_{2}}^{M}(x,u_{1},u_{2})
  &:= \big(G_{1,x}^{(1)}(u_{1},u_{2})-G_{1,x}^{(1)}(u(k_{1}),u_{2})\\
	&\ \ -G_{1,x}^{(1)}(u_{1},u(k_{2}))+G_{1,x}^{(1)}(u(k_{1}),u(k_{2}))\big).
\end{align*}
Changing the coordinates $(s_{1},s_{2},u_{1},u_{2})$ by $(s:=s_{1},x:=\varepsilon^{-\frac{2}{3}}(s_{2}-s_{1}),u_{1},u_{1})$ in the case $s_{2}\geq s_{1}$ and by $(s:=s_{2},x:=\varepsilon^{-\frac{2}{3}}(s_{1}-s_{2}),u_{1},u_{1})$ in the case $s_{1}\geq s_{2}$, and integrating the variable $s$, we deduce that there exists a constant $C>0$, such that 
\begin{align}\label{ineq:JRA}
\frac{\varepsilon^{d-2}}{\log(1/\varepsilon)}\Norm{\widetilde{J}_{T,1}^{\varepsilon,M}-R_{T,M}^{\varepsilon}}_{L^{2}(\Omega)}^2
  &\leq \frac{CT}{\log(1/\varepsilon)}\int_{0}^{\varepsilon^{-\frac{2}{3}}T}\int_{\R_{+}^2}\sum_{k_{1},k_{2}=2}^{M2^{M}}\Indi{I_{k_{1}}}(u_{1})\Indi{I_{k_{2}}}(u_{2})\nonumber\\
	&\times \Abs{A_{k_{1},k_{2}}^{M}(x,u_{1},u_{2})}du_{1}du_{2}dx.
\end{align}
In order to bound the term $\Abs{A_{k_{1},k_{2}}^{M}(x,u_{1},u_{2})}$ we proceed as follows. Consider the function 
\begin{align*}
D_{x}^{M}(u_{1},u_{2})
  &:=\psi(u_{1}-2^{-M},u_{2}-2^{-M})\mu(x,u_{1}+2^{-M})^{2}\\
	&-\psi(u_{1}+2^{-M},u_{2}+2^{-M})\mu(x,u_{1}-2^{-M})^{2}, 
\end{align*}
where $\psi(u_1,u_2)$ is defined by \eqref{psi}. By relation \eqref{eq:IpHbig}, we have that 
\begin{align}\label{mudecomplog}
\mu(x,u_{1},u_{2})
  &=\frac{3}{8}\int_{0}^{u_{1}}\int_{x}^{x+u_{2}}\Abs{v_{1}-v_{2}}^{-\frac{1}{2}}dv_{1}dv_{2}\nonumber\\
	&=\frac{3u_{1}u_{2}}{8}\int_{[0,1]^2}\Abs{x+v_{2}u_{2}-v_{1}u_{1}}^{-\frac{1}{2}}dv_{1}dv_{2},
\end{align}
and consequently, $\mu(x,u_{1},u_{2})\leq \mu(x,v_{1},v_{2})$ for every  $u_{1}\leq v_{1}$ and $u_{2}\leq v_{2}$. Using this observation, we can easily show that for every $v_{1}\in[u_{1}-2^{-M},u_{1}+2^{-M}]$ and $v_{2}\in[u_{2}-2^{-M},u_{2}+2^{-M}]$, the following inequality holds
\begin{multline*}
\psi(u_{1}+2^{-M},u_{2}+2^{-M})^{-\frac{d}{2}}\mu(x,u_{1}-2^{-M})^{2}\\
\leq G_{1,x}^{(1)}(v_{1},v_{2})\leq \psi(u_{1}-2^{-M},u_{2}-2^{-M})^{-\frac{d}{2}}\mu(x,u_{1}+2^{-M})^{2}.
\end{multline*}
Hence, for every $u_{1}\in I_{k_{1}}$ and $u_{2}\in I_{k_{2}}$, 
\begin{align}\label{ineq:AD}
\Abs{A_{k_{1},k_{2}}^{M}(u_{1},u_{2})}
  &\leq 2D_{x}^{M}(u_{1},u_{2}).
\end{align}
Using relations \eqref{ineq:JRA} and \eqref{ineq:AD}, as well as the fact that 
$$\sum_{k_{1},k_{2}=2}^{M2^{M}}\Indi{I_{k_{1}}}(u_{1})\Indi{I_{k_{2}}}(u_{2})=\Indi{[2^{-M},M]^2}(u_{1},u_{2}),$$
we obtain 
\begin{align}\label{ineq:JRA21}
\frac{\varepsilon^{d-2}}{\log(1/\varepsilon)}\Norm{\widetilde{J}_{T,1}^{\varepsilon,M}-R_{T,M}^{\varepsilon}}_{L^{2}(\Omega)}^2
  &\leq \frac{CT}{\log(1/\varepsilon)}\int_{0}^{\varepsilon^{-\frac{2}{3}}T}\int_{\R_{+}^2}\Indi{[2^{-M},M]^2}(u_{1},u_{2})D_{x}^{M}(u_{1},u_{2})du_{1}du_{2}dx.
\end{align}
To bound the integral in the right-hand side we proceed as follows. Define $N:=\varepsilon^{-\frac{2}{3}}$, so that $\log(1/\varepsilon)=\frac{3\log N}{2}$. Then, applying L'H\^opital's rule in \eqref{ineq:JRA21}, we deduce that there is a constant $C>0$, such that
\begin{multline}\label{ineq:JRA212}
\limsup_{\varepsilon\rightarrow0}\frac{\varepsilon^{d-2}}{\log(1/\varepsilon)}\Norm{\widetilde{J}_{T,1}^{\varepsilon,M}-R_{T,M}^{\varepsilon}}_{L^{2}(\Omega)}^2\\
\begin{aligned}
  &\leq \limsup_{N\rightarrow\infty}\frac{CT}{\log N}\int_{0}^{NT}\int_{\R_{+}^2}\Indi{[2^{-M},M]^2}(u_{1},u_{2})D_{x}^{M}(u_{1},u_{2})du_{1}du_{2}dx\\
	&= \limsup_{N\rightarrow\infty}CT\int_{\R_{+}^2}\Indi{[2^{-M},M]^2}(u_{1},u_{2})NTD_{NT}^{M}(u_{1},u_{2})du_{1}du_{2}.
\end{aligned}
\end{multline}
On the other hand, using \eqref{mudecomplog} and equation \eqref{eq:inequR321} in Lemma \ref{ineq:techR}, we get that for every $(x,u_{1},u_{2})\in\Sc_{3}$,
\begin{align}\label{eq:limxmu}
\lim_{x\rightarrow\infty}x\mu(x,u_{1},u_{2})^{2}
  &=\frac{3^2u_{1}^2u_{2}^2}{2^{6}},
\end{align}
and 
\begin{align*}
x\mu(x,u_{1},u_{2})^{2}
  &\leq x(x+u_{1}+u_{2})^{-1}(u_{1}u_{2})^2\leq (u_{1}u_{2})^2.
\end{align*}
Hence, by applying the dominated convergence theorem in \eqref{ineq:JRA212}, we deduce that there is a constant $C>0$, such that
\begin{multline}\label{eq:limlogRiemm}
\limsup_{\varepsilon\rightarrow0}\frac{\varepsilon^{d-2}}{\log(1/\varepsilon)}\Norm{\widetilde{J}_{T,1}^{\varepsilon,M}-R_{T,M}^{\varepsilon}}_{L^{2}(\Omega)}^2\\
  \leq CT\int_{\R_{+}^{2}}\Indi{[2^{-M},M]^2}(u_{1},u_{2})\bigg(\psi(u_{1}-2^{-M},u_{2}-2^{-M})((u_{1}+2^{-M})(u_{1}+2^{-M}))^{2}\\
  -\psi(u_{1}+2^{-M},u_{2}+2^{-M})((u_{1}-2^{-M})(u_{1}-2^{-M}))^{2}\bigg)du_{1}du_{2}.
\end{multline}
Let $M_{0}\in\N$ and $\delta>0$ be fixed. Using the fact that integrands in \eqref{eq:limlogRiemm} are decreasing on $M$ and 
$$\sum_{k_{1},k_{2}=2}^{M_{0}2^{M_{0}}}\Indi{I_{k_{1}}}(x_{1})\Indi{I_{k_{2}}}(x_{2})=\Indi{[2^{-M_{0}},M_{0}]}(x_{1})\Indi{[2^{-M_{0}},M_{0}]}(x_{2})\leq 1,$$
we can easily check from the definition of the convergence \eqref{eq:limlogRiemm}, that there exists $\gamma=\gamma(M_{0},\delta)>0$ such that for every $M>M_{0}$, the following inequality holds
\begin{multline}\label{ineq:D0gamma}
\sup_{\varepsilon\in(0,\gamma)}\frac{\varepsilon^{d-2}}{\log(1/\varepsilon)}\Norm{\widetilde{J}_{T,1}^{\varepsilon,M}-R_{T,M}^{\varepsilon}}_{L^{2}(\Omega)}^2\\
  \leq\delta+CT\int_{\R_{+}^{2}}\bigg(\psi(u_{1}-2^{-M_{0}},u_{2}-2^{-M_{0}})((u_{1}+2^{-M_{0}})(u_{1}+2^{-M_{0}}))^{2}\\
  -\psi(u_{1}+2^{-M_{0}},u_{2}+2^{-M_{0}})((u_{1}-2^{-M_{0}})(u_{1}-2^{-M_{0}}))^{2}\bigg).
\end{multline}
To handle the term $\sup_{\varepsilon\in(\gamma,1/e)}\frac{\varepsilon^{d-2}}{\log(1/\varepsilon)}\Norm{\widetilde{J}_{T,1}^{\varepsilon,M}-R_{T,M}^{\varepsilon}}_{L^{2}(\Omega)}^2$, we use \eqref{ineq:JRA21} to get 
\begin{align}\label{ineq:D0gammac2}
\sup_{\varepsilon\in(\gamma,1/e)}\frac{\varepsilon^{d-2}}{\log(1/\varepsilon)}\Norm{\widetilde{J}_{T,1}^{\varepsilon,M}-R_{T,M}^{\varepsilon}}_{L^{2}(\Omega)}^2
  \leq CT\int_{0}^{\gamma^{-\frac{2}{3}}T}\int_{\R_{+}^2}\Indi{[2^{-M},M]^2}(u_{1},u_{2})D_{x}^{M}(u_{1},u_{2})du_{1}du_{2}dx.
\end{align}
From \eqref{ineq:D0gamma} and \eqref{ineq:D0gammac2}, we conclude that there exists a constant $C>0$, only depending on $T$, such that for every $M>M_{0}$,
\begin{multline}\label{ineq:D0gammaf}
\sup_{\varepsilon\in(0,1/e)}\frac{\varepsilon^{d-2}}{\log(1/\varepsilon)}\Norm{\widetilde{J}_{T,1}^{\varepsilon,M}-R_{T,M}^{\varepsilon}}_{L^{2}(\Omega)}^2\\
\begin{aligned}
  &\leq\delta+CT\int_{\R_{+}^{2}}\sum_{k_{1},k_{2}=2}^{M_{0}2^{M_{0}}}\bigg(\psi(u_{1}-2^{-M_{0}},u_{2}-2^{-M_{0}})((u_{1}+2^{-M_{0}})(u_{1}+2^{-M_{0}}))^{2}\\
  &-\psi(u_{1}+2^{-M_{0}},u_{2}+2^{-M_{0}})((u_{1}-2^{-M_{0}})(u_{1}-2^{-M_{0}}))^{2}\bigg)\\
	&+CT\int_{0}^{\gamma^{-\frac{2}{3}}T}\int_{\R_{+}^2}\Indi{[2^{-M},M]^2}(u_{1},u_{2})D_{x}^{M}(u_{1},u_{2})du_{1}du_{2}dx.
\end{aligned}
\end{multline}
Taking first the limit as $M\rightarrow\infty$ and then as $M_{0}\rightarrow\infty$ in \eqref{ineq:D0gammaf}, and applying the dominated convergence theorem, we get
\begin{align*}
\limsup_{M\rightarrow\infty}\sup_{\varepsilon\in(0,1)}\frac{\varepsilon^{d-2}}{\log(1/\varepsilon)}\Norm{\widetilde{J}_{T,1}^{\varepsilon,M}-R_{T,M}^{\varepsilon}}_{L^{2}(\Omega)}^2  &\leq\delta.
\end{align*}
Relation \eqref{eq:limunif012} is then obtained by taking $\delta\rightarrow0$ in the previous inequality.\\

\noindent\textit{Step II}\\
Next we prove that 
\begin{align}\label{eq:Rcovlim}
\lim_{\varepsilon\rightarrow0}\frac{\varepsilon^{d-2}}{\log(1/\varepsilon)}\E\left[(R_{T,M}^{\varepsilon})^2\right]=T\widetilde{\rho}_{M}^2,
\end{align}
where $\widetilde{\rho}_{M}$ is given by 
\begin{align}\label{eq:rhotildeM}
\widetilde{\rho}_{M}
&=\frac{\sqrt{3d}}{2^{\frac{d+5}{2}}\pi^{\frac{d}{2}}2^{M}}\sum_{k=2}^{M2^{M}}(1+u(k)^{\frac{3}{2}})^{-\frac{d}{2}-1}u(k)^{2},
\end{align}
and $u(k)=\frac{k}{2^{M}}$. Notice that in particular, $\widetilde{\rho}_{M}^2$ satisfies 
\begin{align*}
\lim_{M\rightarrow\infty}\widetilde{\rho}_{M}^2
&=\rho^{2},
\end{align*}
where $\rho^{2}$ is defined by \eqref{eq:rhodeflog}. To prove \eqref{eq:rhotildeM} we proceed as follows. Recall that the constant $c_{log}$ is defined by $c_{log}=\frac{(2\pi)^{-\frac{d}{2}}}{2}$. Then, from the definition of $R_{T,M}^{\varepsilon}$ (see equation \eqref{eq:RTMdef}), it easily follows that
\begin{align*}
\frac{\varepsilon^{d-2}}{\log(1/\varepsilon)}\E\left[(R_{T,M}^{\varepsilon})^2\right]
  &=\frac{2dc_{log}^2}{\log(1/\varepsilon)2^{2M}}\int_{[0,T]^2}\sum_{k_{1},k_{2}=2}^{M2^{M}}\varepsilon^{-\frac{2}{3}}G_{1,\varepsilon^{-\frac{2}{3}}(s_{2}-s_{1})}^{(1)}(u(k_{1}),u(k_{2}))ds_{1}ds_{2}.
\end{align*}
Changing the coordinates $(s_{1},s_{2})$ by $(s_{1},x:=s_{2}-s_{1})$, and then integrating the variable $s_{1}$, we get 
\begin{align*}
\frac{\varepsilon^{d-2}}{\log(1/\varepsilon)}\E\left[(R_{T,M}^{\varepsilon})^2\right]
  &=\frac{4dc_{log}^2}{\log(1/\varepsilon)2^{2M}}\int_{0}^{T}T\sum_{k_{1},k_{2}=2}^{M2^{M}}\varepsilon^{-\frac{2}{3}}G_{1,\varepsilon^{-\frac{2}{3}}x}^{(1)}(\varepsilon^{\frac{2}{3}}u(k_{1}),\varepsilon^{\frac{2}{3}}u(k_{2}))dx\\
	&-\frac{4dc_{log}^2}{\log(1/\varepsilon)2^{2M}}\int_{0}^{T}x\sum_{k_{1},k_{2}=2}^{M2^{M}}\varepsilon^{-\frac{2}{3}}G_{1,\varepsilon^{-\frac{2}{3}}x}^{(1)}(\varepsilon^{\frac{2}{3}}u(k_{1}),\varepsilon^{\frac{2}{3}}u(k_{2}))dx.
\end{align*}
Using relation \eqref{eq:Glogdisjoint2} as well as the Cauchy-Schwarz inequality $\mu(x,u_{1},u_{2})\leq (u_{1}u_{2})^{\frac{3}{4}}$, we can easily deduce that there exists a constant $C>0$, depending on $u_{1},\dots, u_{M^{2M}}$, but not on $x$ or $\varepsilon$, such that 
$$G_{1,\varepsilon^{-\frac{2}{3}}x}^{(1)}(u(k_{1}),u(k_{2}))\leq C \varepsilon^{\frac{2}{3}}x^{-1},$$ 
and hence,   
\begin{align*}
\lim_{\varepsilon\rightarrow0}\frac{1}{\log(1/\varepsilon)}\int_{0}^{T}x\sum_{k_{1},k_{2}=2}^{M2^{M}}\varepsilon^{-\frac{2}{3}}G_{1,\varepsilon^{-\frac{2}{3}}x}^{(1)}(u(k_{1}),u(k_{2}))dx=0,
\end{align*}
which implies that 
\begin{align*}
\lim_{\varepsilon\rightarrow0}\frac{\varepsilon^{d-2}}{\log(1/\varepsilon)}\E\left[(R_{T,M}^{\varepsilon})^2\right]
  &=\lim_{\varepsilon\rightarrow0}\frac{4dc_{log}^2T}{\log(1/\varepsilon)2^{2M}}\int_{0}^{T}\sum_{k_{1},k_{2}=2}^{M2^{M}}\varepsilon^{-\frac{2}{3}}G_{1,\varepsilon^{-\frac{2}{3}}x}(u(k_{1}),u(k_{2}))dx\\
	&=\lim_{\varepsilon\rightarrow0}\frac{4dc_{log}^2T}{\log(1/\varepsilon)2^{2M}}\int_{0}^{\varepsilon^{-\frac{2}{3}}T}\sum_{k_{1},k_{2}=2}^{M2^{M}}G_{1,x}(u(k_{1}),u(k_{2}))dx,
\end{align*}
where the last equality follows by making the change of variables $\widetilde{x}:=\varepsilon^{-\frac{2}{3}}x$. Hence, writing $N:=\varepsilon^{-\frac{2}{3}}$, so that $\log(1/\varepsilon)=\frac{2 \log N}{3}$, and using L'H\^opital's rule, we get 
\begin{align}\label{eq:varlimRTM}
\lim_{\varepsilon\rightarrow0}\frac{\varepsilon^{d-2}}{\log(1/\varepsilon)}\E\left[(R_{T,M}^{\varepsilon})^2\right]
  &=\lim_{N\rightarrow\infty}\frac{8dc_{log}^2T}{3\log N2^{2M}}\int_{0}^{NT}\sum_{k_{1},k_{2}=2}^{M2^{M}}G_{1,x}^{(1)}(u(k_{1}),u(k_{2}))dx\nonumber\\
	&=\lim_{N\rightarrow\infty}\frac{8dc_{log}^2T}{3\cdot2^{2M}}\sum_{k_{1},k_{2}=2}^{M2^{M}}NTG_{1,NT}^{(1)}(u(k_{1}),u(k_{2}))dx=\widetilde{\rho}_{M}^2,
\end{align}
where the last identity follows from \eqref{eq:Gdef} and \eqref{eq:limxmu}. This finishes the proof of \eqref{eq:Rcovlim}.

\noindent\textit{Step III}\\
Next we prove the convergence in law of $\frac{\varepsilon^{\frac{d}{2}-1}}{\sqrt{\log(1/\varepsilon)}}\widetilde{J}_{T}^{\varepsilon}$ to a Gaussian random variable with variance $\rho^{2}$. From Steps I and II, it suffices to show that 
\begin{align}\label{eq:RTMtoGaussian}
R_{T,M}^{\varepsilon}
  &\stackrel{Law}{\rightarrow}\mathcal{N}(0,\widetilde{\rho}_{M}^{2}),\ \ \ \ \ \text{ as }\varepsilon\rightarrow0,
\end{align}

In order to prove \eqref{eq:RTMtoGaussian} we proceed as follows. Define the random vector 
\begin{align*}
D^{\varepsilon}
  &=\left(D_k^{\varepsilon}\right)_{k=2}^{M2^{M}},
\end{align*}
where
\begin{align*}
D_k^{\varepsilon}
  &:=-\frac{c_{log}u(k)^{\frac{3}{2}}}{2^{M}(1+u(k)^{\frac{3}{2}})^{\frac{d}{2}+1}}
 \sum_{j=1}^{d}\frac{1}{\varepsilon^{\frac{1}{3}}\sqrt{\log(1/\varepsilon)}}\int_{0}^{T}H_{2}\left(\frac{B_{s+\varepsilon^{\frac{2}{3}}u(k)}^{(j)}-B_{s}^{(j)}}{\sqrt{\varepsilon}u(k)^{\frac{3}{4}}}\right)ds,
\end{align*}
and $c_{log}=\frac{(2\pi)^{-\frac{d}{2}}}{2}$. Notice that 
$$\frac{\varepsilon^{\frac{d}{2}-1}}{\sqrt{\log(\varepsilon)}}R_{T,M}^{\varepsilon}
  =\sum_{k=2}^{M2^{M}}D_{k}^{\varepsilon}.$$
We will prove that $D^{\varepsilon}$ converges to a centered Gaussian vector. By the Peccati-Tudor criterion (see \cite{PeTu}), it suffices to prove that the components of the vector $D^{\varepsilon}$ converge to a Gaussian distribution, and the covariance matrix of $D^{\varepsilon}$ is convergent. To prove the former statement,  define 
\begin{align*}
\Psi_{k_{1},k_{2}}^{j}(\varepsilon):=\E\left[\int_{0}^{T}H_{2}\left(\frac{B_{s_{1}+\varepsilon^{\frac{2}{3}}u(k_{1})}^{(j)}-B_{s_{1}}^{(j)}}{\sqrt{\varepsilon}u(k_{1})^{\frac{3}{4}}}\right)ds_{1}\int_{0}^{T}H_{2}\left(\frac{B_{s_{2}+\varepsilon^{\frac{2}{3}}u(k_{2})}^{(j)}-B_{s_{2}}^{(j)}}{\sqrt{\varepsilon}u(k_{2})^{\frac{3}{4}}}\right)ds_{2}\right].
\end{align*}
Proceeding as in the proof of \eqref{eq:varlimRTM}, we can show that for $2\leq k_{1},k_{2}\leq M 2^{M}$,  
\begin{align*}
\Psi_{k_{1},k_{2}}^{j}(\varepsilon)
  &=\frac{2(u(k_{1})u(k_{2}))^{-\frac{3}{2}}}{\varepsilon^{\frac{8}{3}}\log(1/\varepsilon)}\int_{[0,T]^2}\mu(s_{2}-s_{1},\varepsilon^{\frac{2}{3}}u(k_{1}),\varepsilon^{\frac{2}{3}}u(k_{2}))^2ds_{1}ds_{2}\\
	&=\frac{8(u(k_{1})u(k_{2}))^{-\frac{3}{2}}}{3\log(\varepsilon^{-\frac{2}{3}})}\int_{0}^{\varepsilon^{-\frac{2}{3}}T}\int_{0}^{T-\varepsilon^{\frac{2}{3}}x}\mu(x,u(k_{1}),u(k_{2}))^2dsdx.
\end{align*}
As in the proof of \eqref{eq:varlimRTM}, we can use L'H\^opital's rule, \eqref{eq:limxmu} and the previous identity, to get
\begin{align*}
\lim_{\varepsilon\rightarrow0}\Psi_{n}^{i,j}
  &=\lim_{\varepsilon\rightarrow0}\frac{8(u(k_{1})u(k_{2}))^{-\frac{3}{2}}T}{3\log(\varepsilon^{-\frac{2}{3}})}\int_{0}^{\varepsilon^{-\frac{2}{3}}T}\mu(x,u(k_{1}),u(k_{2}))^2dsdx
	=\frac{3}{2^3}T\sqrt{u(k)u(j)}.
\end{align*}
From here, it follows that
\begin{align*}
\lim_{\varepsilon\rightarrow0}\E\left[D_{k_{1}}^{\varepsilon}D_{k_{2}}^{\varepsilon}\right]
  &=\Sigma_{i,j}:=\frac{3dT}{2^{d+5}\pi^{d}2^{2M}}\psi(u(k_{1}),u(k_{2}))(u(k_{1})u(k_{2}))^{2},
\end{align*} 
namely, the covariance matrix of $D^{\varepsilon}$ converges to the matrix $\Sigma=(\Sigma_{k,j})_{2\leq k,j\leq M2^{M}}$.
In addition, by \cite[Equation(1.4)]{DaNoNu}~, for $2\leq k\leq M2^{M}$ fixed, the sequence of random variables $D_{k}^{\varepsilon}$ converges to a Gaussian random variable as $\varepsilon\rightarrow0$. Therefore, by the Peccati-tudor criterion, the random vector $D$ converges to a jointly Gaussian vector $Z=(Z_{k})_{k=2}^{M2^{M}}$, with mean zero and covariance $\Sigma$. In particular, we have 
\begin{align*}
\frac{\varepsilon^{\frac{d}{2}-1}}{\sqrt{\log(\varepsilon)}}R_{T,M}^{\varepsilon}
  &=\sum_{k=2}^{M2^{M}}D_{k}^{\varepsilon}
  \stackrel{Law}{\rightarrow}\mathcal{N}\left(0,\sum_{j,k=2}^{M2^{M}}\Sigma_{k,j}\right)\ \ \ \ \ \text{ as } \ \ \varepsilon\rightarrow0.
\end{align*}
Relation \eqref{eq:RTMtoGaussian} easily follows from the previous identity.

Since \eqref{eq:fddlog} holds, in order to finish the proof of Theorem \ref{Teo:convergencelog} it suffices to prove tightness. As before, we define, for $T_{1}\leq T_{2}$ belonging to a compact interval $[0,K]$ the random variable $Z_{\varepsilon}$ by the formula \eqref{eq:Zdef}. Then, by the Billingsley criterion,  it suffices to prove that there exist constants $C>0$ and $p>2$, only depending on $K$, such that 
\begin{align}\label{ineq:Zplogfinal2}
\E\left[\Abs{\frac{\varepsilon^{\frac{d}{2}-1}}{\sqrt{\log(1/\varepsilon)}}Z_{\varepsilon}}^{p}\right]
  &\leq C(T_{2}-T_{1})^{\frac{p}{2}}.
\end{align} 
Using relation \eqref{ineq:Ze2} with $H=\frac{3}{4}$, we can easily check that 
\begin{align}\label{ineq:Ze2log}
\frac{\varepsilon^{d-2}}{\log(1/\varepsilon)}\Norm{Z_{\varepsilon}}_{L^{p}(\Omega)}^{2}
  &\leq \frac{C(T_{2}-T_{1})}{\log(1/\varepsilon)}\int_{\R_{+}^2}\int_{0}^{T_{2}}\varepsilon^{-\frac{8}{3}}\mu(x,\varepsilon^{\frac{2}{3}}u_{1},\varepsilon^{\frac{2}{3}}u_{2})^2(u_{1}u_{2})^{-2H}\nonumber\\
	&\times\Theta_{1}(\varepsilon^{-\frac{2}{3}}x,u_{1},u_{2})^{-\frac{d}{p}}dx du_{1}du_{2}\\
	&\leq \sup_{\varepsilon\in(0,1/e)}\frac{C(T_{2}-T_{1})}{\log(1/\varepsilon)}\int_{\R_{+}^2}\int_{0}^{T_{2}}\varepsilon^{-\frac{8}{3}}\mu(x,\varepsilon^{\frac{2}{3}}u_{1},\varepsilon^{\frac{2}{3}}u_{2})^2(u_{1}u_{2})^{-2H}\nonumber\\
	&\times\Theta_{1}(\varepsilon^{-\frac{2}{3}}x,u_{1},u_{2})^{-\frac{d}{p}}dx du_{1}du_{2}.\nonumber
\end{align}
The right-hand side in the previous identity is finite for $p>2$ sufficiently small by Lemma \ref{lem:finiteFGintegrallog}, and hence, there exists a constant $p>2$ such that
\begin{align*}
\frac{\varepsilon^{d-2}}{\log(1/\varepsilon)}\E\left[\Abs{Z_{\varepsilon}}^{p}\right]
  &\leq C(T_{2}-T_{1})^{\frac{p}{2}}.
\end{align*}
This finishes the proof of the tightness property for $\frac{\varepsilon^{\frac{d}{2}-1}}{\sqrt{\log(1/\varepsilon)}}(I_{T}^{\varepsilon}-\E\left[I_{T_{1}}^{\varepsilon}\right])$. The proof of Theorem \ref{Teo:convergencelog} is now complete.

\section{Technical lemmas}\label{sec:tech}
In this section we prove some technical lemmas, which where used in the proof of Theorems \ref{Teo:convergence}, \ref{Teo:convergenceHermite} and \ref{Teo:convergencelog}.
\begin{Lema}\label{lem:local_non_determinism}
Let $s_{1},s_{2},t_{1},t_{2}\in\R_{+}$ be such that $s_{1}\leq s_{2}$, and $s_{i}\leq t_{i}$  for $i=1,2$. Denote by $\Sigma$ the covariance matrix of $(B_{t_{1}}-B_{s_{1}},B_{t_{2}}-B_{s_{2}})$. Then, there exists a constants $0<\delta<1$ and $k>0$, such that the following inequalities hold 
\begin{enumerate}
\item If $s_{1}< s_{2}< t_{1}< t_{2}$, 
\begin{align}\label{ineq:local_non_determinism_R1}
\left|\Sigma\right|
  \geq \delta ((a+b)^{2H}c^{2H}+(b+c)^{2H}a^{2H}),
\end{align}
where $a:=s_{2}-s_{1}$, $b:=t_{1}-s_{2}$ and $c:=t_{2}-t_{1}$.
\item If $s_{1}< s_{2}< t_{2}< t_{1}$,
\begin{align}\label{ineq:local_non_determinism_R2}
\left|\Sigma\right|
  \geq \delta b^{2H}(a^{2H}+c^{2H}),
\end{align}
where $a:=s_{2}-s_{1}$, $b:=t_{2}-s_{2}$ and $c:=t_{1}-t_{2}$.
\item If $s_{1}<t_{1}<s_{2}<t_{2}$,
\begin{align}\label{ineq:local_non_determinism_R3}
\left|\Sigma\right|
  \geq \delta a^{2H}c^{2H},
\end{align}
where $a:=t_{1}-s_{1}$ and $c:=t_{2}-s_{2}$.
\end{enumerate}
\end{Lema}
\begin{proof}
Relations \eqref{ineq:local_non_determinism_R1}-\eqref{ineq:local_non_determinism_R3} follow from Lemma B.1. in \cite{JMTANAKA}. The inequalities \eqref{ineq:local_non_determinism_R1} and \eqref{ineq:local_non_determinism_R3} where also proved in \cite[Lemma~9]{HuNu}, but the lower bound given in this lemma for the case $s_1<s_2<t_2<t_1$ is not correct. 
\end{proof}
\begin{Lema}\label{ineq:techR}
There exists a constant $k>0$, such that for every $s_{1}<t_{1}<s_{2}<t_{2}$,
\begin{align}\label{eq:inequR32}
\mu(a+b,a,c)
  &\leq kb^{2H-2}ac,
\end{align}
where $a:=t_{1}-s_{1}$, $b:=s_{2}-t_{1}$ and $c:=t_{2}-s_{2}$. In addition, if $H>\frac{1}{2}$,
\begin{align}\label{eq:inequR321}
\mu(x,u_{1},u_{2})
  &\leq k(x+u_{1}+u_{2})^{2H-2}u_{1}u_{2},
\end{align}
where $x:=s_{2}-s_{1}$, $u_{1}:=t_{1}-s_{1}$ and $u_{2}:=t_{2}-s_{2}$.
\end{Lema}
\begin{proof}
We can easily check that 
\begin{align*}
\mu(a+b,a,c)
  &=\frac{1}{2}((a+b+c)^{2H}+b^{2H}-(b+c)^{2H}-(a+b)^{2H}),
\end{align*}
and hence, 
\begin{align*}
\mu(a+b,a,c)
  &=H(2H-1)ac\int_{[0,1]^2}\Abs{b+av_{1}+cv_{2}}^{2H-2}dv_{1}dv_{2},
\end{align*}
Relation \eqref{eq:inequR32} follows by dropping the term $av_{1}+cv_{2}$ in the previous integral, while  \eqref{eq:inequR321} follows from the following computation, which is valid for every $H>\frac{1}{2}$, 
\begin{align*}
\mu(a+b,a,c)
  &=H(2H-1)ac\int_{[0,1]^2}\Abs{b+av_{1}+cv_{2}}^{2H-2}dv_{1}dv_{2}\\
	&\leq H(2H-1)ac\int_{0}^{1}\Abs{(a\vee b\vee c)v}^{2H-2}dv\\
	&=Hac\Abs{a\vee b\vee c}^{2H-2}\leq H4^{2H-2}ac\Abs{2a+b+c}^{2H-2}\\
	&=4^{2H-2}H(x+u_{1}+u_{2})^{2H-2}u_{1}u_{2}.
\end{align*}
\end{proof}

\begin{Lema}\label{lem:inegraltightness}
Define the functions $\mu$ and $\Theta_{1}$ by \eqref{eq:mudef} and \eqref{eq:Ldef} respectively. Let $\frac{3}{2d}<H<1$, and $0<p<\frac{4Hd}{3}$ be fixed. {\color{red}Then, the following integral is convergent 
\begin{align}\label{eq:integralexpressiontight2p}
\int_{\Sc_{i}}\frac{\mu(x,u_{1},u_{2})^2}{u_{1}^{2H}u_{2}^{2H}}\Theta_{1}(x,u_{1},u_{2})^{-\frac{d}{p}}\text{d}x\text{d}u_{1}\text{d}u_{2}<\infty,
\end{align}
for $i=1,2$, where the sets $\Sc_{i}$ are defined by \eqref{eq:Scdef}.} Moreover, if $H<\frac{3}{4}$, then 
\begin{align}\label{eq:integralexpressiontight2}
\int_{\R_{+}^{3}}\frac{\mu(x,u_{1},u_{2})^2}{u_{1}^{2H}u_{2}^{2H}}\Theta_{1}(x,u_{1},u_{2})^{-\frac{d}{p}}\text{d}x\text{d}u_{1}\text{d}u_{2}<\infty.
\end{align}
\end{Lema}
\begin{proof}
Denote the integrand in \eqref{eq:integralexpressiontight2} and \eqref{eq:integralexpressiontight2p} by $\Psi(x,u_{1},u_{2})$, namely, 
\begin{align}\label{eq:Psidef}
\Psi(x,u_{1},u_{2})
  &=\mu(x,u_{1},u_{2})^2(u_{1}u_{2})^{-2H}\Theta_{1}(x,u_{1},u_{2})^{-\frac{d}{p}}.
\end{align}
We can decompose the domain of integration of \eqref{eq:integralexpressiontight2}, as $\R_{+}^{3}=\Sc_{1}\cup\Sc_{2}\cup\Sc_{3}$, where $\Sc_{1},\Sc_{2},\Sc_{3}$ are defined by \eqref{eq:Scdef}. Then, it suffices to show that 
\begin{align}\label{eq:Psiplain}
\int_{\Sc_{i}}\Psi(x,u_{1},u_{2})dxdu_{1}du_{2}<\infty,
\end{align} 
for $i=1,2$ provided that $0<p<\frac{4Hd}{3}$, and for $i=3$, provided that $0<p<\frac{4Hd}{3}$ and $H<\frac{3}{4}$. First consider the case $i=1$. Changing the coordinates $(x,u_{1},u_{2})$  by $(a:=x,b:=u_{1}-x,c:=x+u_{2}-u_{1} )$ in \eqref{eq:Psiplain} for $i=1$, we get 
\begin{align*}
\int_{\Sc_{1}}\Psi(x,u_{1},u_{2})dxdu_{1}du_{2}
  &=\int_{\R_{+}^{3}}\Psi(a,a+b,b+c)dadbdc.
\end{align*}
To bound the integral in the right-hand side we proceed as follows. First we notice that the term $\mu(a,a+b,b+c)$ is given by 
\begin{align*}
\mu(a,a+b,b+c)
  &=\frac{1}{2}((a+b+c)^{2H}+b^{2H}-c^{2H}-a^{2H}).
\end{align*}
By the Cauchy-Schwarz inequality, $\Abs{\mu(a,a+b,b+c)}\leq(a+b)^{H}(b+c)^{H}$. In addition, by \eqref{ineq:local_non_determinism_R1} there exists a constant $\delta>0$ such that  
\begin{align}\label{ineq:lndr1}
(a+b)^{2H}(b+c)^{2H}-\mu(a,a+b,b+c)^2
  &\geq\delta((a+b)^{2H}c^{2H}+(b+c)^{2H}a^{2H}).
\end{align}
As a consequence, 
\begin{align*}
\Psi(a,a+b,b+c)
  &\leq\left(1+(a+b)^{2H}+(b+c)^{2H}+\delta((a+b)^{2H}c^{2H}+(b+c)^{2H}a^{2H})\right)^{-\frac{d}{p}}\nonumber.
\end{align*}
Hence, we deduce that there exists a constant $K>0$ such that the following inequalities hold
\begin{align*}
\Psi(a,a+b,b+c)
&\leq K\left(1+c^{2H}+c^{2H}b^{2H}\right)^{-\frac{d}{p}}~~~~~~~~~~~~~\text{if }~~a\leq b\leq c,\\
\Psi(a,a+b,b+c)
&\leq K\left(1+c^{2H}+c^{2H}a^{2H}\right)^{-\frac{d}{p}}~~~~~~~~~~~~~\text{if }~~b\leq a\leq c,\\
\Psi(a,a+b,b+c)
&\leq K\left(1+b^{2H}+c^{2H}b^{2H}\right)^{-\frac{d}{p}}~~~~~~~~~~~~~~~~\text{if }~~a\leq c\leq b,\\
\Psi(a,a+b,b+c)
&\leq K\left(1+b^{2H}+b^{2H}a^{2H}\right)^{-\frac{d}{p}}~~~~~~~~~~~~~~~~\text{if }~~c\leq a\leq b,\\
\Psi(a,a+b,b+c)
&\leq K\left(1+a^{2H}+c^{2H}a^{2H}\right)^{-\frac{d}{p}}~~~~~~~~~~~~~\text{if }~~b\leq c\leq a,\\
\Psi(a,a+b,b+c)
&\leq K\left(1+a^{2H}+b^{2H}a^{2H}\right)^{-\frac{d}{p}}~~~~~~~~~~~~~\text{if }~~c\leq b\leq a.
\end{align*}
Using the condition $p<\frac{4Hd}{3}$, as well as the previous inequalities, we can easily check that $\Psi(a,a+b,b+c)$ is integrable
in $\R_{+}^{3}$, which in turn implies that $\Psi(x,u_{1},u_{2})$ is integrable in $\Sc_{1}$, as required.\\~\\
\noindent Next we consider the case $i=2$.  Changing the coordinates $(x,u_{1},u_{2})$  by $(a:=x,b:=u_{2},c:=u_{1}-x-u_{2})$ in \eqref{eq:Psiplain} for $i=2$, we get 
\begin{align*}
\int_{\Sc_{2}}\Psi(x,u_{1},u_{2})dxdu_{1}du_{2}
  &=\int_{\R_{+}^{3}}\Psi(a,a+b+c,b)dadbdc.
\end{align*}
To bound the integral in the right-hand side we proceed as follows. First notice that the term $\mu(a,a+b+c,b)$ is given by
\begin{align}\label{eq:mudefr2}
\mu(a,a+b+c,b)
  &=\frac{1}{2}((b+c)^{2H}+(a+b)^{2H}-c^{2H}-a^{2H}).
\end{align}
By the Cauchy-Schwarz inequality, $\Abs{\mu(a,a+b+c,b)}\leq b^{H}(a+b+c)^{H}$. In addition, by \eqref{ineq:local_non_determinism_R2}, there exists a constant $\delta>0$ such that 
\begin{align*}
b^{2H}(a+b+c)^{2H}-\mu(a,a+b+c,b)^{2}
  &\geq\delta b^{2H}(a^{2H}+c^{2H}).
\end{align*}
As a consequence, 
\begin{align*}
\Psi(a,a+b+c,b)
  &\leq \left(1+b^{2H}+(a+b+c)^{2H}+\delta b^{2H}(a^{2H}+c^{2H})\right)^{-\frac{d}{p}}.
\end{align*}
From here it follows that there exists a constant $K>0$ such that the following inequalities hold
\begin{align}\label{eq:R2sub1}
\Psi(a,a+b+c,b)
&\leq K\left(1+c^{2H}+b^{2H}c^{2H}\right)^{-\frac{d}{p}}~~~~~~~~~~~~~\text{if }~~a\leq b\leq c,\nonumber\\
\Psi(a,a+b+c,b)
&\leq K\left(1+b^{2H}+b^{2H}c^{2H}\right)^{-\frac{d}{p}}~~~~~~~~~~~~~~~\text{if }~~a\leq c\leq b,\nonumber\\
\Psi(a,a+b+c,b)
&\leq K\left(1+b^{2H}+b^{2H}a^{2H}\right)^{-\frac{d}{p}}~~~~~~~~~~~~~~~\text{if }~~c\leq a\leq b,\nonumber\\
\Psi(a,a+b+c,b)
&\leq K\left(1+a^{2H}+b^{2H}a^{2H}\right)^{-\frac{d}{p}}~~~~~~~~~~~~\text{if }~~c\leq b\leq a.
\end{align}
Using the condition $p<\frac{4Hd}{3}$, as well as the previous inequalities, we can easily check that $\Psi(a,a+b+c,b)$ is integrable in the region $\{(a,b,c)\in\R_{+}^{3}\ |\ b\geq a\wedge c\}$.\\

Next we check the integrability of $\Psi(a,a+b+c,b)$ in $\{(a,b,c)\in\R_{+}^{3}\ |\ b\leq a\wedge c\}$. Applying the mean value theorem in  \eqref{eq:mudefr2}, we can easily check that 
\begin{align}\label{eq:mumv}
\mu(a,a+b+c,b)
  &=\frac{1}{2}(2H(a+\xi_{1})^{2H-1}b+2H(c+\xi_{2})^{2H-1}b),
\end{align}
for some $\xi_{1},\xi_{2}$ between $0$ and $b$. Therefore, if $H<\frac{1}{2}$, we obtain 
\begin{align}\label{eq:lndr2}
\mu(a,a+b+c,b)
  &\leq H(a^{2H-1}+c^{2H-1})b,
\end{align}
which in turn implies that  
\begin{align}\label{eq:ineq2reg2}
\Psi(a,a+b+c,b)
  &\leq H^2(a^{2H-1}+c^{2H-1})^2b^{2-2H}(a+b+c)^{-2H}\nonumber\\
	&\left(1+b^{2H}+(a+b+c)^{2H}+\delta b^{2H}(a^{2H}+c^{2H})\right)^{-\frac{d}{p}}.
\end{align}
For the case $H\geq\frac{1}{2}$, we use \eqref{eq:mumv}, in order to obtain
\begin{align*}
\mu(a,a+b+c,b)
  &\leq H((a+b)^{2H-1}+(c+b)^{2H-1})b,
\end{align*}
which in turn implies that  
\begin{align}\label{eq:ineq2reg2p}
\Psi(a,a+b+c,b)
  &\leq H^2((a+b)^{2H-1}+(c+b)^{2H-1})^2b^{2-2H}(a+b+c)^{-2H}\nonumber\\
	&\left(1+b^{2H}+(a+b+c)^{2H}+\delta b^{2H}(a^{2H}+c^{2H})\right)^{-\frac{d}{p}}.
\end{align}
From \eqref{eq:ineq2reg2}, we deduce that, if $H<\frac{1}{2}$, there exists a constant $K>0$ such that
\begin{align}\label{eq:ineqlong2p}
\Psi(a,a+b+c,b)
&\leq Ka^{4H-2}b^{2-2H}c^{-2H}\left(1+c^{2H}+b^{2H}c^{2H}\right)^{-\frac{d}{p}}~~~~~~~\text{if }~~b\leq a\leq c,\nonumber\\
\Psi(a,a+b+c,b)
&\leq Kc^{4H-2}b^{2-2H}a^{-2H}\left(1+a^{2H}+b^{2H}a^{2H}\right)^{-\frac{d}{p}}~~~~~~~\text{if }~~b\leq c\leq a.
\end{align}
In turn, from \eqref{eq:ineq2reg2p}, it follows that if $H\geq\frac{1}{2}$, there exists a constant $K>0$, such that 
\begin{align}\label{eq:ineqlong2pp}
\Psi(a,a+b+c,b)
&\leq Kc^{4H-2}b^{2-2H}\left(1+c^{2H}+b^{2H}c^{2H}\right)^{-\frac{d}{p}}~~~~~~~\text{if }~~b\leq a\leq c,\nonumber\\
\Psi(a,a+b+c,b)
&\leq Ka^{4H-2}b^{2-2H}\left(1+a^{2H}+b^{2H}a^{2H}\right)^{-\frac{d}{p}}~~~~~~~\text{if }~~b\leq c\leq a.
\end{align}
Using the conditions $H<\frac{3}{4}$ and $p<\frac{4Hd}{3}$,  we can easily check that $2H<\frac{Hd}{2p}$, which, by \eqref{eq:ineqlong2p} and \eqref{eq:ineqlong2pp}, implies that $\Psi(a,a+b+c,b)$ is integrable in $\{(a,b,c)\in\R_{+}^{3}\ |\ b\leq a\wedge c\}$. From here it follows that $\Psi(a,a+b+c,b)$ is integrable in $\R_{+}^{3}$, and hence $\Psi(x,u_{1},u_{2})$ is integrable in $\Sc_{2}$, as required.\\~\\
\noindent Finally we consider the case $i=3$ for $H<\frac{3}{4}$.  Changing the coordinates $(x,u_{1},u_{2})$  by $(a:=u_{1},b:=x-u_{1},c:=u_{2})$ in \eqref{eq:Psiplain} for $i=3$, we get 
\begin{align*}
\int_{\Sc_{3}}\Psi(x,u_{1},u_{2})dxdu_{1}du_{2}
  &=\int_{\R^{3}}\Psi(a+b,a,c)dadbdc.
\end{align*}
To bound the integral in the right-hand side we proceed as follows. First we notice that the term $\mu(a+b,a,c)$ is given by 
\begin{align}\label{eq:mudefR3}
\mu(a+b,a,c)
  &=\frac{1}{2}((a+b+c)^{2H}+b^{2H}-(b+c)^{2H}-(a+b)^{2H}).
\end{align}
 By the Cauchy-Schwarz inequality, $\mu(a+b,a,c)\leq a^{H}c^{H}$. In addition, by \eqref{ineq:local_non_determinism_R3}, there exist constants $k,\delta>0$ such that  
\begin{align}\label{eq:inequR31}
a^{2H}c^{2H}-\mu(a+b,a,c)^{2}
  &\geq\delta a^{2H}c^{2H},
\end{align}
and 
\begin{align}\label{eq:inequR32p}
\mu(a+b,a,c)
  &\leq kb^{2H-2}ac.
\end{align}
From \eqref{eq:inequR31}-\eqref{eq:inequR32p}, we deduce the following bounds for $\Psi$
\begin{align}
\Psi(a+b,a,c)
  &\leq \left(1+a^{2H}+c^{2H}+\delta a^{2H}c^{2H}\right)^{-\frac{d}{p}},\label{eq:ineq1reg3}\\
\Psi(a+b,a,c)
  &\leq 2Hb^{4H-4}(ac)^{-2H+2}\left(1+a^{2H}+c^{2H}+\delta a^{2H}c^{2H}\right)^{-\frac{d}{p}}.\label{eq:ineq3reg3}	
\end{align}
Using \eqref{eq:ineq1reg3}, as well as the condition $p<\frac{4Hd}{3}$, we can easily check that $\Psi(a+b,a,c)$ is integrable in the region $\{(a,b,c)\in\R_{+}^{3}\ |\ b\leq a\wedge c\}$.\\~\\
Next we check the integrability of $\Psi(a+b,a,c)$ in the region $\{(a,b,c)\in\R_{+}^{3}\ |\ b\geq a\vee c\}$. Since $H<\frac{3}{4}$, from \eqref{eq:ineq3reg3} it follows that there exists a constant $C>0$ such that 
\begin{align*}
\int_{(a\vee c)}^{\infty}\Psi(a+b,a,c)db
  &\leq C(ac)^{-2H+2}(a\vee c)^{4H-3}\left(1+a^{2H}+c^{2H}+a^{2H}c^{2H}\right)^{-\frac{d}{p}}\nonumber\\
	&\leq C(ac)^{\frac{1}{2}}\left(1+a^{2H}+c^{2H}+a^{2H}c^{2H}\right)^{-\frac{d}{p}}.
\end{align*}
The integrability of $\Psi(a+b,a,c)$ in the region $\{(a,b,c)\in\R_{+}^{3}\ |\ b\geq a\vee c\}$ then follows from condition the $p<\frac{4Hd}{3}$.

Finally, we prove the integrability of $\Psi(a+b,a,c)$ in the regions $\{(a,b,c)\in\R_{+}^{3}\ |\ a\leq b\leq c\}$ and $\{(a,b,c)\in\R_{+}^{3}\ |\ c\leq b\leq a\}$. Let $a,b,c\geq0$ be such that $a\leq b\leq c$. Applying the mean value theorem to \eqref{eq:mudefR3}, we can easily show that
\begin{align*}
\mu(a+b,a,c)
  &=\frac{1}{2}(\xi_{1}^{2H-1}a-\xi_{2}^{2H-1}a),
\end{align*}
for some $\xi_{1}$ between $c+b$ and $a+b+c$, and $\xi_{2}$ between $b$ and $a+b$. Hence, if $H\leq \frac{1}{2}$, it follows that
\begin{align*}
\Abs{\mu(a+b,a,c)}
  &\leq\frac{1}{2}(\Abs{\xi_{1}}^{2H-1}a+\Abs{\xi_{2}}^{2H-1}a)\\
	&\leq\frac{1}{2}((c+b)^{2H-1}a+b^{2H-1}a).
\end{align*} 
From here it follows that there exists a constant $C>0$, only depending on $H$ such that 
\begin{align}\label{eq:muboundr3}
\Abs{\mu(a+b,a,c)}
  &\leq Cb^{2H-1}a.
\end{align}
Using inequalities \eqref{eq:inequR31} and \eqref{eq:muboundr3}, we deduce that there exists a constant $K>0$ such that 
\begin{align*}
\Psi(a+b,a,c)
  &\leq Kb^{4H-2}a^{2-2H}c^{-2H}(1+a^{2H}+c^{2H}+a^{2H}c^{2H})^{-\frac{d}{p}}.
\end{align*}
From here, it follows that 
\begin{align}\label{ineq:psir31}
\Psi(a+b,a,c)
  &\leq Kb^{4H-2}a^{2-2H}c^{-2H}(1+a^{2H}+c^{2H}+a^{2H}c^{2H})^{-\frac{d}{p}}.
\end{align}
Using the condition $H\leq \frac{3}{4}$, we can easily show that $2H-\frac{2Hd}{p}\leq \frac{3}{2}-\frac{2Hd}{p}<0$. Hence, from \eqref{ineq:psir31}, we deduce that $\Psi(a+b,a,c)$ is integrable in $\{(a,b,c)\in\R_{+}^{3}\ |\ a\leq b\leq c\}$. The integrability of $\Psi(a+b,a,c)$ over the region $\{(a,b,c)\in\R_{+}^{3}\ |\ c\leq b\leq a\}$ in the case $H\leq \frac{1}{2}$, follows from a similar argument. To handle the case $H>\frac{1}{2}$, we proceed as follows. From \eqref{eq:mudefR3}, we can easily show that for every $a,b,c\geq0$ such that $a\leq b\leq c$,
\begin{align*}
\mu(a+b,a,c)
  &=H(2H-1)ac\int_{[0,1]^2}(b+a\xi+c\eta)^{2H-2}d\xi d\eta\\
	&\leq H(2H-1)ac\int_{0}^{1}(c\eta)^{2H-2}d\eta,
\end{align*}
and hence 
\begin{align*}
\mu(a+b,a,c)
  &\leq Hac^{2H-1}.
\end{align*}
From here it follows that 
\begin{align*}
\Psi(a+b,a,c)
  &\leq a^{2-2H}c^{2H-2}(1+a^{2H}+c^{2H}+a^{2H}c^{2H})^{-\frac{d}{p}}.
\end{align*}
Using the condition $p<\frac{4Hd}{3}$, we deduce that $\Psi(a+b,a,c)$ is integrable in $\{(a,b,c)\in\R_{+}^{3}\ |\ a\leq b\leq c\}$. The integrability of $\Psi(a+b,a,c)$ over the region $\{(a,b,c)\in\R_{+}^{3}\ |\ c\leq b\leq a\}$ in the case $H> \frac{1}{2}$, follows from a similar argument. From the previous analysis it follows that $\Psi(a+b,a,c)$ is integrable in $\R_{+}^3$, and hence $\Psi(x,u_{1},u_{2})$ is integrable in $\Sc_{3}$, as required. The proof is now complete.\\~\\
\end{proof}
Following similar arguments to those presented in the proof of Lemma \ref{lem:inegraltightness}, we can prove the following result
\begin{Lema}\label{lem:inegraltightness2}
Let the functions $\mu$ and $\Theta_{1}$ be defined by \eqref{eq:mudef} and \eqref{eq:Ldef} respectively. Then, for every $\frac{3}{4}<H<1$ and $0<p<\frac{4Hd}{3}$, 
\begin{align}\label{ineq:intsup}
\sup_{\varepsilon\in(0,1)}\int_{\R_{+}^2}\int_{0}^{T}\varepsilon^{-\frac{2}{H}}\frac{\mu(x,\varepsilon^{\frac{1}{2H}}u_{1},\varepsilon^{\frac{1}{2H}}u_{2})^2}{u_{1}^{2H}u_{2}^{2H}}\Theta_{1}(\varepsilon^{-\frac{1}{2H}}x,u_{1},u_{2})^{-\frac{d}{p}}dx du_{1}du_{2}<\infty.
\end{align}
\end{Lema}
\begin{proof}
Denote by $\kappa_{\varepsilon}(x,u_{1},u_{2})$ the function 
\begin{align*}
\kappa_{\varepsilon}(x,u_{1},u_{2})
  &:=\varepsilon^{-\frac{2}{H}}\mu(x,\varepsilon^{\frac{1}{2H}}u_{1},\varepsilon^{\frac{1}{2H}}u_{2})^2(u_{1}u_{2})^{-2H}\Theta_{1}(\varepsilon^{-\frac{1}{2H}}x,u_{1},u_{2})^{-\frac{d}{p}}. 
\end{align*}
To prove \eqref{ineq:intsup}, it suffices to show that 
\begin{align}\label{ineq:intsup2}
\sup_{\varepsilon\in(0,1)}\int_{\R_{+}^2}\int_{0}^{T}\Indi{\Sc_{i}}(x,\varepsilon^{\frac{1}{2H}}u_{1},\varepsilon^{\frac{1}{2H}}u_{2})\kappa_{\varepsilon}(x,u_{1},u_{2})dxdu_{1}du_{2}<\infty,
\end{align}
for $i=1,2,3$. To prove \eqref{ineq:intsup2} in the case $i=1,2$, we make the change of variable $\widehat{x}:=\varepsilon^{-\frac{1}{2H}}x$, in order to get
\begin{multline*}
\int_{\R_{+}^2}\int_{0}^{T}\Indi{\Sc_{i}}(x,\varepsilon^{\frac{1}{2H}}u_{1},\varepsilon^{\frac{1}{2H}}u_{2})\kappa_{\varepsilon}(x,u_{1},u_{2})dxdu_{1}du_{2}\\
\begin{aligned}
  &=\varepsilon^{-\frac{3}{2H}+2}\int_{\R_{+}^2}\int_{0}^{\varepsilon^{-\frac{1}{2H}}T}\Indi{\Sc_{i}}(\widehat{x},u_{1},u_{2})\Psi(\widehat{x},u_{1},u_{2})d\widehat{x}du_{1}du_{2},
\end{aligned}
\end{multline*}
where $\Psi$ is defined by \eqref{eq:Psidef}. Hence,  
\begin{align}\label{ineq:intsup3}
\int_{\R_{+}^2}\int_{0}^{T}\Indi{\Sc_{i}}(x,\varepsilon^{\frac{1}{2H}}u_{1},\varepsilon^{\frac{1}{2H}}u_{2})\kappa_{\varepsilon}(x,u_{1},u_{2})dxdu_{1}du_{2}
  &\leq \int_{\Sc_{i}}\Psi(x,u_{1},u_{2})dxdu_{1}du_{2}.
\end{align}
In Lemma \ref{lem:inegraltightness}, we proved that  $\int_{\Sc_{1}}\Psi(x,u_{1},u_{2})dxdu_{1}du_{2}<\infty$, provided that $p<\frac{4Hd}{3}$. To handle the case $i=2$, we change the coordinates $(x,u_{1},u_{2})$ by $(a:=x,b:=u_{2},c:=u_{1}-x-u_{2})$, in order to get 
\begin{align*}
\int_{\Sc_{2}}\Psi(x,u_{1},u_{2})dxdu_{1}du_{2}
  &=\int_{\R_{+}^{3}}\Psi(a,a+b+c,b)dadbdc.
\end{align*}
By \eqref{eq:R2sub1}, $\Psi(a,a+b+c,b)$ is integrable in $\{(a,b,c)\in\R_{+}^{3}\ |\ b\geq a\wedge c\}$. In addition, since $2H-\frac{1}{2}\leq \frac{3}{2}<Hd$, by \eqref{eq:ineqlong2pp}, $\Psi(a,a+b+c,b)$ is integrable in $\{(a,b,c)\in\R_{+}^{3}\ |\ b\leq a\wedge c\}$, and hence, $\Psi(x,u_{1},u_{2})$ is integrable in $\Sc_{2}$, as required. It then remains to prove \eqref{ineq:intsup2} in the case $i=3$. Using \eqref{eq:IpHbig}, we can easily check that for every $(x,v_{1},v_{2})\in\Sc_{3}$, 
\begin{align*}
\Abs{\mu(x,v_{1},v_{2})}
  &=\Abs{\Ip{\Indi{[0,v_{1}]},\Indi{[x,x+v_{2}]}}_{\Hg^{d}}}\\
	&=H(2H-1)v_{1}v_{2}\int_{[0,1]^{2}}\Abs{x+v_{2}w_{2}-v_{1}w_{1}}^{2H-2}dw_{1}dw_{2}\\
	&\leq H(2H-1)v_{1}v_{2}\int_{[0,1]^{2}}\Abs{x-xw_{1}}^{2H-2}dw_{1}dw_{2},
\end{align*}
and hence, there exists a constant $C>0$ only depending on $H$, such that for every $(x,v_{1},v_{2})\in\Sc_{3}$, 
\begin{align}\label{ineq:muR3x}
\Abs{\mu(x,v_{1},v_{2})}
  &\leq Cv_{1}v_{2}x^{2H-2}.
\end{align}
On the other hand, for every $(x,\varepsilon^{\frac{1}{2H}}u_{1},\varepsilon^{\frac{1}{2H}}u_{2})\in\Sc_{3}$, it holds $(\varepsilon^{-\frac{1}{2H}}x,u_{1},u_{2})\in\Sc_{3}$, and hence, by \eqref{eq:inequR31}, 
\begin{align}\label{ineq:muR3theta}
\Theta_{1}(\varepsilon^{-\frac{1}{2H}}x,u_{1},u_{2})
  &\geq\delta u_{1}^{2H}u_{2}^{2H}
\end{align}
By \eqref{ineq:muR3x} and \eqref{ineq:muR3theta}, we obtain
\begin{align}\label{ineq:kappaR3}
\kappa_{\varepsilon}(x,u_{1},u_{2})
  &\leq C(u_{1}u_{2})^{2-2H}x^{4H-4}(1+u_{1}^{2H}+u_{2}^{2H}+u_{1}^{2H}u_{2}^{2H})^{-\frac{d}{p}},
\end{align}
for some constant $C>0$, and hence, 
\begin{multline*}
\int_{\R_{+}^2}\int_{0}^{T}\Indi{\Sc_{3}}(x,\varepsilon^{\frac{1}{2H}}u_{1},\varepsilon^{\frac{1}{2H}}u_{2})\kappa_{\varepsilon}(x,u_{1},u_{2})dxdu_{1}du_{2}\\
\leq\int_{\R_{+}^2}\int_{0}^{T}(u_{1}u_{2})^{2-2H}x^{4H-4}(1+u_{1}^{2H}+u_{2}^{2H}+u_{1}^{2H}u_{2}^{2H})^{-\frac{d}{p}}dxdu_{1}du_{2}.
\end{multline*}
Since $H>\frac{3}{4}$, then $3-2H<\frac{3}{2}<Hd$, and hence, the integral in the right-hand side of the previous identity is finite, which implies that \eqref{ineq:intsup2} holds for $i=3$, as required. The proof is now complete.
\end{proof}
\begin{Lema}\label{lem:finiteFGintegrallog}
Let $d\geq 3$, and $T>0$ be fixed. Let the functions $\mu$ and $\Theta_{\varepsilon}$ be defined by \eqref{eq:mudef} and \eqref{eq:Ldef} respectively and and assume that $H=\frac{3}{4}$. Then, for every $0<p<d$,  
\begin{align*}
\sup_{\varepsilon\in(0,1/e)}\frac{\varepsilon^{-8/3}}{\log(1/\varepsilon)}\int_{\R_{+}^2}\int_{0}^{T}\frac{\mu(x,\varepsilon^{\frac{2}{3}}u_{1},\varepsilon^{\frac{2}{3}}u_{2})^2}{(u_{1}u_{2})^{\frac{3}{2}}}\Theta_{1}(\varepsilon^{-\frac{2}{3}}x,u_{1},u_{2})^{-\frac{d}{p}}\text{d}x\text{d}u_{1}\text{d}u_{2}
<\infty.
\end{align*}  
\end{Lema}
\begin{proof}
Denote by $\kappa_{\varepsilon}(x,u_{1},u_{2})$ the function 
\begin{align*}
\kappa_{\varepsilon}(x,u_{1},u_{2})
  &:=\frac{\varepsilon^{-8/3}}{\log(1/\varepsilon)}\mu(x,\varepsilon^{\frac{2}{3}}u_{1},\varepsilon^{\frac{2}{3}}u_{2})^2(u_{1}u_{2})^{-\frac{3}{2}}\Theta_{1}(\varepsilon^{-\frac{2}{3}}x,u_{1},u_{2})^{-\frac{d}{p}}. 
\end{align*}

As in Lemma \ref{lem:inegraltightness2}, it suffices to show that 
\begin{align}\label{ineq:intsup2log}
\sup_{\varepsilon\in(0,1)}\int_{\R_{+}^{2}}\int_{0}^{T}\Indi{\Sc_{i}}(x,\varepsilon^{\frac{2}{3}}u_{1},\varepsilon^{\frac{2}{3}}u_{2})\kappa_{\varepsilon}(x,u_{1},u_{2})dxdu_{1}du_{2}<\infty,
\end{align}
for $i=1,2,3$, where the regions $\Sc_{i}$ are defined by \eqref{eq:Scdef}.  The cases $i=1,2$ are handled similarly to Lemma \ref{lem:inegraltightness2}, so it suffices to prove \eqref{ineq:intsup2} in the case $i=3$. 

Suppose $(x,\varepsilon^{\frac{2}{3}}u_{1},\varepsilon^{\frac{2}{3}}u_{2})\in\Sc_{3}$. Then, by Lemma \ref{ineq:techR}, there exists a constant $C>0$, such that 
\begin{align*}
\Abs{\mu(x,\varepsilon^{\frac{2}{3}}u_{1},\varepsilon^{\frac{2}{3}}u_{2})}
  &\leq C\varepsilon^{4/3}(x+\varepsilon^{\frac{2}{3}}u_{1}+\varepsilon^{\frac{2}{3}}u_{2})^{-\frac{1}{2}}u_{1}u_{2}\\
	&= C\varepsilon(\varepsilon^{-\frac{2}{3}}x+u_{1}+u_{2})^{-\frac{1}{2}}u_{1}u_{2}
\end{align*}
In addition, by Lemma \ref{lem:local_non_determinism} we have that 
$u_{1}^{\frac{3}{2}}u_{2}^{\frac{3}{2}}-\mu(\varepsilon^{-\frac{2}{3}}x,u_{1},u_{2})^2\geq \delta (u_{1}u_{2})^{\frac{3}{2}}$, for some $\delta>0$. Therefore, we conclude that there exists a constant $C>0$, such that 
\begin{align*}
\kappa_{\varepsilon}(x,u_{1},u_{2})
  &\leq \frac{\varepsilon^{-\frac{2}{3}}C^2}{\log(1/\varepsilon)}(\varepsilon^{-\frac{2}{3}}x+u_{1}+u_{2})^{-1}\sqrt{u_{1}u_{2}}\left(1+u_{1}^{\frac{3}{2}}+u_{2}^{\frac{3}{2}}+u_{1}^{\frac{3}{2}}u_{2}^{\frac{3}{2}}-\mu(x,u_{1},u_{2})^2\right)^{-\frac{d}{p}}\\
	&\leq \frac{\varepsilon^{-\frac{2}{3}}C^2\delta^{-\frac{d}{p}}}{\log(1/\varepsilon)}(\varepsilon^{-\frac{2}{3}}x+u_{1}+u_{2})^{-1}\sqrt{u_{1}u_{2}}\left(1+u_{1}^{\frac{3}{2}}+u_{2}^{\frac{3}{2}}+u_{1}^{\frac{3}{2}}u_{2}^{\frac{3}{2}}\right)^{-\frac{d}{p}}.
\end{align*}
Consequently, there exists a constant $C>0$, such that 
\begin{multline*}
\int_{\R_{+}^{2}}\int_{0}^{T}\Indi{\Sc_{i}}(x,u_{1},u_{2})\kappa_{\varepsilon}(x,u_{1},u_{2})dxdu_{1}du_{2}\\
  \leq \frac{C\varepsilon^{-\frac{2}{3}}}{\log(1/\varepsilon)}\int_{0}^{T}\int_{\R_{+}^{2}}(\varepsilon^{-\frac{2}{3}}x+u_{1}+u_{2})^{-1}\sqrt{u_{1}u_{2}}\left(1+u_{1}^{\frac{3}{2}}+u_{2}^{\frac{3}{2}}+u_{1}^{\frac{3}{2}}u_{2}^{\frac{3}{2}}\right)^{-\frac{d}{p}}. 
\end{multline*}
Hence, making the change of variable $\widetilde{x}:=\varepsilon^{-\frac{2}{3}}x$, we obtain
\begin{multline}\label{eq:Lineq}
\int_{\R_{+}^{2}}\int_{0}^{T}\Indi{\Sc_{i}}(x,u_{1},u_{2})\kappa_{\varepsilon}(x,u_{1},u_{2})dxdu_{1}du_{2}\\
\begin{aligned}
  &\leq \frac{C}{\log(1/\varepsilon)}\int_{\R_{+}^{2}}\int_{0}^{\varepsilon^{-\frac{2}{3}}T}(x+u_{1}+u_{2})^{-1}\sqrt{u_{1}u_{2}}\left(1+u_{1}^{\frac{3}{2}}+u_{2}^{\frac{3}{2}}+u_{1}^{\frac{3}{2}}u_{2}^{\frac{3}{2}}\right)^{-\frac{d}{p}}dxdu_{1}du_{2}\\
	&= \frac{C}{\log(1/\varepsilon)}\int_{\R_{+}^{2}}\int_{0}^{1}(x+u_{1}+u_{2})^{-1}\sqrt{u_{1}u_{2}}\left(1+u_{1}^{\frac{3}{2}}+u_{2}^{\frac{3}{2}}+u_{1}^{\frac{3}{2}}u_{2}^{\frac{3}{2}}\right)^{-\frac{d}{p}}dxdu_{1}du_{2}\\
	&+ \frac{C}{\log(1/\varepsilon)}\int_{\R_{+}^{2}}\int_{1}^{\varepsilon^{-\frac{2}{3}}T}(x+u_{1}+u_{2})^{-1}\sqrt{u_{1}u_{2}}\left(1+u_{1}^{\frac{3}{2}}+u_{2}^{\frac{3}{2}}+u_{1}^{\frac{3}{2}}u_{2}^{\frac{3}{2}}\right)^{-\frac{d}{p}}dxdu_{1}du_{2}.
\end{aligned}
\end{multline}
Applying the inequalities $(x+u_{1}+u_{2})^{-1}\leq (u_{1}+u_{2})^{-1}\leq \frac{1}{2}(u_{1}u_{2})^{-\frac{1}{2}}$ for $x\in[0,1]$, and $(x+u_{1}+u_{2})^{-1}\leq x^{-1}$ for $x\geq 1$, in the first and second terms in the right-hand side of \eqref{eq:Lineq}, and then integrating the variable $x$, we can show that
\begin{align*}
\int_{\R_{+}^{2}}\int_{0}^{T}\Indi{\Sc_{i}}(x,u_{1},u_{2})\kappa_{\varepsilon}(x,u_{1},u_{2})dxdu_{1}du_{2}
  &\leq C\int_{\R_{+}^{2}}\left(\frac{(u_{1}u_{2})^{-\frac{1}{2}}+\frac{2}{3}\log(1/\varepsilon)+ \log(T)}{\log(1/\varepsilon)}\right)\\
	&\times\sqrt{u_{1}u_{2}}\left(1+u_{1}^{\frac{3}{2}}+u_{2}^{\frac{3}{2}}+u_{1}^{\frac{3}{2}}u_{2}^{\frac{3}{2}}\right)^{-\frac{d}{p}}dxdu_{1}du_{2}, 
\end{align*}
and consequently, for every $\varepsilon<1/e$,
\begin{multline*}
\int_{\R_{+}^{2}}\int_{0}^{T}\Indi{\Sc_{i}}(x,u_{1},u_{2})\kappa_{\varepsilon}(x,u_{1},u_{2})dxdu_{1}du_{2}\\
\begin{aligned}
  &\leq C\int_{\R_{+}^{2}}\left((u_{1}u_{2})^{-\frac{1}{2}}+\log(T)\right)\sqrt{u_{1}u_{2}}\left(1+u_{1}^{\frac{3}{2}}+u_{2}^{\frac{3}{2}}+u_{1}^{\frac{3}{2}}u_{2}^{\frac{3}{2}}\right)^{-\frac{d}{p}}dxdu_{1}du_{2}. 
\end{aligned}
\end{multline*}
The right-hand side of the previous inequality is finite due to the condition $0<p<d$. This finishes the proof of \eqref{ineq:intsup2log}.
\end{proof}


\end{document}